\documentclass[reqno]{amsart}

\usepackage{a4wide}
\usepackage{color}
\usepackage{mathrsfs}
\usepackage{mathtools}
\usepackage{tikz-cd}
\usepackage{amsmath}
\usepackage{amssymb}
\usepackage{bbm}
\numberwithin{equation}{section}
\usepackage[colorlinks,citecolor=green,linkcolor=red]{hyperref}
\usepackage{hyphenat}
\usepackage[latin1]{inputenc}

\newcommand{\B}{\mathbb{B}}
\newcommand{\N}{\mathbb{N}}
\newcommand{\R}{\mathbb{R}}
\newcommand{\sfd}{{\sf d}}
\renewcommand{\d}{{\mathrm d}}
\newcommand{\restr}[1]{\lower3pt\hbox{\(|_{#1}\)}}

\newcommand{\nchi}{{\raise.3ex\hbox{\(\chi\)}}}
\newcommand{\1}{\mathbbm 1}
\newcommand{\fr}{\penalty-20\null\hfill\(\blacksquare\)}
\newcommand{\X}{{\rm X}}
\newcommand{\Y}{{\rm Y}}
\newcommand{\XX}{\mathbb{X}}
\newcommand{\YY}{\mathbb{Y}}
\newcommand{\mm}{\mathfrak m}

\newtheorem{theorem}{Theorem}[section]
\newtheorem{corollary}[theorem]{Corollary}
\newtheorem{lemma}[theorem]{Lemma}
\newtheorem{proposition}[theorem]{Proposition}
\newtheorem{definition}[theorem]{Definition}

\newtheorem{remark}[theorem]{Remark}

\title[Projective and injective tensor products of Banach $L^0$-modules]
{Projective and injective tensor \\ products of Banach $L^0$-modules}
\author{Enrico Pasqualetto}
\address{Department of Mathematics and Statistics,
P.O.\ Box 35 (MaD), FI-40014 University of Jyvaskyla}
\email{enrico.e.pasqualetto@jyu.fi}

\begin{document}
\date{\today} 
\keywords{Banach module, tensor product, Schauder basis}
\subjclass[2020]{46M05, 47A80, 53C23, 16D90, 18F15}
\begin{abstract}
We study projective and injective tensor products of Banach $L^0$-modules over a $\sigma$-finite measure space.
En route, we extend to Banach $L^0$-modules several technical tools of independent interest, such as quotient
operators, summable families, and Schauder bases.
\end{abstract}
\maketitle
\tableofcontents
\section{Introduction}
As of now, the language of normed modules introduced by Gigli in \cite{Gigli14} has become an indispensable tool
in analysis on metric measure spaces, especially on those verifying synthetic lower Ricci curvature bounds (the so-called
\(\sf RCD\) spaces). Normed modules allow to define several spaces of measurable tensor fields, whose investigation
has remarkable analytic and geometric consequences. In this respect, three constructions are particularly important:
duals, pullbacks, and (in the case of Hilbert modules) tensor products. For example, the dual of the pullback is
important for constructing the differential of a map of bounded deformation or the velocity of a test plan (cf.\ with
the introduction of \cite{GLP22}), while the tensor product of Hilbert modules is a fundamental tool when studying
the second order differential calculus on \(\sf RCD\) spaces (see \cite[Section 3]{Gigli14}). However, since many
spaces of interest are `non-Riemannian', it would be interesting to study tensor products of non-Hilbert normed modules,
as well as to understand their relation with duals and pullbacks: this is the main goal of this paper.
We assume the reader is familiar with (projective and injective) tensor products of Banach spaces,
for which we refer e.g.\ to the authoritative monograph \cite{Ryan02}.
\medskip

Let us briefly describe the content of the paper. Given a \(\sigma\)-finite measure space \(\XX=(\X,\Sigma,\mm)\),
we consider the class of \emph{Banach \(L^0(\XX)\)-modules}, i.e.\ modules over the commutative ring \(L^0(\XX)\)
that are endowed with a complete pointwise norm operator (cf.\ with Definition \ref{def:Ban_mod}). Even though we are mostly
interested to their applications in metric measure geometry, we consider Banach \(L^0\)-modules over general measure spaces.
Our choice is due to the fact that Banach \(L^0\)-modules play an important role also in other research areas,
see for example \cite{HLR91}, as well as \cite{Guo-2011} and the references therein. The only results where
we need to require an additional assumption on the base measure space (verified in the case of metric
measure spaces) are Theorems \ref{thm:pullback_and_proj} and \ref{thm:pullback_and_inj}.
Given two Banach \(L^0(\XX)\)-modules \(\mathscr M\) and \(\mathscr N\), we first provide a useful criterion to detect
the null tensors of the algebraic tensor product \(\mathscr M\otimes\mathscr N\); see Lemma \ref{lem:crit_null_tensor}.
Its proof is quite subtle, one reason being the fact that the algebraic dual of a module might not separate the points
(differently from duals of vector spaces); cf.\ with Remark \ref{rmk:alg_tens_diff}. Having Lemma \ref{lem:crit_null_tensor}
at our disposal, we can:
\begin{itemize}
\item Define and study the \emph{projective tensor product} \(\mathscr M\hat\otimes_\pi\mathscr N\), see Section \ref{s:proj_tens}.
\item Define and study the \emph{injective tensor product} \(\mathscr M\hat\otimes_\varepsilon\mathscr N\), see Section \ref{s:inj_tens}.
\end{itemize}
Motivated by the analysis on metric spaces, our attention is focussed on the following results:
\begin{itemize}
\item The \emph{dual} of \(\mathscr M\hat\otimes_\pi\mathscr N\) can be identified with the space
\({\rm B}(\mathscr M,\mathscr N)\) of bounded \(L^0(\XX)\)-bilinear maps from \(\mathscr M\times\mathscr N\) to \(L^0(\XX)\)
(see Theorem \ref{thm:dual_proj_tensor}), while the dual of \(\mathscr M\hat\otimes_\varepsilon\mathscr N\) is a quotient
of the dual of the space \({\rm C}_{\rm pb}(\mathbb D_{\mathscr M^*}^{w^*}\times\mathbb D_{\mathscr N^*}^{w^*};L^0(\XX))\)
(see Definition \ref{def:ptwse_cont} and Theorem \ref{thm:dual_inj_prod}).
\item The operation of taking \emph{pullbacks} of Banach \(L^0(\XX)\)-modules commutes both with projective tensor
products (Theorem \ref{thm:pullback_and_proj}) and with injective tensor products (Theorem \ref{thm:pullback_and_inj}).
\end{itemize}
While some of the concepts and results we presented above are natural extensions of their version for Banach spaces,
other ones are non-trivial generalisations (see e.g.\ the two different notions of a continuous module-valued map
in Section \ref{s:cont_mod-valued}) or have no counterpart in the Banach space setting (as in the case of
pullback modules). We conclude the introduction by mentioning that a significant portion of the paper is devoted
to the development of several technical tools (new in the setting of Banach \(L^0(\XX)\)-modules), which are needed
in Sections \ref{s:proj_tens} and \ref{s:inj_tens}, and can be useful in the future research concerning normed modules:
we study \emph{quotient operators} (Section \ref{s:quotient_oper}), \emph{summable families} in Banach \(L^0(\XX)\)-modules
(Section \ref{s:summability}), and \emph{local Schauder bases} (Section \ref{s:Schauder}).
\section{Preliminaries}
Given an arbitrary set \(I\neq\varnothing\), we denote by \(\mathscr P(I)\) its power set (i.e.\ the set of its subsets) and
\[
\mathscr P_f(I)\coloneqq\big\{F\in\mathscr P(I)\;\big|\;F\text{ is finite}\big\}.
\]
Given any couple of indexes \(i,j\in I\), we define \(\delta_{ij}\in\{0,1\}\) as \(\delta_{ij}\coloneqq 1\)
if \(i=j\) and \(\delta_{ij}\coloneqq 0\) if \(i\neq j\). Moreover, if \(\X\) is a set, then the
characteristic function \(\1_E\colon\X\to\{0,1\}\) of a subset \(E\subseteq\X\) is
\[
\1_E(x)\coloneqq\left\{\begin{array}{ll}
1\\
0
\end{array}\quad\begin{array}{ll}
\text{ for every }x\in E,\\
\text{ for every }x\in\X\setminus E.
\end{array}\right.
\]
For any map \(\varphi\colon\X\to\Y\) between two sets \(\X\) and \(\Y\), we denote by \(\varphi[\X]\subseteq\Y\) the image of \(\varphi\).
\subsection{Tensor products of modules}
In this section, we recall the basics of the theory of tensor products of modules, which is originally due
to \cite{Bourbaki48}. See also \cite{Conrad18} and the references indicated therein. Our standing convention
is that all rings are assumed to have a multiplicative identity.
\begin{theorem}[Tensor products of modules]
Let \(R\) be a commutative ring. Let \(M\) and \(N\) be modules over \(R\). Then there exists a unique couple \((M\otimes N,\otimes)\),
where \(M\otimes N\) is an \(R\)-module and \(\otimes\colon M\times N\to M\otimes N\) is an \(R\)-bilinear map, such that the following
universal property holds: given any \(R\)-module \(Q\) and any \(R\)-bilinear map \(b\colon M\times N\to Q\), there exists a unique \(R\)-linear
map \(\tilde b\colon M\otimes N\to Q\), called the \textbf{\(R\)-linearisation} of \(b\), for which the diagram
\[\begin{tikzcd}
M\times N \arrow[r,"b"] \arrow[d,swap,"\otimes"] & Q \\
M\otimes N \arrow[ur,swap,"\tilde b"] &
\end{tikzcd}\]
commutes. The couple \((M\otimes N,\otimes)\) is unique up to a unique isomorphism: given any \((T,\tilde\otimes)\)
with the same properties, there exists a unique isomorphism of \(R\)-modules \(\Phi\colon M\otimes N\to T\) such that
\[\begin{tikzcd}
M\times N \arrow[r,"\otimes"] \arrow[dr,swap,"\tilde\otimes"] & M\otimes N \arrow[d,"\Phi"] \\
& T
\end{tikzcd}\]
commutes. We say that \((M\otimes N,\otimes)\), or just \(M\otimes N\), is the \textbf{tensor product} of \(M\) and \(N\).
\end{theorem}

Those elements of \(M\otimes N\) of the form \(v\otimes w\) are called \textbf{elementary tensors}.
Any \(\alpha\in M\otimes N\) is a sum of elementary tensors: \(\alpha=\sum_{i=1}^n v_i\otimes w_i\)
for some \(v_1,\ldots,v_n\in M\) and \(w_1,\ldots,w_n\in N\).
\medskip

We recall the following criterion to detect when an element \(\sum_{i=1}^n v_i\otimes w_i\in M\otimes N\) is null:
\begin{equation}\label{eq:gen_criter_null_tensor}
\sum_{i=1}^n v_i\otimes w_i=0\quad\Longleftrightarrow\quad\sum_{i=1}^n b(v_i,w_i)=0\;\text{ for every }R\text{-bilinear map }b\colon M\times N\to Q.
\end{equation}
Differently from the case of tensor products of vector spaces, in \eqref{eq:gen_criter_null_tensor}
one has to consider \(R\)-bilinear maps \(b\) with values into an arbitrary \(R\)-module \(Q\) (taking \(Q=R\) is
not sufficient). Indeed, it can happen that no non-null bilinear map \(b\colon M\times N\to R\) exists even if \(M\), \(N\) are non-trivial; see \cite{Pabst00}.
\begin{lemma}[Tensor products of \(R\)-linear maps]\label{lem:alg_tensor_hom}
Let \(R\) be a commutative ring. Let \(T\colon M\to\tilde M\) and \(S\colon N\to\tilde N\) be \(R\)-linear maps between \(R\)-modules.
Then there exists a unique \(R\)-linear map \(T\otimes S\colon M\otimes N\to\tilde M\otimes\tilde N\)
such that \((T\otimes S)(v\otimes w)=T(v)\otimes S(w)\) for every \(v\in M\) and \(w\in N\).
\end{lemma}

Each commutative ring \(R\) is an \(R\)-module. Each \(R\)-module \(M\) is canonically isomorphic (as an \(R\)-module) to \(R\otimes M\)
via the \(R\)-linear map \(M\ni v\mapsto 1_R\otimes v\in R\otimes M\). In particular, \(R\otimes R\cong R\).
\subsection{The space \texorpdfstring{\(L^0(\XX)\)}{L0(X)}}
Let \(\XX=(\X,\Sigma,\mm)\) be a \(\sigma\)-finite measure space. We denote by \(L^0(\XX)\) the space
of all real-valued measurable functions from \(\X\) to \(\R\), quotiented up to \(\mm\)-a.e.\ identity.
The equivalence class in \(L^0(\XX)\) of a given measurable function \(\bar f\colon\X\to\R\) will be denoted by \([\bar f]_\mm\).
The space \(L^0(\XX)\) is a vector space and a commutative ring if endowed with the natural pointwise operations.
Moreover, fixed a probability measure \(\tilde\mm\) on \((\X,\Sigma)\) with \(\mm\ll\tilde\mm\ll\mm\), we have that
\[
\sfd_{L^0(\XX)}(f,g)\coloneqq\int|f-g|\wedge 1\,\d\tilde\mm\quad\text{ for every }f,g\in L^0(\XX)
\]
is a complete distance, and \(L^0(\XX)\) becomes a topological vector space and a topological ring if endowed
with \(\sfd_{L^0(\XX)}\). The distance \(\sfd_{L^0(\XX)}\) depends on the chosen auxiliary measure \(\tilde\mm\),
but its induced topology does not. We also have that a given sequence \((f_n)_{n\in\N}\subseteq L^0(\XX)\) converges
to a limit function \(f\in L^0(\XX)\) with respect to \(\sfd_{L^0(\XX)}\) if and only if there exists a subsequence
\((n_i)_{i\in\N}\subseteq\N\) such that \(f(x)=\lim_i f_{n_i}(x)\) for \(\mm\)-a.e.\ \(x\in\X\). Finally, \(L^0(\XX)\)
is a Riesz space if endowed with the natural partial order defined in the following way: given \(f,g\in L^0(\XX)\),
we declare that \(f\leq g\) if and only if \(f(x)\leq g(x)\) for \(\mm\)-a.e.\ \(x\in\X\). The \textbf{positive cone}
of \(L^0(\XX)\) is then denoted by
\[
L^0(\XX)^+\coloneqq\big\{f\in L^0(\XX)\;\big|\;f\geq 0\big\}.
\]
We also point out that \(L^0(\XX)\) is \textbf{Dedekind complete}, i.e.\ every subset \(\{f_i\}_{i\in I}\) of
\(L^0(\XX)\) that is \textbf{order-bounded} (which means that there exists \(g\in L^0(\XX)^+\) such that \(|f_i|\leq g\) for
every \(i\in I\)) has both a supremum \(\bigvee_{i\in I}f_i\in L^0(\XX)\) and an infimum \(\bigwedge_{i\in I}f_i\in L^0(\XX)\).
Furthermore, \(L^0(\XX)\) has both the \textbf{countable sup property} and the \textbf{countable inf property}, i.e.\ for
any order-bounded set \(\{f_i\}_{i\in I}\subseteq L^0(\XX)\) one can find \(C\subseteq I\) countable with
\(\bigvee_{i\in C}f_i=\bigvee_{i\in I}f_i\) and \(\bigwedge_{i\in C}f_i=\bigwedge_{i\in I}f_i\).
More generally, the space \(L^0_{\rm ext}(\XX)\) of measurable functions from \(\X\) to \([-\infty,+\infty]\), quotiented up
to \(\mm\)-a.e.\ identity, is a Dedekind complete Riesz space with the countable sup/inf properties. Notice that every set in
\(L^0_{\rm ext}(\XX)\) is order-bounded and that \(L^0(\XX)\) is a solid Riesz subspace of \(L^0_{\rm ext}(\XX)\).
\medskip

Given any measurable set \(E\in\Sigma\) with \(\mm(E)>0\), we will use the following shorthand notation:
\[
\XX|_E\coloneqq(\X,\Sigma,\mm|_E),
\]
where \(\mm|_E\) stands for the restriction of \(\mm\) to \(E\), i.e.\ we set \(\mm|_E(F)\coloneqq\mm(E\cap F)\) for every \(F\in\Sigma\).
\begin{remark}\label{rmk:aux_estim_sup}{\rm
Let \(\XX=(\X,\Sigma,\mm)\) be a \(\sigma\)-finite measure space and \(\{f_i\}_{i\in I}\subseteq L^0_{\rm ext}(\XX)\).
Fix a representative \(\bar f_i\) of \(f_i\) for any \(i\in I\). Suppose there exists a measurable function \(\bar g\colon\X\to[-\infty,+\infty]\)
such that \(\sup_{i\in I}\bar f_i(x)\leq\bar g(x)\) for all \(x\in\X\); we do not require that \(x\mapsto\sup_{i\in I}\bar f_i(x)\) is measurable.
Then \(\bigvee_{i\in I}f_i\leq[\bar g]_\mm\). Indeed, we can find a countable set \(C\subseteq I\) such that
\(\bigvee_{i\in C}f_i=\bigvee_{i\in I}f_i\). As \(\sup_{i\in C}\bar f_i(x)\leq\bar g(x)\) for every \(x\in\X\) and
\(\bigvee_{i\in C}f_i=\big[\sup_{i\in C}\bar f_i\big]_\mm\), we get \(\bigvee_{i\in I}f_i\leq[\bar g]_\mm\).
\fr}\end{remark}

We also point out that the metric space \((L^0(\XX),\sfd_{L^0(\XX)})\) is separable if and only if the measure space
\((\X,\Sigma,\mm)\) is \textbf{separable}, which means that we can find a sequence \((E_n)_{n\in\N}\subseteq\Sigma\) such that
\[
\inf_{n\in\N}\mm(E_n\Delta E)=0\quad\text{ for every }E\in\Sigma\text{ such that }\mm(E)<+\infty.
\]
We refer e.g.\ to \cite[Section 1.1.2]{GP20} for a more detailed discussion on \(L^0(\XX)\) spaces.
See also \cite{AliprantisBorder99,Bogachev07}.
\subsection{Banach spaces}
We briefly recall some definitions and results concerning Banach spaces.
\medskip

Given an index set \(I\neq\varnothing\) and an exponent \(p\in[1,\infty]\), we denote by \(\ell_p(I)\) the vector space
\[
\ell_p(I)\coloneqq\big\{a=(a_i)_{i\in I}\in\R^I\;\big|\;\|a\|_{\ell_p(I)}<+\infty\big\},
\]
where for any \(a=(a_i)_{i\in I}\in\R^I\) we define the quantity \(\|a\|_{\ell_p(I)}\in[0,+\infty]\) as
\[
\|a\|_{\ell_p(I)}\coloneqq\left\{\begin{array}{ll}
\big(\sum_{i\in I}|a_i|^p\big)^{1/p}\\
\sup_{i\in I}|a_i|
\end{array}\quad\begin{array}{ll}
\text{ if }p<\infty,\\
\text{ if }p=\infty,
\end{array}\right.
\]
where \(\sum_{i\in I}|a_i|^p\coloneqq\sup_{F\in\mathscr P_f(I)}\sum_{i\in F}|a_i|^p\).
It holds that \((\ell_p(I),\|\cdot\|_{\ell_p(I)})\) is a Banach space.
\medskip

An \textbf{(unconditional) Schauder basis} of a Banach space \(\B\) is a family of vectors \(\{v_i\}_{i\in I}\subseteq\B\)
such that for any \(v\in\B\) there is a unique \((\lambda_i)_{i\in I}\subseteq\R^I\) for which
\(\{\lambda_i v_i\}_{i\in I}\subseteq\B\) is summable and
\[
v=\sum_{i\in I}\lambda_i v_i.
\]
We recall that this means that for any \(\varepsilon>0\) there exists \(F_\varepsilon\in\mathscr P_f(I)\) such that
\[
\bigg\|v-\sum_{i\in F}\lambda_i v_i\bigg\|_\B<\varepsilon\quad\text{ for every }F\in\mathscr P_f(I)\text{ with }F_\varepsilon\subseteq F.
\]
We point out that, letting \(\overline{\rm span}(S)\) denote the closure of the linear span of a set \(S\subseteq\B\), we have
\begin{equation}\label{eq:direct_decomp_Schauder}
v_i\notin\overline{\rm span}\big\{v_j\;\big|\;j\in I\setminus\{i\}\big\}\quad\text{ for every }i\in I.
\end{equation}
We also recall that the canonical elements \(({\sf e}_i)_{i\in I}\subseteq\ell_1(I)\), which are given by
\begin{equation}\label{eq:def_canonical_elements}
{\sf e}_i\coloneqq(\delta_{ij})_{j\in I}\in\ell_1(I),
\end{equation}
form an (unconditional) Schauder basis of \(\ell_1(I)\). See e.g.\ \cite{FHHMZ10} for an account of Schauder bases.
\medskip

A separable Banach space \(\B\) is said to be a \textbf{universal separable Banach space} if every separable Banach space
can be embedded linearly and isometrically in \(\B\). The Banach--Mazur theorem states that universal separable Banach spaces
exist; for instance, the space \({\rm C}([0,1])\) endowed with the supremum norm has this property. See e.g.\ \cite{BP75} for a proof
of this result.
\subsection{Banach \texorpdfstring{\(L^0\)}{L0}-modules}
The notion of normed/Banach \(L^0(\XX)\)-module we are going to recall was introduced in \cite{Gigli14},
but the axiomatisation we will present is taken from \cite{Gigli17} (with a slight difference in the terminology,
since here we distinguish between non-complete normed modules and Banach modules). Unless otherwise specified,
the discussion is essentially taken from \cite{Gigli14,Gigli17,Ben18}.
\begin{definition}[Banach \(L^0\)-module]\label{def:Ban_mod}
Let \(\XX\) be a \(\sigma\)-finite measure space, \(\mathscr M\) a module over \(L^0(\XX)\).
Then we say that \(\mathscr M\) is a \textbf{normed \(L^0(\XX)\)-module} if it is endowed with
a map \(|\cdot|\colon\mathscr M\to L^0(\XX)^+\), which is said to be a \textbf{pointwise norm}
on \(\mathscr M\), verifying the following properties:
\[\begin{split}
|v|\geq 0&\quad\text{ for every }v\in\mathscr M,\text{ with equality if and only if }v=0,\\
|v+w|\leq|v|+|w|&\quad\text{ for every }v,w\in\mathscr M,\\
|f\cdot v|=|f||v|&\quad\text{ for every }f\in L^0(\XX)\text{ and }v\in\mathscr M.
\end{split}\]
Moreover, we say that \(\mathscr M\) is a \textbf{Banach \(L^0(\XX)\)-module} when the following distance is complete:
\[
\sfd_{\mathscr M}(v,w)\coloneqq\sfd_{L^0(\XX)}(|v-w|,0)\quad\text{ for every }v,w\in\mathscr M.
\]
\end{definition}

In the case where \(\XX_{\sf o}=(\{{\sf o}\},\delta_{\sf o})\) is the one-point probability space,
the normed \(L^0(\XX_{\sf o})\)-modules (resp.\ the Banach \(L^0(\XX_{\sf o})\)-modules) can be identified
with the normed spaces (resp.\ the Banach spaces), with the only caveat that the distance \(\sfd_{\mathscr M}\)
associated with a normed \(L^0(\XX_{\sf o})\)-module \(\mathscr M\) is not induced by the norm
\(\|\cdot\|_{\mathscr M}\) of \(\mathscr M\). However, one has \(\sfd_{\mathscr M}(v,0)=\|v\|_{\mathscr M}\wedge 1\)
for all \(v\in\mathscr M\).
\medskip

Next, we recall/introduce a number of definitions related to normed and Banach \(L^0(\XX)\)-modules.
Let \(\XX=(\X,\Sigma,\mm)\) be a \(\sigma\)-finite measure space. Let \(\mathscr M\) and \(\mathscr N\)
be normed \(L^0(\XX)\)-modules. Given any measurable set \(E\in\Sigma\), we can consider the `localisation'
of \(\mathscr M\) on \(E\), i.e.\ the space
\[
\mathscr M|_E\coloneqq\1_E\cdot\mathscr M=\big\{v\in\mathscr M\;\big|\;\1_{\X\setminus E}\cdot v=0\big\}.
\]
We can regard \(\mathscr M|_E\) either as a normed \(L^0(\XX|_E)\)-module or as a normed \(L^0(\XX)\)-submodule of \(\mathscr M\).
We say that some elements \(v_1,\ldots,v_n\in\mathscr M\) are \textbf{independent} on \(E\) provided the mapping
\[
L^0(\XX|_E)^n\ni(f_1,\ldots,f_n)\mapsto\sum_{i=1}^n f_i\cdot v_i\in\mathscr M|_E
\]
is injective, while a vector subspace \(\mathcal V\subseteq\mathscr M\) \textbf{generates} \(\mathscr M\)
on \(E\) if \(\mathscr M|_E={\rm cl}_{\mathscr M}\big(\1_E\cdot\mathscr G(\mathcal V)\big)\), where
\[
\mathscr G(\mathcal S)\coloneqq\bigg\{\sum_{i=1}^n\1_{E_i}\cdot v_i\in\mathscr M\;\bigg|\;
n\in\N,\,(E_i)_{i=1}^n\subseteq\Sigma\text{ partition of }\X,\,(v_i)_{i=1}^n\subseteq\mathscr M\bigg\}
\]
for every subset \(\mathcal S\subseteq\mathscr M\). The module \(\mathscr M\) is said to be \textbf{finitely-generated}
if there exists a finite-dimensional vector subspace \(\mathcal V\subseteq\mathscr M\) that generates \(\mathscr M\)
(on \(\X\)). A \textbf{local basis} for \(\mathscr M\) on \(E\) is a collection of
elements \(v_1,\ldots,v_n\in\mathscr M\) that are independent on \(E\) and have the property that their linear span generates
\(\mathscr M\) on \(E\). In this case, \(L^0(\XX|_E)^n\ni(f_1,\ldots,f_n)\mapsto\sum_{i=1}^n f_i\cdot v_i\in\mathscr M|_E\)
is bijective. Since two local bases on \(E\) must have the same cardinality, one can unambiguously say that \(\mathscr M\)
has \textbf{local dimension} \(n\) on \(E\). Moreover, \(\mathscr M\) admits a (\(\mm\)-a.e.\ essentially unique)
\textbf{dimensional decomposition} \((D_n)_{n\in\N\cup\{\infty\}}\), which means that \((D_n)_{n\in\N\cup\{\infty\}}\subseteq\Sigma\)
is a partition of \(\X\) having the following property: \(\mathscr M\) has local dimension \(n\) on \(D_n\) for every \(n\in\N\),
and \(\mathscr M|_E\) is not finitely-generated whenever \(E\in\Sigma\) is a set satisfying \(E\subseteq D_\infty\) and \(\mm(E)>0\).
\medskip

The \textbf{support} of \(\mathscr M\) is the `biggest' subset \({\rm S}(\mathscr M)\) of \(\X\) where some element of \(\mathscr M\) is not null, i.e.
\[
{\rm S}(\mathscr M)\in\Sigma\quad\text{ is }\mm\text{-a.e.\ characterised by }\1_{{\rm S}(\mathscr M)}\coloneqq\bigvee_{v\in\mathscr M}\1_{\{|v|>0\}}.
\]
The space \(L^0(\XX)\) itself is a Banach \(L^0(\XX)\)-module with \({\rm S}(L^0(\XX))=\X\).
The \textbf{unit sphere} of \(\mathscr M\) is
\[
\mathbb S_{\mathscr M}\coloneqq\big\{v\in\mathscr M\;\big|\;|v|(x)\in\{0,1\}\text{ for }\mm\text{-a.e.\ }x\in\X\big\}.
\]
The \textbf{signum map} \({\rm sgn}\colon\mathscr M\to\mathbb S_{\mathscr M}\) on \(\mathscr M\) is defined as
\[
{\rm sgn}(v)\coloneqq\frac{\1_{\{|v|>0\}}}{|v|}\cdot v\in\mathbb S_{\mathscr M}\quad\text{ for every }v\in\mathscr M.
\]
Notice that \(v=|v|\cdot{\rm sgn}(v)\) for every \(v\in\mathscr M\). Moreover, we define the \textbf{unit disc} of \(\mathscr M\) as
\[
\mathbb D_{\mathscr M}\coloneqq\big\{v\in\mathscr M\;\big|\;|v|(x)\leq 1\text{ for }\mm\text{-a.e.\ }x\in\X\big\}.
\]
A map \(T\colon\mathscr M\to\mathscr N\) is said to be a \textbf{homomorphism of normed \(L^0(\XX)\)-modules}
provided it is \(L^0(\XX)\)-linear and continuous, or equivalently if it is linear and there exists \(g\in L^0(\XX)^+\) such that
\begin{equation}\label{eq:def_hom_nmod}
|T(v)|\leq g|v|\quad\text{ for every }v\in\mathscr M.
\end{equation}
We denote by \(\textsc{Hom}(\mathscr M;\mathscr N)\) the space of all homomorphisms of normed \(L^0(\XX)\)-modules from
\(\mathscr M\) to \(\mathscr N\). It is a normed \(L^0(\XX)\)-module if endowed with the natural pointwise operations and with
\[
|T|\coloneqq\bigvee_{v\in\mathscr M}\frac{\1_{\{|v|>0\}}|T(v)|}{|v|}=
\bigwedge\big\{g\in L^0(\XX)^+\;\big|\;g\text{ satisfies }\eqref{eq:def_hom_nmod}\big\}
\quad\text{ for all }T\in\textsc{Hom}(\mathscr M;\mathscr N).
\]
If \(\mathscr N\) is complete, then \(\textsc{Hom}(\mathscr M;\mathscr N)\) is a Banach \(L^0(\XX)\)-module. By an \textbf{isomorphism of normed \(L^0(\XX)\)-modules} we mean
a bijective homomorphism of normed \(L^0(\XX)\)-modules \(T\colon\mathscr M\to\mathscr N\) that preserves the pointwise norm, i.e.\ \(|T(v)|=|v|\) holds for every \(v\in\mathscr M\).
Whenever an isomorphism between \(\mathscr M\) and \(\mathscr N\) exists, we write \(\mathscr M\cong\mathscr N\).
The \textbf{kernel} \({\rm ker}(T)\) of any \(T\in\textsc{Hom}(\mathscr M;\mathscr N)\), which is given by
\({\rm ker}(T)\coloneqq\big\{v\in\mathscr M\,:\,T(v)=0\big\}\), is a closed normed \(L^0(\XX)\)-submodule of \(\mathscr M\).
\medskip

The \textbf{dual} of a normed \(L^0(\XX)\)-module \(\mathscr M\) is defined as
\[
\mathscr M^*\coloneqq\textsc{Hom}(\mathscr M;L^0(\XX)).
\]
If \(\mathscr M\) is a Banach \(L^0(\XX)\)-module and \(\mathscr V\) is a Banach \(L^0(\XX)\)-submodule of \(\mathscr M\), then
we have that the \textbf{quotient module} \(\mathscr M/\mathscr V\) is a Banach \(L^0(\XX)\)-module if endowed with the pointwise norm
\[
|w+\mathscr V|\coloneqq\bigwedge_{v\in\mathscr V}|w+v|\quad\text{ for every }w+\mathscr V\in\mathscr M/\mathscr V.
\]
Any normed \(L^0(\XX)\)-module \(\mathscr M\) has a unique \textbf{completion} \((\bar{\mathscr M},\iota)\),
i.e.\ \(\bar{\mathscr M}\) is a Banach \(L^0(\XX)\)-module and \(\iota\colon\mathscr M\to\bar{\mathscr M}\) is a
pointwise norm preserving homomorphism of normed \(L^0(\XX)\)-modules such that \(\iota[\mathscr M]\) is dense
in \(\bar{\mathscr M}\). Uniqueness is in this sense:
given any \((\tilde{\mathscr M},\tilde\iota)\) having the same properties as \((\bar{\mathscr M},\iota)\), there is a unique
isomorphism of Banach \(L^0(\XX)\)-modules \(\phi\colon\bar{\mathscr M}\to\tilde{\mathscr M}\) with \(\tilde\iota=\phi\circ\iota\).
\begin{definition}[Categories of Banach \(L^0(\XX)\)-modules]
Let \(\XX\) be a \(\sigma\)-finite measure space. Then:
\begin{itemize}
\item[\(\rm i)\)] We denote by \({\bf BanMod}_\XX\) the category whose objects are the Banach \(L^0(\XX)\)-modules, and where
the morphisms between Banach \(L^0(\XX)\)-modules \(\mathscr M\), \(\mathscr N\) are given by \(\textsc{Hom}(\mathscr M;\mathscr N).\)
\item[\(\rm ii)\)] We denote by \({\bf BanMod}_\XX^1\) the category whose objects are the Banach \(L^0(\XX)\)-modules, and where
the morphisms between Banach \(L^0(\XX)\)-modules \(\mathscr M\),\(\mathscr N\) are given by \(\mathbb D_{\textsc{Hom}(\mathscr M;\mathscr N)}\).
\end{itemize}
\end{definition}

Notice that \({\bf BanMod}_\XX^1\) is a lluf subcategory of \({\bf BanMod}_\XX\). It is proved in \cite[Theorem 3.13]{Pas22}
that \({\bf BanMod}_\XX^1\) is a bicomplete category (i.e.\ it admits all small limits and colimits), while it is observed
in \cite[Remark 3.1]{Pas22} that \({\bf BanMod}_\XX\) admits all finite limits and colimits.
\begin{theorem}[Hahn--Banach]\label{thm:Hahn-Banach}
Let \(\XX\) be a \(\sigma\)-finite measure space and \(\mathscr M\) a Banach \(L^0(\XX)\)-module. Then for any given
\(v\in\mathscr M\) there exists an element \(\omega_v\in\mathbb S_{\mathscr M^*}\) such that \(\omega_v(v)=|v|\).
\end{theorem}

Theorem \ref{thm:Hahn-Banach} appeared in \cite{Gigli14} and was obtained as a consequence of the classical
Hahn--Banach theorem. For a more direct proof tailored to normed modules, we refer to \cite[Theorem 3.30]{LP23}.
\medskip

A \textbf{norming subset} of \(\mathscr M^*\) is a subset \(\mathcal T\subseteq\mathbb D_{\mathscr M^*}\)
satisfying \(|v|=\bigvee_{\omega\in\mathcal T}\omega(v)\) for every \(v\in\mathscr M\).
The Hahn--Banach theorem ensures that the unit disc \(\mathbb D_{\mathscr M^*}\) itself is a norming subset of \(\mathscr M^*\).
\begin{definition}[Weak\(^*\) topology]\label{def:weak_star_top}
Let \(\XX\) be a \(\sigma\)-finite measure space and let \(\mathscr M\) be a Banach \(L^0(\XX)\)-module.
Then we define the \textbf{weak\(^*\) topology} on \(\mathscr M^*\) as the coarsest topology induced by
\[
\{\delta_v\,:\,v\in\mathscr M\},\quad\text{ where }\delta_v\colon\mathscr M^*\to L^0(\XX)\text{ is given by }
\delta_v(\omega)\coloneqq\omega(v)\text{ for every }\omega\in\mathscr M^*.
\]
\end{definition}
\begin{remark}{\rm
Similarly, one could define a weak topology on \(\mathscr M\). Moreover, the weak and weak\(^*\) topologies
on Banach \(L^0\)-modules verify properties that generalise the corresponding ones for Banach spaces.
However, in this paper we will not investigate further in this direction.
\fr}\end{remark}
\subsubsection{Examples of Banach \texorpdfstring{\(L^0\)}{L0}-modules}
We recall two key examples of Banach \(L^0(\XX)\)-modules.
\begin{definition}[The space \(L^0(\XX;\B)\)]
Let \(\XX=(\X,\Sigma,\mm)\) be a \(\sigma\)-finite measure space, \(\B\) a Banach space. Then we denote by \(L^0(\XX;\B)\) the space of all measurable maps from \(\X\)
to \(\B\) taking values into a separable subset of \(\B\) (which depends on the map itself), quotiented up to \(\mm\)-a.e.\ equality.
\end{definition}

The \emph{\(L^0\)-Lebesgue--Bochner space} \(L^0(\XX;\B)\) is a Banach \(L^0(\XX)\)-module if endowed with
\[
|v|(x)\coloneqq\|v(x)\|_\B\quad\text{ for every }v\in L^0(\XX;\B)\text{ and }\mm\text{-a.e.\ }x\in\X.
\]
Given a vector \({\sf v}\in\B\), we denote by \(\underline{\sf v}\in L^0(\XX;\B)\) the vector field that is
a.e.\ equal to \(\sf v\), i.e.\ we set
\begin{equation}\label{eq:def_const_vf}
\underline{\sf v}(x)\coloneqq{\sf v}\in\B\quad\text{ for }\mm\text{-a.e.\ }x\in\X.
\end{equation}

We also recall the following definition of \emph{module-valued space of generalised sequences}:
\begin{definition}[The space \(\ell_p(I,\mathscr M)\)]
Let \(I\neq\varnothing\) be an index set and \(p\in[1,\infty)\). Let \(\XX\) be a \(\sigma\)-finite measure space and \(\mathscr M\) a Banach \(L^0(\XX)\)-module.
Given any \(v=(v_i)_{i\in I}\in\mathscr M^I\), we set
\[
|v|_p\coloneqq\bigvee_{F\in\mathscr P_f(I)}\Big(\sum_{i\in F}|v_i|^p\Big)^{1/p}.
\]
Notice that \(|v|_p\in L^0_{\rm ext}(\XX)^+\) for every \(v\in\mathscr M^I\). Then we define the space \(\ell_p(I,\mathscr M)\) as
\[
\ell_p(I,\mathscr M)\coloneqq\big\{v\in\mathscr M^I\;\big|\;|v|_p\in L^0(\XX)^+\big\}.
\]
\end{definition}

The space \(\big(\ell_p(I,\mathscr M),|\cdot|_p\big)\) is a Banach \(L^0(\XX)\)-module, as it follows from \cite[Proposition 3.10]{Pas22} (notice indeed that
\(\ell_p(I,\mathscr M)\) is a particular example of \(\ell_p\)-sum in the sense of \cite[Definition 3.9]{Pas22}).
\subsubsection{Fiberwise representation of a Banach \texorpdfstring{\(L^0\)}{L0}-module}
One can easily check that the space of measurable sections of a measurable Banach bundle is a Banach \(L^0(\XX)\)-module,
a particular example being given by the \(L^0\)-Lebesgue--Bochner space \(L^0(\XX;\B)\), which corresponds to the constant
bundle \(\B\). On the other hand, it is much more difficult to show the converse, i.e.\ that any Banach \(L^0(\XX)\)-module
can be represented as the space of sections of some measurable Banach bundle. Results in this direction have been obtained
in \cite{LP18,DMLP21}. We will use one such result (i.e.\ Theorem \ref{thm:Serre-Swan} below) to prove Lemma \ref{lem:fin-gen_cont},
which in turn will be essential in order to obtain Lemma \ref{lem:crit_null_tensor}, and accordingly to introduce projective
and injective tensor products of Banach \(L^0\)-modules.
\medskip

Given a \(\sigma\)-finite measure space \(\XX=(\X,\Sigma,\mm)\), a separable Banach space \(\B\), and measurable maps
\(v_1,\ldots,v_n\colon\X\to\B\), we say that the multivalued map \(\X\ni x\mapsto{\bf E}(x)\subseteq\B\), which we define as
\[
{\bf E}(x)\coloneqq{\rm span}\{v_1(x),\ldots,v_n(x)\}\subseteq\B\quad\text{ for every }x\in\X,
\]
is a \textbf{measurable Banach bundle} on \(\XX\). Notice that each \({\bf E}(x)\) is a closed vector subspace of \(\B\).
The space \(\Gamma_\mm({\bf E})\) of all \(\mm\)-measurable \textbf{sections} of \(\bf E\) is then defined as the set of
all measurable maps \(v\colon\X\to\B\) satisfying \(v(x)\in{\bf E}(x)\) for \(\mm\)-a.e.\ \(x\in\X\), quotiented up to
\(\mm\)-a.e.\ equality. It turns out that \(\Gamma_\mm({\bf E})\) is a Banach \(L^0(\XX)\)-module if endowed with the
natural pointwise operations.
\begin{theorem}[Fiberwise representation of Banach \(L^0\)-modules]\label{thm:Serre-Swan}
Let \(\XX\) be a \(\sigma\)-finite measure space and \(\mathscr M\) a Banach \(L^0(\XX)\)-module. Let \(\B\)
be a universal separable Banach space. Suppose \(\mathscr M\) has local dimension \(n\in\N\) on a set \(E\in\Sigma\).
Then there exist measurable maps \(v_1,\ldots,v_n\colon\X\to\B\) such that \(\mathscr M|_E\cong\Gamma_\mm({\bf E})\),
where we set \({\bf E}(x)\coloneqq{\rm span}\{v_1(x),\ldots,v_n(x)\}\) for every \(x\in\X\).
\end{theorem}

Theorem \ref{thm:Serre-Swan} was first proved in \cite{LP18}, but we preferred to present its reformulation from \cite{DMLP21}.
\begin{lemma}\label{lem:fin-gen_cont}
Let \(\XX\) be a \(\sigma\)-finite measure space and let \(\mathscr M\) be a finitely-generated Banach \(L^0(\XX)\)-module.
Let \(T\colon\mathscr M\to L^0(\XX)\) be an \(L^0(\XX)\)-linear operator. Then it holds that \(T\in\mathscr M^*\).
\end{lemma}
\begin{proof}
Let \(D_0,\ldots,D_{\bar n}\) be the dimensional decomposition of \(\mathscr M\). Fix any \(n=1,\ldots,\bar n\). Thanks to Theorem \ref{thm:Serre-Swan},
we can find measurable vector fields \(v_1,\ldots,v_n\colon\X\to\B\), where \(\B\) is any given universal separable Banach space, such that
\(v_1(x),\ldots,v_n(x)\in\B\) are linearly independent for every \(x\in D_n\) and \(\Gamma_\mm({\bf E})\) is isomorphic to \(\mathscr M|_{D_n}\),
where we set \({\bf E}(x)\coloneqq{\rm span}\{v_1(x),\ldots,v_n(x)\}\) for every \(x\in\X\). For any \(i=1,\ldots,n\), we choose a measurable
representative \(\phi_i\colon\X\to\R\) of the function \(T(v_i)\in L^0(\XX)\), where we are identifying \(\mathscr M|_{D_n}\) with \(\Gamma_\mm({\bf E})\).
Given any point \(x\in D_n\), the unique linear operator from \({\bf E}(x)\) to \(\R\) sending each \(v_i(x)\) to \(\phi_i(x)\) is continuous, thus
\[
g_n(x)\coloneqq\sup\bigg\{\frac{\big|\sum_{i=1}^n q_i\phi_i(x)\big|}{\big\|\sum_{i=1}^n q_i v_i(x)\big\|_\B}\;\bigg|\;(q_1,\ldots,q_n)\in\mathbb Q^n\setminus\{0\}\bigg\}<+\infty.
\]
Notice that \(g_n\) is measurable by construction. Moreover, any element \(v\in\mathscr M|_{D_n}\) can be written (in a unique way) as
\(v=\sum_{i=1}^n f_i\cdot v_i\) for some \(f_1,\ldots,f_n\in L^0(\XX)\), so that we can estimate
\[
|T(v)|(x)=\bigg|\sum_{i=1}^n f_i(x)\phi_i(x)\bigg|\leq g_n(x)\bigg\|\sum_{i=1}^n f_i(x)v_i(x)\bigg\|_\B=g_n(x)|v|(x)\quad\text{ for }\mm\text{-a.e.\ }x\in D_n.
\]
Therefore, letting \(g\coloneqq\sum_{n=1}^{\bar n}\1_{D_n}g_n\in L^0(\XX)\), we conclude that \(|T(v)|\leq g|v|\) for every \(v\in\mathscr M\).
\end{proof}
\subsubsection{Pullback modules}
Let \(\XX=(\X,\Sigma_\X,\mm_\X)\) and \(\YY=(\Y,\Sigma_\Y,\mm_\Y)\) be \(\sigma\)-finite measure spaces. Let
\(\varphi\colon\X\to\Y\) be a measurable map such that \(\varphi_\#\mm_\X\ll\mm_\Y\). Notice that \(\varphi\) induces via
composition a ring homomorphism \(L^0(\YY)\ni f\mapsto f\circ\varphi\in L^0(\XX)\) that is also a Riesz homomorphism.
\medskip

This is an instance of a general phenomenon: given a Banach \(L^0(\YY)\)-module \(\mathscr M\), there is a unique couple
\((\varphi^*\mathscr M,\varphi^*)\), where \(\varphi^*\mathscr M\) is a Banach \(L^0(\XX)\)-module,
\(\varphi^*\colon\mathscr M\to\varphi^*\mathscr M\) is linear, and
\[\begin{split}
|\varphi^*v|=|v|\circ\varphi&\quad\text{ for every }v\in\mathscr M,\\
\varphi^*[\mathscr M]&\quad\text{ generates }\varphi^*\mathscr M\text{ on }\X.
\end{split}\]
We say that \(\varphi^*\mathscr M\) is the \textbf{pullback module} of \(\mathscr M\) under \(\varphi\).
Uniqueness is in the sense of the following universal property: given any couple \((\mathscr N,T)\)
having the same properties as \((\varphi^*\mathscr M,\varphi^*)\), there exists a unique isomorphism of
Banach \(L^0(\XX)\)-modules \(\phi\colon\varphi^*\mathscr M\to\mathscr N\) with \(T=\phi\circ\varphi^*\).
\medskip

The pullback of the dual \(\varphi^*\mathscr M^*\) is isomorphic to a Banach \(L^0(\XX)\)-submodule of the dual of the pullback
\((\varphi^*\mathscr M)^*\), but in general the two spaces do not coincide. More precisely, the unique homomorphism of Banach
\(L^0(\XX)\)-modules \({\sf I}_\varphi\colon\varphi^*\mathscr M^*\to(\varphi^*\mathscr M)^*\) satisfying
\begin{equation}\label{eq:def_I_varphi}
{\sf I}_\varphi(\varphi^*\omega)(\varphi^*v)=\omega(v)\circ\varphi\quad\text{ for every }\omega\in\mathscr M^*\text{ and }v\in\mathscr M
\end{equation}
preserves the pointwise norm, but in general is not surjective. However, the following fact holds:
\begin{theorem}[Sequential weak\(^*\) density of \(\varphi^*\mathscr M^*\) in \((\varphi^*\mathscr M)^*\)]
\label{thm:seq_weak-star_density}
Let \(\XX\), \(\YY\) be separable, \(\sigma\)-finite measure spaces. Let \(\varphi\colon\X\to\Y\) be a measurable map such that \(\varphi_\#\mm_\X\ll\mm_\Y\).
Let \(\mathscr M\) be a Banach \(L^0(\YY)\)-module. Let \(\Theta\in(\varphi^*\mathscr M)^*\) be given. Then there exists a sequence
\((\Theta_n)_{n\in\N}\subseteq\varphi^*\mathscr M^*\) such that \({\sf I}_\varphi(\Theta_n)\to\Theta\) with respect to the weak\(^*\) topology of \((\varphi^*\mathscr M)^*\)
introduced in Definition \ref{def:weak_star_top}.
\end{theorem}

Theorem \ref{thm:seq_weak-star_density} was proved in \cite[Theorem B.1]{GLP22}. It is unclear whether the separability assumption on \(\XX\) and \(\YY\),
which is due only to the proof strategy of \cite[Theorem B.1]{GLP22}, can be dropped.
\subsubsection{Bounded \texorpdfstring{\(L^0\)}{L0}-bilinear operators}
In Sections \ref{s:proj_tens} and \ref{s:inj_tens} we will need to use the space \({\rm B}(\mathscr M,\mathscr N)\):
\begin{definition}[The space \({\rm B}(\mathscr M,\mathscr N;\mathscr Q)\)]
Let \(\XX\) be a \(\sigma\)-finite measure space. Let \(\mathscr M\), \(\mathscr N\), \(\mathscr Q\) be normed \(L^0(\XX)\)-modules. Then we denote by
\({\rm B}(\mathscr M,\mathscr N;\mathscr Q)\) the space of all those \(L^0(\XX)\)-bilinear operators \(b\colon\mathscr M\times\mathscr N\to\mathscr Q\)
that are also continuous. We also set \({\rm B}(\mathscr M,\mathscr N)\coloneqq{\rm B}(\mathscr M,\mathscr N;L^0(\XX))\).
\end{definition}

One can readily check that a bilinear map \(b\colon\mathscr M\times\mathscr N\to\mathscr Q\) is \(L^0(\XX)\)-bilinear and continuous
(i.e.\ it belongs to \({\rm B}(\mathscr M,\mathscr N;\mathscr Q)\)) if and only if there exists a function \(g\in L^0(\XX)^+\) such that
\begin{equation}\label{eq:def_bilin_ptwse_norm}
|b(v,w)|\leq g|v||w|\quad\text{ for every }(v,w)\in\mathscr M\times\mathscr N.
\end{equation}
Moreover, \({\rm B}(\mathscr M,\mathscr N;\mathscr Q)\) is a normed \(L^0(\XX)\)-module if endowed with the pointwise operations and
\[
|b|\coloneqq\bigvee_{(v,w)\in\mathscr M\times\mathscr N}\frac{\1_{\{|v||w|>0\}}|b(v,w)|}{|v||w|}=
\bigwedge\big\{g\in L^0(\XX)^+\;\big|\;g\text{ satisfies }\eqref{eq:def_bilin_ptwse_norm}\big\}
\]
for every \(b\in{\rm B}(\mathscr M,\mathscr N;\mathscr Q)\). If \(\mathscr Q\) is complete, then \({\rm B}(\mathscr M,\mathscr N;\mathscr Q)\) is a Banach \(L^0(\XX)\)-module.
\medskip

If \(\mathscr M\), \(\mathscr N\), \(\mathscr Q\) are normed \(L^0(\XX)\)-modules, then each \(b\in{\rm B}(\mathscr M,\mathscr N;\mathscr Q)\) can be uniquely extended
to \(\bar b\in{\rm B}(\bar{\mathscr M},\bar{\mathscr N};\bar{\mathscr Q})\), where \(\bar{\mathscr M}\), \(\bar{\mathscr N}\), \(\bar{\mathscr Q}\) are the
completions of \(\mathscr M\), \(\mathscr N\), \(\mathscr Q\), respectively, and \(|\bar b|=|b|\).
\section{Auxiliary results on Banach \texorpdfstring{\(L^0\)}{L0}-modules}
\subsection{Quotient operators between Banach \texorpdfstring{\(L^0\)}{L0}-modules}\label{s:quotient_oper}
We begin with the key definition:
\begin{definition}[Quotient operator]
Let \(\XX\) be a \(\sigma\)-finite measure space. Let \(\mathscr M\) and \(\mathscr N\) be normed \(L^0(\XX)\)-modules.
Then we say that a homomorphism \(T\colon\mathscr M\to\mathscr N\) of normed \(L^0(\XX)\)-modules is a
\textbf{quotient operator} provided it is surjective and it satisfies
\[
|w|=\bigwedge_{v\in T^{-1}(w)}|v|\quad\text{ for every }w\in\mathscr N.
\]
\end{definition}

Notice that each quotient operator \(T\colon\mathscr M\to\mathscr N\) verifies \(|T|\leq 1\). More precisely, it holds that
\begin{equation}\label{eq:norm_quotient_oper}
|T|=\1_{{\sf S}(\mathscr N)}.
\end{equation}
If \(\mathscr M\), \(\mathscr N\) are two Banach \(L^0(\XX)\)-modules and \(T\colon\mathscr M\to\mathscr N\)
is a homomorphism of Banach \(L^0(\XX)\)-modules, then \(T\) is a quotient operator if and only if the unique
map \(\hat T\colon\mathscr M/{\rm ker}(T)\to\mathscr N\) satisfying
\[\begin{tikzcd}
\mathscr M \arrow[r,"T"] \arrow[d,swap,"\pi"] & \mathscr N \\
\mathscr M/{\rm ker}(T) \arrow[ur,swap,"\hat T"] &
\end{tikzcd}\]
is an isomorphism of Banach \(L^0(\XX)\)-modules, with \(\pi\colon\mathscr M\to\mathscr M/{\rm ker}(T)\)
the canonical projection.

\begin{remark}\label{rmk:ext_quotient_oper}{\rm
If \(T\colon\mathscr M\to\mathscr N\) is a quotient operator between normed \(L^0(\XX)\)-modules, then its unique linear
continuous extension \(\bar T\colon\bar{\mathscr M}\to\bar{\mathscr N}\) to the completions is a quotient operator.
\fr}\end{remark}

The glueing property of \(\mathscr M\) ensures that if \(T\colon\mathscr M\to\mathscr N\) is a quotient operator and \(w\in\mathscr N\) is given,
then for every \(\varepsilon>0\) we can find an element \(v\in\mathscr M\) such that \(T(v)=w\) and \(|v|\leq|w|+\varepsilon\).
\begin{lemma}\label{lem:annihilator}
Let \(\XX\) be a \(\sigma\)-finite measure space and let \(\mathscr M\) be a Banach \(L^0(\XX)\)-module.
Given any Banach \(L^0(\XX)\)-submodule \(\mathscr V\) of \(\mathscr M\), we define the \textbf{annihilator}
\(\mathscr V^\perp\) of \(\mathscr V\) in \(\mathscr M^*\) as
\[
\mathscr V^\perp\coloneqq\big\{\omega\in\mathscr M^*\;\big|\;\omega(v)=0\text{ for every }v\in\mathscr V\big\}.
\]
Then \(\mathscr V^\perp\) is a Banach \(L^0(\XX)\)-submodule of \(\mathscr M^*\). Moreover, it holds that
\[
\mathscr V^*\cong\mathscr M^*/\mathscr V^\perp,
\]
an isomorphism of Banach \(L^0(\XX)\)-modules being given by
\(\mathscr M^*/\mathscr V^\perp\ni\omega+\mathscr V^\perp\mapsto\omega|_{\mathscr V}\in\mathscr V^*\).
\end{lemma}
\begin{proof}
It is straightforward to check that \(\mathscr V^\perp\) is a Banach \(L^0(\XX)\)-submodule of \(\mathscr M^*\).
Consider the homomorphism of Banach \(L^0(\XX)\)-modules \(T\colon\mathscr M^*\to\mathscr V^*\) given by
\(T(\omega)\coloneqq\omega|_{\mathscr V}\) for all \(\omega\in\mathscr M^*\). Observe that \(|T|\leq 1\).
Moreover, the Hahn--Banach theorem ensures that for any \(\eta\in\mathscr V^*\) we can find \(\omega\in\mathscr M^*\)
such that \(T(\omega)=\eta\) and \(|\omega|=|\eta|\). This shows that \(T\) is a quotient operator.
Since \({\rm ker}(T)=\mathscr V^\perp\), we conclude that the operator
\(\mathscr M^*/\mathscr V^\perp\ni\omega+\mathscr V^\perp\mapsto\omega|_{\mathscr V}\in\mathscr V^*\)
is an isomorphism of Banach \(L^0(\XX)\)-modules. Therefore, the proof of the statement is complete.
\end{proof}

We conclude with a sufficient condition for a given homomorphism to be a quotient operator:
\begin{lemma}\label{lem:suff_cond_quotient_oper}
Let \(\XX\) be a \(\sigma\)-finite measure space. Let \(\mathscr M\), \(\mathscr N\) be Banach \(L^0(\XX)\)-modules.
Let \(\mathscr W\) be a dense vector subspace of \(\mathscr N\). Let \(T\colon\mathscr M\to\mathscr N\) be a homomorphism of Banach \(L^0(\XX)\)-modules
with \(|T|\leq 1\) satisfying the following property: given any \(w\in\mathscr N\) and \(\varepsilon>0\), there exists
\(v\in\mathscr M\) such that \(\sfd_{\mathscr N}(T(v),w)<\varepsilon\) and \(\sfd_{L^0(\XX)}(|v|,|w|)<\varepsilon\).
Then \(T\) is a quotient operator.
\end{lemma}
\begin{proof}
Let \(w\in\mathscr N\) and \(k\in\N\) be given. Set \(u^k_0\coloneqq 0\in\mathscr M\) and find recursively
\(u^k_n\in\mathscr M\) for \(n\in\N\) such that \(\sfd_{\mathscr N}\big(T(u^k_n),w-\sum_{i=0}^{n-1}T(u^k_i)\big)<2^{-k-n-1}\)
and \(\sfd_{L^0(\XX)}\big(|u^k_n|,\big|w-\sum_{i=0}^{n-1}T(u^k_i)\big|\big)<2^{-k-n-1}\).
Now define \(v^k_n\coloneqq\sum_{i=1}^n u^k_i\) for every \(n\in\N\). Then we have that \(\sum_{n\in\N}\sfd_{\mathscr M}(v^k_{n+1},v^k_n)<+\infty\), since
\[
\sfd_{\mathscr M}(v^k_{n+1},v^k_n)\leq\sfd_{L^0(\XX)}\bigg(|u^k_{n+1}|,\Big|w-\sum_{i=0}^n T(u^k_i)\Big|\bigg)
+\sfd_{\mathscr N}\bigg(T(u^k_n),w-\sum_{i=0}^{n-1}T(u^k_i)\bigg)<\frac{3}{2^{k+n+2}}.
\]
It follows that \((v^k_n)_{n\in\N}\subseteq\mathscr M\) is Cauchy, thus it makes sense to define
\(v^k\coloneqq\lim_n v^k_n\in\mathscr M\). Since
\[
\sfd_{\mathscr N}(T(v^k_n),w)=\sfd_{\mathscr N}\bigg(T(u^k_n),w-\sum_{i<n}T(u^k_i)\bigg)\leq\frac{1}{2^{k+n+1}}\to 0
\quad\text{ as }n\to\infty,
\]
the continuity of the map \(T\) ensures that \(T(v^k)=w\). Moreover, we can estimate
\[
|v^k|\leq\sum_{n=1}^\infty|u^k_n|\leq|w|+\underset{\eqqcolon r_k}{\underbrace{\sum_{n=1}^\infty\bigg||u^k_n|-\Big|w-\sum_{i<n}T(u^k_i)\Big|\bigg|
+\sum_{n=2}^\infty\bigg|T(u^k_{n-1})-\Big(w-\sum_{i<n-1}T(u^k_i)\Big)\bigg|}}.
\]
Given that \(\sfd_{L^0(\XX)}(r_k,0)<2^{-k}\) and \(|w|=|T(v^k)|\leq|v^k|\), we can extract a subsequence
\((k_j)_{j\in\N}\subseteq\N\) such that \(|v^{k_j}|\to|w|\) in the \(\mm\)-a.e.\ sense. This implies that
\(|w|=\bigwedge_{v\in T^{-1}(w)}|v|\), as desired.
\end{proof}
\subsection{Summability in Banach \texorpdfstring{\(L^0\)}{L0}-modules}\label{s:summability}
First, we introduce a notion of \emph{summable family}:
\begin{definition}[Summable family in a Banach \(L^0\)-module]\label{def:summable_family}
Let \(\XX\) be a \(\sigma\)-finite measure space and \(\mathscr M\) a Banach \(L^0(\XX)\)-module. Then we say that a family \(\{v_i\}_{i\in I}\subseteq\mathscr M\)
is \textbf{summable} in \(\mathscr M\) if
\begin{equation}\label{eq:summable_mod}
\bigwedge_{F\in\mathscr P_f(I)}\,\bigvee_{G\in\mathscr P_f(I\setminus F)}\Big|v-\sum_{i\in F\cup G}v_i\Big|=0\quad\text{ for some }v\in\mathscr M.
\end{equation}
The element \(v\in\mathscr M\) is unique, is called the \textbf{sum} of \(\{v_i\}_{i\in I}\) in \(\mathscr M\), and is denoted by \(\sum_{i\in I}v_i\).
\end{definition}

In the sequel, instead of the definition, we will often need the following summability criterion:
\begin{proposition}[Cauchy summability criterion]\label{prop:Cauchy_sum_criterion}
Let \(\XX\) be a \(\sigma\)-finite measure space. Let \(\mathscr M\) be a Banach \(L^0(\XX)\)-module.
Then it holds that a family \(\{v_i\}_{i\in I}\subseteq\mathscr M\) is summable if and only if
\begin{equation}\label{eq:Cauchy_criterion}
\bigwedge_{F\in\mathscr P_f(I)}\,\bigvee_{G\in\mathscr P_f(I\setminus F)}\Big|\sum_{i\in G}v_i\Big|=0.
\end{equation}
In this case, the set \(J\coloneqq\big\{i\in I\,:\,v_i\neq 0\big\}\) is at most countable. Moreover, given any increasing sequence \((F_n)_{n\in\N}\) of finite subsets
of \(J\) satisfying \(J=\bigcup_{n\in\N}F_n\), we have that
\begin{equation}\label{eq:summable_aux}
\sum_{i\in F_n}v_i\to\sum_{i\in I}v_i\quad\text{ as }n\to\infty.
\end{equation}
\end{proposition}
\begin{proof}
Suppose \(\{v_i\}_{i\in I}\) is summable and set \(v\coloneqq\sum_{i\in I}v_i\in\mathscr M\) for brevity. We have that
\[
\bigg|\sum_{i\in G}v_i\bigg|\leq\bigg|v-\sum_{i\in F}v_i\bigg|+\bigg|v-\sum_{i\in F\cup G}v_i\bigg|\quad\text{ for every }F\in\mathscr P_f(I)\text{ and }G\in\mathscr P_f(I\setminus F),
\]
whence it follows that \(\bigvee_{G\in\mathscr P_f(I\setminus F)}\big|\sum_{i\in G}v_i\big|\leq 2\bigvee_{G\in\mathscr P_f(I\setminus F)}\big|v-\sum_{i\in F\cup G}v_i\big|\) and accordingly
\[
\bigwedge_{F\in\mathscr P_f(I)}\,\bigvee_{G\in\mathscr P_f(I\setminus F)}\Big|\sum_{i\in G}v_i\Big|\leq 2\bigwedge_{F\in\mathscr P_f(I)}\,\bigvee_{G\in\mathscr P_f(I\setminus F)}\Big|v-\sum_{i\in F\cup G}v_i\Big|=0.
\]

Conversely, suppose that \eqref{eq:Cauchy_criterion} holds. Then we can find an increasing sequence \((\tilde F_k)_{k\in\N}\) of finite subsets of \(J\) such that
\(\psi_k\coloneqq\bigvee_{G\in\mathscr P_f(I\setminus\tilde F_k)}\big|\sum_{i\in G}v_i\big|\searrow 0\) holds \(\mm\)-a.e. Notice that \(J=\bigcup_{k\in\N}\tilde F_k\).
Up to a non-relabelled subsequence, we can also assume that \(\sfd_{L^0(\XX)}(\psi_k\wedge 1,0)\leq k^{-1}\) for all \(k\in\N\). Define
\(J_k\coloneqq\big\{i\in I\,:\,\sfd_{\mathscr M}(v_i,0)>k^{-1}\big\}\) for every \(k\in\N\). Given that \(|v_i|\leq\psi_k\) for every \(i\in I\setminus\tilde F_k\),
we deduce that \(\sfd_{\mathscr M}(v_i,0)\leq\sfd_{L^0(\XX)}(\psi_k\wedge 1,0)\leq k^{-1}\), so that \(i\notin J_k\). This shows that \(J_k\subseteq\tilde F_k\),
thus in particular \(J_k\) is finite. Since \(J=\bigcup_{k\in\N}J_k\), we deduce that \(J\) is at most countable. Moreover,
\[
\bigg|\sum_{i\in\tilde F_m}v_i-\sum_{i\in\tilde F_k}v_i\bigg|=\bigg|\sum_{i\in\tilde F_m\setminus\tilde F_k}v_i\bigg|\leq\psi_k\quad\text{ for every }k,m\in\N\text{ with }m\geq k
\]
implies that \(\big(\sum_{i\in\tilde F_k}v_i\big)_{k\in\N}\) is a Cauchy sequence in \(\mathscr M\). Denoting by \(v\in\mathscr M\) its limit, we claim that
\(\{v_i\}_{i\in I}\) is summable and \(v=\sum_{i\in I}v_i\). First, letting \(\phi_k\coloneqq\big|v-\sum_{i\in\tilde F_k}v_i\big|\in L^0(\XX)^+\),
we have that \(\phi_k\to 0\) in \(L^0(\XX)\) as \(k\to\infty\), thus in particular \(\bigwedge_{k\in\N}\phi_k=0\). Therefore, we deduce that
\[
\bigwedge_{F\in\mathscr P_f(I)}\,\bigvee_{G\in\mathscr P_f(I\setminus F)}\bigg|v-\sum_{i\in F\cup G}v_i\bigg|\leq
\bigwedge_{k\in\N}\,\bigvee_{G\in\mathscr P_f(I\setminus\tilde F_k)}\bigg|v-\sum_{i\in\tilde F_k\cup G}v_i\bigg|
\leq\bigwedge_{k\in\N}\phi_k+\bigwedge_{k\in\N}\psi_k=0,
\]
which shows that \(\{v_i\}_{i\in I}\) is summable with sum \(v\), as we claimed. Finally, given any increasing sequence \((F_n)_{n\in\N}\) of finite subsets of \(J\)
with \(J=\bigcup_{n\in\N}F_n\), we can extract a subsequence \((n_k)_{k\in\N}\) such that \(\tilde F_k\subseteq F_{n_k}\) for every \(k\in\N\), so that
\(\big|v-\sum_{i\in F_{n_k}}v_i\big|\leq\phi_k+\psi_k\) for every \(k\in\N\), whence it follows that \(\sum_{i\in F_{n_k}}v_i\to v\) as \(k\to\infty\).
Given that the limit \(v\) does not depend on the specific choice of the sequence \((F_n)_{n\in\N}\), we can conclude that \eqref{eq:summable_aux} is verified. The proof is complete.
\end{proof}

Furthermore, given any family \(\{f_i\}_{i\in I}\subseteq L^0(\XX)^+\), we define
\[
\sum_{i\in I}f_i\coloneqq\bigvee_{F\in\mathscr P_f(I)}\,\sum_{i\in F}f_i\in L^0_{\rm ext}(\XX)^+.
\]
This definition is not in conflict with Definition \ref{def:summable_family}, since \(\bigvee_{F\in\mathscr P_f(I)}\sum_{i\in F}f_i\in L^0(\XX)^+\)
if and only if \(\{f_i\}_{i\in I}\subseteq L^0(\XX)\) is summable. In this case, its sum coincides with \(\bigvee_{F\in\mathscr P_f(I)}\sum_{i\in F}f_i\). Moreover,
\[
|v|_p=\Big(\sum_{i\in I}|v_i|^p\Big)^{1/p}\quad\text{ for every }v=(v_i)_{i\in I}\in\ell_p(I,\mathscr M)
\]
holds whenever \(\mathscr M\) is a Banach \(L^0(\XX)\)-module and \(p\in[1,\infty)\). Let us also observe that
\begin{equation}\label{eq:sum_elem_ell1_I_M}
v=\sum_{i\in I}(\delta_{ij}v_i)_{j\in I}\quad\text{ for every }v=(v_i)_{i\in I}\in\ell_1(I,\mathscr M).
\end{equation}
Indeed, using the summability of \(\{|v_i|\}_{i\in I}\) in \(L^0(\XX)\) and Proposition \ref{prop:Cauchy_sum_criterion} we obtain that
\[
\bigwedge_{F\in\mathscr P_f(I)}\,\bigvee_{G\in\mathscr P_f(I\setminus F)}\bigg|v-\sum_{i\in F\cup G}(\delta_{ij}v_i)_{j\in I}\bigg|_1=
\bigwedge_{F\in\mathscr P_f(I)}\,\bigvee_{G\in\mathscr P_f(I\setminus F)}\,\bigvee_{H\in\mathscr P_f(I\setminus(F\cup G))}\,\sum_{i\in H}|v_i|=0,
\]
whence the claimed identity \eqref{eq:sum_elem_ell1_I_M} follows.
\begin{remark}\label{rmk:summable_ptwse_norms}{\rm
Let \(\{v_i\}_{i\in I}\) be a summable family in a Banach \(L^0(\XX)\)-module \(\mathscr M\). Then it holds
\begin{equation}\label{eq:summable_ptwse_norms}
\Big|\sum_{i\in I}v_i\Big|\leq\sum_{i\in I}|v_i|.
\end{equation}
Indeed, by Proposition \ref{prop:Cauchy_sum_criterion} we find \((F_n)_{n\in\N}\subseteq\mathscr P_f(I)\)
for which \(\big|\sum_{i\in F_n}v_i-\sum_{i\in I}v_i\big|\to 0\) in the \(\mm\)-a.e.\ sense as \(n\to\infty\), so that
\(\big|\sum_{i\in I}v_i\big|=\lim_n\big|\sum_{i\in F_n}v_i\big|\leq\lim_n\sum_{i\in F_n}|v_i|\leq\sum_{i\in I}|v_i|\). Also,
\[
\{v_i\}_{i\in I}\subseteq\mathscr M\,\text{ is summable}\quad\text{ for every }v=(v_i)_{i\in I}\in\ell_1(I,\mathscr M).
\]
Indeed, arguing as in the proof of \eqref{eq:sum_elem_ell1_I_M} we deduce that \eqref{eq:Cauchy_criterion} is verified, so that \(\{v_i\}_{i\in I}\)
is summable by Proposition \ref{prop:Cauchy_sum_criterion}. Notice also that \(\big|\sum_{i\in I}v_i\big|\leq|v|_1\) holds by \eqref{eq:summable_ptwse_norms}.
\fr}\end{remark}
\begin{lemma}\label{lem:summable_comp_dual}
Let \(\XX\) be a \(\sigma\)-finite measure space. Let \(\varphi\colon\mathscr M\to\mathscr N\) be a homomorphism
of Banach \(L^0(\XX)\)-modules \(\mathscr M\), \(\mathscr N\). Let \(\{v_i\}_{i\in I}\subseteq\mathscr M\)
be summable. Then \(\{\varphi(v_i)\}_{i\in I}\subseteq\mathscr N\) is summable and
\begin{equation}\label{eq:summable_comp_dual}
\varphi\Big(\sum_{i\in I}v_i\Big)=\sum_{i\in I}\varphi(v_i).
\end{equation}
\end{lemma}
\begin{proof}
Set \(v\coloneqq\sum_{i\in I}v_i\) for brevity. Notice that if \(F\in\mathscr P_f(I)\) and \(G\in\mathscr P_f(I\setminus F)\), then
\[
\bigg|\varphi(v)-\sum_{i\in F\cup G}\varphi(v_i)\bigg|=\bigg|\varphi\bigg(v-\sum_{i\in F\cup G}v_i\bigg)\bigg|
\leq|\varphi|\bigg|v-\sum_{i\in F\cup G}v_i\bigg|.
\]
By taking first the supremum over \(G\) and then the infimum over \(F\), we thus obtain \eqref{eq:summable_comp_dual}.
\end{proof}
\subsection{Local Schauder bases}\label{s:Schauder}
We propose a notion of (unconditional) Schauder basis in a Banach \(L^0\)-module. The term `unconditional'
will be often omitted, as no other kind of basis is considered.
\begin{definition}[Local Schauder basis]
Let \(\XX=(\X,\Sigma,\mm)\) be a \(\sigma\)-finite measure space and \(\mathscr M\) a Banach \(L^0(\XX)\)-module. Let \(E\in\Sigma\) satisfy \(\mm(E)>0\).
Then we say that a family \(\{v_i\}_{i\in I}\subseteq\mathscr M\) is an \textbf{(unconditional) local Schauder basis} of \(\mathscr M\) on \(E\) provided
for any given \(v\in\mathscr M|_E\) there exists a unique \((f_i)_{i\in I}\in L^0(\XX|_E)^I\) such that the family \(\{f_i\cdot v_i\}_{i\in I}\) is summable
in \(\mathscr M\) and
\[
v=\sum_{i\in I}f_i\cdot v_i.
\]
In the case where \(E=\X\), we just say that \(\{v_i\}_{i\in I}\) is a local Schauder basis of \(\mathscr M\).
\end{definition}
\begin{lemma}
Let \(\XX\) be a \(\sigma\)-finite measure space, \(\B\) a Banach space with a Schauder basis \(\{{\sf v}_i\}_{i\in I}\).
Then it holds that the family \(\{\underline{\sf v}_i\}_{i\in I}\) defined as in \eqref{eq:def_const_vf} is a local Schauder basis of \(L^0(\XX;\B)\).
\end{lemma}
\begin{proof}
Let \(v\in L^0(\XX;\B)\) be given. Fix a measurable representative \(\bar v\colon\X\to\B\) of \(v\). Since \(\{{\sf v}_i\}_{i\in I}\) is a Schauder basis
of \(\B\), for any point \(x\in\X\) we can find a unique \((\bar f_i(x))_{i\in I}\in\R^I\) such that
\begin{equation}\label{eq:Schauder_Leb-Boch}
\bar v(x)=\sum_{i\in I}\bar f_i(x){\sf v}_i.
\end{equation}
Thanks to \eqref{eq:direct_decomp_Schauder} and the classical Hahn--Banach theorem, for any index \(i\in I\) we can find \(\omega_i\in\B'\) (where \(\B'\) stands
for the topological dual of \(\B\)) with \(\omega_i({\sf v}_j)=0\) for every \(j\in I\setminus\{i\}\) and \(\omega_i({\sf v}_i)=1\). Hence, Lemma \ref{lem:summable_comp_dual} gives
\[
\omega_i(\bar v(x))=\omega_i\Big(\sum_{j\in I}\bar f_j(x){\sf v}_j\Big)=\sum_{j\in I}\bar f_j(x)\,\omega_i({\sf v}_j)=\bar f_i(x)\quad\text{ for every }x\in\X,
\]
whence it follows that \(\bar f_i\colon\X\to\R\) is measurable. Define \(f_i\coloneqq[\bar f_i]_\mm\in L^0(\XX)\) for every \(i\in I\). Since
\[
\inf_{F\in\mathscr P_f(I)}\,\sup_{G\in\mathscr P_f(I\setminus F)}\bigg\|\bar v(x)-\sum_{i\in F\cup G}\bar f_i(x){\sf v}_i\bigg\|_\B=0\quad\text{ for every }x\in\X
\]
by \eqref{eq:Schauder_Leb-Boch}, taking into account also Remark \ref{rmk:aux_estim_sup} (as well as its natural variants) we deduce that
\[
\bigwedge_{F\in\mathscr P_f(I)}\,\bigvee_{G\in\mathscr P_f(I\setminus F)}\Big|v-\sum_{i\in F\cup G}f_i\cdot\underline{\sf v}_i\Big|=0.
\]
This proves that \(\{f_i\cdot\underline{\sf v}_i\}_{i\in I}\) is summable in \(L^0(\XX;\B)\) and \(v=\sum_{i\in I}f_i\cdot\underline{\sf v}_i\).
Finally, let us check that \((f_i)_{i\in I}\in L^0(\XX)^I\) is the unique family with this property.
Suppose \((g_i)_{i\in I}\in L^0(\XX)^I\) satisfies the identity \(\sum_{i\in I}g_i\cdot\underline{\sf v}_i=v\) in \(L^0(\XX;\B)\).
By virtue of Proposition \ref{prop:Cauchy_sum_criterion}, the set \(J\coloneqq\big\{i\in I\,:\,g_i\neq 0\big\}\)
is at most countable. Fix a measurable representative \(\bar g_i\colon\X\to\R\) of \(g_i\) for every \(i\in J\).
Since the family of all couples \((F,G)\) with \(F\in\mathscr P_f(J)\) and \(G\in\mathscr P_f(J\setminus F)\) is
at most countable, we can find a set \(N\in\Sigma\) such that \(\mm(N)=0\) and
\[
\inf_{F\in\mathscr P_f(J)}\,\sup_{G\in\mathscr P_f(J\setminus F)}\bigg\|\bar v(x)-\sum_{i\in F\cup G}\bar g_i(x){\sf v}_i\bigg\|_\B=0
\quad\text{ for every }x\in\X\setminus N.
\]
It follows that \(\bar g_i(x)=\bar f_i(x)\) for every \(i\in J\) and \(x\in\X\setminus N\), as well as
\(\bar f_i(x)=0\) for every \(i\in I\setminus J\) and \(x\in\X\setminus N\). Hence, we conclude that
\((g_i)_{i\in I}=(f_i)_{i\in I}\), so that the statement is achieved.
\end{proof}
\subsubsection{Applications to spaces of generalised sequences and \texorpdfstring{\(L^0\)}{L0}-Lebesgue--Bochner spaces}
Fix an arbitrary index set \(I\neq\varnothing\). Given any exponent \(p\in[1,\infty)\) and any index \(i\in I\), we define
\[
p_i\big((a_j)_{j\in I}\big)\coloneqq a_i\quad\text{ for every }(a_j)_{j\in I}\in\ell_p(I).
\]
The resulting map \(p_i\colon\ell_p(I)\to\R\) is a \(1\)-Lipschitz linear operator. Hence, it makes sense to define
\[
a(\cdot)_i\coloneqq p_i\circ a\in L^0(\XX)\quad\text{ for every }i\in I\text{ and }a\in L^0(\XX;\ell_p(I))
\]
whenever \(\XX\) is a \(\sigma\)-finite measure space. Moreover, recall that any element \({\sf a}\in\ell_p(I)\)
is associated with the a.e.\ constant vector field \(\underline{\sf a}\in L^0(\XX;\ell_p(I))\), which is given by
\(\underline{\sf a}(x)\coloneqq{\sf a}\) for \(\mm\)-a.e.\ \(x\in\X\).
\begin{lemma}\label{lem:basic_L0_l1(I)}
Let \(\XX\) be a \(\sigma\)-finite measure space and \(I\neq\varnothing\) an index set. Let \(a\in L^0(\XX;\ell_1(I))\) be given.
Let \(({\sf e}_i)_{i\in I}\) be as in \eqref{eq:def_canonical_elements}.
Then the family \(\{a(\cdot)_i\cdot\underline{\sf e}_i\}_{i\in I}\) is summable in \(L^0(\XX;\ell_1(I))\) and
\begin{equation}\label{eq:basic_L0_l1(I)_cl1}
a=\sum_{i\in I}a(\cdot)_i\cdot\underline{\sf e}_i.
\end{equation}
In particular, the family \(\big\{|a(\cdot)_i|\big\}_{i\in I}\) is summable in \(L^0(\XX)\) and it holds that
\begin{equation}\label{eq:basic_L0_l1(I)_cl2}
|a|=\sum_{i\in I}|a(\cdot)_i|.
\end{equation}
\end{lemma}
\begin{proof}
Fix a measurable representative \(\bar a\colon\X\to\ell_1(I)\) of \(a\). As \(\{{\sf e}_i\}_{i\in I}\) is a Schauder basis of \(\ell_1(I)\),
\[
\inf_{F\in\mathscr P_f(I)}\,\sup_{G\in\mathscr P_f(I\setminus F)}\bigg\|\bar a(x)-\sum_{i\in F\cup G}\bar a(x)_i\,{\sf e}_i\bigg\|_{\ell_1(I)}=0\quad\text{ for every }x\in\X.
\]
Using that \(\bar a(x)_i=(p_i\circ\bar a)(x)\) and taking into account Remark \ref{rmk:aux_estim_sup}, we can thus conclude that
\[
\bigwedge_{F\in\mathscr P_f(I)}\,\bigvee_{G\in\mathscr P_f(I\setminus F)}\Big|a-\sum_{i\in F\cup G}a(\cdot)_i\cdot\underline{\sf e}_i\Big|=0,
\]
which gives the first claim \eqref{eq:basic_L0_l1(I)_cl1}.
Finally, \eqref{eq:basic_L0_l1(I)_cl2} follows from \eqref{eq:basic_L0_l1(I)_cl1} together with the fact that
\[
\bigg||a|-\sum_{i\in F\cup G}|a(\cdot)_i|\bigg|=\bigg||a|-\Big|\sum_{i\in F\cup G}a(\cdot)_i\cdot\underline{\sf e}_i\Big|\bigg|
\leq\bigg|a-\sum_{i\in F\cup G}a(\cdot)_i\cdot\underline{\sf e}_i\bigg|
\]
for every \(F\in\mathscr P_f(I)\) and \(G\in\mathscr P_f(I\setminus F)\). All in all, the proof of the statement is achieved.
\end{proof}

Finally, the Banach \(L^0(\XX)\)-modules \(L^0(\XX;\ell_1(I))\) and \(\ell_1(I,L^0(\XX))\) can be canonically identified:
\begin{corollary}\label{cor:two_ell1_L0}
Let \(\XX\) be a \(\sigma\)-finite measure space and \(I\neq\varnothing\) an index family. Let us define
\begin{equation}\label{eq:two_ell1_L0}
\phi(a)\coloneqq\big(a(\cdot)_i\big)_{i\in I}\in\ell_1(I,L^0(\XX))\quad\text{ for every }a\in L^0(\XX;\ell_1(I)).
\end{equation}
Then the operator \(\phi\colon L^0(\XX;\ell_1(I))\to\ell_1(I,L^0(\XX))\) is an isomorphism of Banach \(L^0(\XX)\)-modules.
\end{corollary}
\begin{proof}
The fact that \(\phi\) is a homomorphism of Banach \(L^0(\XX)\)-modules satisfying \(|\phi(a)|_1=|a|\) for every \(a\in L^0(\XX;\ell_1(I))\) follows from Lemma \ref{lem:basic_L0_l1(I)}.
Therefore, it remains to check only that \(\phi\) is surjective. To this aim, fix \(f=(f_i)_{i\in I}\in\ell_1(I,L^0(\XX))\). We know that \(J\coloneqq\big\{i\in I\,:\,f_i\neq 0\big\}\)
is at most countable. Take a measurable representative \(\bar f_i\colon\X\to\R\) of \(f_i\) for every \(i\in I\), with \(\bar f_i\equiv 0\) for every \(i\in I\setminus J\).
Since \(\sum_{i\in I}|f_i|=\sum_{i\in J}|f_i|\in L^0(\XX)\), we can also assume (up to modifying the functions \(\bar f_i\) for \(i\in J\) on a null set) that \(\big\{|\bar f_i(x)|\big\}_{i\in I}\subseteq\R\)
is summable for every \(x\in\X\). Then the mapping \(\X\ni x\mapsto\bar a(x)\coloneqq\big(\bar f_i(x)\big)_{i\in I}\in\ell_1(I)\)
is well-defined, is measurable, and takes values into a separable subset of \(\ell_1(I)\) (namely, the closure of the vector subspace
generated by \(\{{\sf e}_i\}_{i\in J}\)). Letting \(a\in L^0(\XX;\ell_1(I))\) be the equivalence class of \(\bar a\colon\X\to\ell_1(I)\),
we have that \(\phi(a)=f\) by construction. This proves the surjectivity of \(\phi\), thus accordingly the statement is achieved.
\end{proof}
\subsection{Some notions of continuous module-valued maps}\label{s:cont_mod-valued}
When dealing with injective tensor products of Banach spaces, a special role is played by the Banach space \({\rm C}(K)\),
where \(K\) is a compact, Hausdorff topological space; cf.\ with the first paragraph of Section \ref{s:rel_with_ord-cont}.
It seems that in the more general setting of Banach \(L^0\)-modules there is no `canonical' counterpart of \({\rm C}(K)\).
Rather, we will propose two generalisations of \({\rm C}(K)\) in Definitions \ref{def:unif_ord-cont} and
\ref{def:ptwse_cont}, respectively.
\medskip

Let \((\Omega,\Phi)\) be a uniform space (see \cite{Bourbaki98}). Given an entourage \(\mathcal U\in\Phi\) and any
\(p\in\Omega\), we define
\[
\mathcal U[p]\coloneqq\big\{q\in\Omega\;\big|\;(p,q)\in\mathcal U\big\}.
\]
Recall that the uniform structure \(\Phi\) induces a topology \(\tau_\Phi\) on \(\Omega\), which is defined as follows:
\[
\tau_\Phi\coloneqq\big\{U\subseteq\Omega\;\big|\;\forall p\in U\;\exists\,\mathcal U\in\Phi:\;\mathcal U[p]\subseteq U\big\}.
\]
We then regard every uniform space \((\Omega,\Phi)\) as a topological space, endowed with \(\tau_\Phi\).
\begin{definition}[Uniform order-continuity]\label{def:unif_ord-cont}
Let \((\Omega,\Phi)\) be a uniform space, \(\XX\) a \(\sigma\)-finite measure space, and \(\mathscr M\) a Banach
\(L^0(\XX)\)-module. Then we say that a map \(v\colon\Omega\to\mathscr M\) is \textbf{order-bounded} if
\begin{equation}\label{eq:def_ptwse_norm_UC}
|v|\coloneqq\bigvee_{p\in\Omega}|v(p)|\in L^0(\XX)^+.
\end{equation}
Moreover, we say that \(v\colon\Omega\to\mathscr M\) is \textbf{uniformly order-continuous} provided
\[
\bigwedge_{\mathcal U\in\Phi}{\rm Var}(v;\mathcal U)=0,\quad\text{ where we define }
{\rm Var}(v;\mathcal U)\coloneqq\bigvee_{(p,q)\in\mathcal U}|v(p)-v(q)|.
\]
We denote by \({\rm UC}_{\rm ord}(\Omega;\mathscr M)\) the space of all order-bounded, uniformly order-continuous maps.
\end{definition}

Given any \(v,w\in{\rm UC}_{\rm ord}(\Omega;\mathscr M)\) and \(f\in L^0(\XX)\), we define \(v+w\colon\Omega\to\mathscr M\)
and \(f\cdot v\colon\Omega\to\mathscr M\) as
\[\begin{split}
(v+w)(p)\coloneqq v(p)+w(p)&\quad\text{ for every }p\in\Omega,\\
(f\cdot v)(p)\coloneqq f\cdot v(p)&\quad\text{ for every }p\in\Omega,
\end{split}\]
respectively. It can be readily checked that \(v+w,f\cdot v\in{\rm UC}_{\rm ord}(\Omega;\mathscr M)\), that
\(\big({\rm UC}_{\rm ord}(\Omega;\mathscr M),+,\cdot\big)\) is a module over \(L^0(\XX)\), and that the map
\(|\cdot|\colon{\rm UC}_{\rm ord}(\Omega;\mathscr M)\to L^0(\XX)^+\) defined in \eqref{eq:def_ptwse_norm_UC} is a pointwise
norm on \({\rm UC}_{\rm ord}(\Omega;\mathscr M)\), so that \(\big({\rm UC}_{\rm ord}(\Omega;\mathscr M),|\cdot|\big)\)
is a normed \(L^0(\XX)\)-module. Moreover:
\begin{lemma}\label{lem:UC_complete}
Let \((\Omega,\Phi)\) be a uniform space. Let \(\XX\) be a \(\sigma\)-finite measure space and
\(\mathscr M\) a Banach \(L^0(\XX)\)-module. Let \(|\cdot|\) be as in \eqref{eq:def_ptwse_norm_UC}.
Then \(\big({\rm UC}_{\rm ord}(\Omega;\mathscr M),|\cdot|\big)\) is a Banach \(L^0(\XX)\)-module.
\end{lemma}
\begin{proof}
In view of the above discussion, it only remains to check that \({\rm UC}_{\rm ord}(\Omega;\mathscr M)\) is complete.
To this aim, let \((v_n)_{n\in\N}\subseteq{\rm UC}_{\rm ord}(\Omega;\mathscr M)\) be a given Cauchy sequence.
 Up to a non-relabelled
subsequence, we can assume that \(\sfd_{L^0(\XX)}(|v_n-v_{n+1}|,0)\leq 2^{-n}\) for all \(n\in\N\). For any \(p\in\Omega\) we can
estimate \(\sfd_{\mathscr M}(v_n(p),v_{n+1}(p))\leq\sfd_{L^0(\XX)}(|v_n-v_{n+1}|,0)\leq 2^{-n}\). It follows that
\((v_n(p))_{n\in\N}\subseteq\mathscr M\) is a Cauchy sequence, so that the limit \(v(p)\coloneqq\lim_n v_n(p)\in\mathscr M\)
exists. To prove that \(v\colon\Omega\to\mathscr M\) is order-bounded, notice that \(\big||v_n|-|v_{n+1}|\big|\leq|v_n-v_{n+1}|\) implies
\(\sfd_{L^0(\XX)}(|v_n|,|v_{n+1}|)\leq 2^{-n}\), so that the sequence \((|v_n|)_{n\in\N}\subseteq L^0(\XX)\) is Cauchy.
Define \(g\coloneqq\lim_n|v_n|\in L^0(\XX)\). Given \(p\in\Omega\), we can extract a subsequence \((n_i)_{i\in\N}\)
such that \(|v_{n_i}(p)|\to|v(p)|\) and \(|v_{n_i}|\to g\) \(\mm\)-a.e.\ as \(i\to\infty\). Hence,
\[
|v(p)|(x)=\lim_{i\to\infty}|v_{n_i}(p)|(x)\leq\lim_{i\to\infty}|v_{n_i}|(x)=g(x)\quad\text{ for }\mm\text{-a.e.\ }x\in\X,
\]
which implies \(|v|=\bigvee_{p\in\Omega}|v(p)|\leq g\). We pass to the verification of the uniform order-continuity
of \(v\). For any \(n\in\N\) we can find a sequence of entourages \((\mathcal U^n_i)_{i\in\N}\subseteq\Phi\) with
\(\bigwedge_{i\in\N}{\rm Var}(v_n;\mathcal U^n_i)=0\). With no loss of generality, we can also require that
\(\mathcal U^n_{i+1}\subseteq\mathcal U^n_i\) for every \(i\in\N\), whence it follows that \(\sfd_{L^0(\XX)}\big({\rm Var}(v_n;\mathcal U^n_i),0\big)\to 0\)
as \(i\to\infty\). Define \(h_n\coloneqq\sum_{k=n}^\infty|v_k-v_{k+1}|\wedge 1\) for all \(n\in\N\). Then
\[
\int h_n\,\d\tilde\mm=\sum_{k=n}^\infty\int|v_k-v_{k+1}|\wedge 1\,\d\tilde\mm
=\sum_{k=n}^\infty\sfd_{L^0(\XX)}(|v_k-v_{k+1}|,0)\leq\sum_{k=n}^\infty\frac{1}{2^k}=\frac{1}{2^{n-1}}
\]
by the monotone convergence theorem, so that \(h_n\in L^1(\tilde\mm)\) and \(\|h_n\|_{L^1(\tilde\mm)}\leq 2^{-n+1}\). Notice that
\[\begin{split}
|v(p)-v(q)|\wedge 1&\leq|v(p)-v_n(p)|\wedge 1+|v_n(p)-v_n(q)|\wedge 1+|v_n(q)-v(q)|\wedge 1\\
&\leq 2h_n+|v_n(p)-v_n(q)|\wedge 1
\end{split}\]
for every \(p,q\in\Omega\) and \(n\in\N\). Fixing \(i\in\N\) and passing to the supremum over all \((p,q)\in\mathcal U^n_i\), we deduce
that \({\rm Var}(v;\mathcal U^n_i)\wedge 1\leq 2h_n+{\rm Var}(v_n;\mathcal U^n_i)\wedge 1\). Integrating with respect to \(\tilde\mm\), we thus get
\[
\sfd_{L^0(\XX)}\big({\rm Var}(v;\mathcal U^n_i),0\big)\leq 2\|h_n\|_{L^1(\tilde\mm)}+\sfd_{L^0(\XX)}\big({\rm Var}(v_n;\mathcal U^n_i),0\big)
\leq\frac{1}{2^{n-2}}+\sfd_{L^0(\XX)}\big({\rm Var}(v_n;\mathcal U^n_i),0\big).
\]
Given any \(k\in\N\), we first choose \(n_k\in\N\) such that \(2^{-n_k+2}<1/(2k)\), then we choose \(i_k\in\N\) such
that \(\sfd_{L^0(\XX)}\big({\rm Var}(v_{n_k};\mathcal U_k),0\big)<1/(2k)\), where we set \(\mathcal U_k\coloneqq \mathcal U^{n_k}_{i_k}\). Hence,
\(\sfd_{L^0(\XX)}\big({\rm Var}(v;\mathcal U_k),0\big)<1/k\) for every \(k\in\N\), so that \(\varliminf_k{\rm Var}(v;\mathcal U_k)(x)=0\)
for \(\mm\)-a.e.\ \(x\in\X\). In particular, we conclude that
\[
\bigwedge_{\mathcal U\in\Phi}{\rm Var}(v;\mathcal U)\leq\bigwedge_{k\in\N}{\rm Var}(v;\mathcal U_k)\leq\varliminf_{k\to\infty}{\rm Var}(v;\mathcal U_k)=0,
\]
which shows that \(v\colon\Omega\to\mathscr M\) is uniformly order-continuous. All in all, \(v\) belongs to
\({\rm UC}_{\rm ord}(\Omega;\mathscr M)\).

In order to conclude, it remains to check that \(\sfd_{{\rm UC}_{\rm ord}(\Omega;\mathscr M)}(v_n,v)\to 0\)
as \(n\to\infty\). Fix \(p\in\Omega\) and take a subsequence \((n_j)_{j\in\N}\) such that
\(|v_{n_j}(p)-v_n(p)|\to|v(p)-v_n(p)|\) in the \(\mm\)-a.e.\ sense. Then
\(|v(p)-v_n(p)|\wedge 1=\lim_j|v_{n_j}(p)-v_n(p)|\wedge 1\leq\varliminf_j|v_{n_j}-v_n|\wedge 1\leq h_n\) \(\mm\)-a.e., whence it follows
that \(|v-v_n|\wedge 1\leq h_n\). Therefore, \(\lim_n\sfd_{{\rm UC}_{\rm ord}(\Omega;\mathscr M)}(v,v_n)\leq\lim_n\|h_n\|_{L^1(\tilde\mm)}=0\), as desired.
\end{proof}
\begin{remark}\label{rmk:evaluations_dual_UC}{\rm
Given \(p\in\Omega\), we define the \textbf{evaluation functional}
\(\delta^{\mathscr M}_p\colon{\rm UC}_{\rm ord}(\Omega;\mathscr M)\to\mathscr M\) as
\[
\delta^{\mathscr M}_p(v)\coloneqq v(p)\quad\text{ for every }v\in{\rm UC}_{\rm ord}(\Omega;\mathscr M).
\]
Observe that \(\delta^{\mathscr M}_p\in\textsc{Hom}\big({\rm UC}_{\rm ord}(\Omega;\mathscr M);\mathscr M\big)\)
and \(|\delta^{\mathscr M}_p|\leq 1\). In particular, \(\delta_p\coloneqq\delta^{L^0(\XX)}_p\) satisfies
\[
\delta_p\in{\rm UC}_{\rm ord}(\Omega;L^0(\XX))^*\quad\text{ and }\quad|\delta_p|\leq 1.
\]
Furthermore, \(\{\delta_p\,:\,p\in\Omega\}\) is a norming subset of \({\rm UC}_{\rm ord}(\Omega;L^0(\XX))^*\).
Indeed, thanks to \eqref{eq:def_ptwse_norm_UC} we have that \(|f|=\bigvee_{p\in\Omega}|f(p)|=\bigvee_{p\in\Omega}|\delta_p(f)|\)
holds for every \(f\in{\rm UC}_{\rm ord}(\Omega;L^0(\XX))\).
\fr}\end{remark}
\begin{definition}[Pointwise bounded continuous maps]\label{def:ptwse_cont}
Let \((\Omega,\tau)\) be a topological space. Let \(\XX\) be a \(\sigma\)-finite measure space and \(\mathscr M\)
a Banach \(L^0(\XX)\)-module. Then we define \({\rm C}_{\rm pb}(\Omega;\mathscr M)\) as
\[
{\rm C}_{\rm pb}(\Omega;\mathscr M)\coloneqq\big\{v\colon\Omega\to\mathscr M\;\big|\;v\text{ is continuous and order-bounded}\big\}.
\]
We say that \({\rm C}_{\rm pb}(\Omega;\mathscr M)\) is the space of
\textbf{pointwise bounded continuous maps} from \(\Omega\) to \(\mathscr M\).
\end{definition}

The space \({\rm C}_{\rm pb}(\Omega;\mathscr M)\) is a Banach \(L^0(\XX)\)-module if endowed with the pointwise
norm in \eqref{eq:def_ptwse_norm_UC}. This claim can be proved by repeating almost verbatim the arguments for
Lemma \ref{lem:UC_complete}, the main difference being in the verification of the completeness, where
one can use the following remark:
\begin{remark}{\rm
Take \((v_n)_{n\in\N}\subseteq{\rm C}_{\rm pb}(\Omega;\mathscr M)\) and an order-bounded map
\(v\colon\Omega\to\mathscr M\) such that
\[
\delta_n\coloneqq\sup_{p\in\Omega}\sfd_{\mathscr M}(v_n(p),v(p))\to 0\quad\text{ as }n\to\infty.
\]
Then \(v\in{\rm C}_{\rm pb}(\Omega;\mathscr M)\). Indeed, given any \(p\in\Omega\) and \(\varepsilon>0\),
we can fix \(n_0\in\N\) such that \(\delta_{n_0}<\varepsilon/4\) and choose a neighbourhood \(U\) of
\(p\) such that \(\sfd_{\mathscr M}(v_{n_0}(q),v_{n_0}(p))<\varepsilon/2\) for every \(q\in U\). Then
\[
\sfd_{\mathscr M}(v(q),v(p))\leq\sfd_{\mathscr M}(v(q),v_{n_0}(q))+\sfd_{\mathscr M}(v_{n_0}(q),v_{n_0}(p))+\sfd_{\mathscr M}(v_{n_0}(p),v(p))<2\delta_{n_0}+\frac{\varepsilon}{2}<\varepsilon
\]
for every \(q\in U\), which implies that \(v\) is continuous at each point \(p\in\Omega\), as we claimed.
\fr}\end{remark}

We also point out that if \((\Omega,\Phi)\) is a uniform space, then we have that
\begin{equation}\label{eq:UC_in_C}
{\rm UC}_{\rm ord}(\Omega;\mathscr M)\quad\text{ is a Banach }L^0(\XX)\text{-submodule of }{\rm C}_{\rm pb}(\Omega;\mathscr M).
\end{equation}
Indeed, if \(v\in{\rm UC}_{\rm ord}(\Omega;\mathscr M)\), \(p\in\Omega\), and \(\varepsilon>0\) are given, then
we can find an entourage \(\mathcal U\in\Phi\) such that \(\sfd_{L^0(\XX)}\big({\rm Var}(v;\mathcal U),0\big)<\varepsilon\).
Hence, for every point \(q\) in the open set \(\mathcal U[p]\) we have that
\[
\sfd_{\mathscr M}(v(q),v(p))=\sfd_{L^0(\XX)}\big(|v(q)-v(p)|,0\big)\leq\sfd_{L^0(\XX)}\big({\rm Var}(v;\mathcal U),0\big)<\varepsilon.
\]
\begin{remark}{\rm
If \((K,\Phi)\) is compact and \(\B\) is Banach (so that \(\B\) is a Banach \(L^0(\XX_{\sf o})\)-module,
where \(\XX_{\sf o}\) is the one-point probability space), then \({\rm UC}_{\rm ord}(K;\B)={\rm C}_{\rm pb}(K;\B)={\rm C}(K;\B)\).
\fr}\end{remark}
\subsection{Algebraic tensor products of Banach \texorpdfstring{\(L^0\)}{L0}-modules}
In order to define tensor products of Banach \(L^0\)-modules, the following criterion to
detect null tensors will play a fundamental role:
\begin{lemma}[Null tensors in Banach \(L^0\)-modules]\label{lem:crit_null_tensor}
Let \(\XX\) be a \(\sigma\)-finite measure space. Let \(\mathscr M\), \(\mathscr N\) be Banach \(L^0(\XX)\)-modules.
Fix any \(\alpha=\sum_{i=1}^n v_i\otimes w_i\in\mathscr M\otimes\mathscr N\). Then \(\alpha=0\) if and only if
\begin{equation}\label{eq:null_tensor_cl}
\sum_{i=1}^n\omega(v_i)\eta(w_i)=0\quad\text{ for every }\omega\in\mathscr M^*\text{ and }\eta\in\mathscr N^*.
\end{equation}
\end{lemma}
\begin{proof}
Assume \(\alpha=0\). For any \(\omega\in\mathscr M^*\) and \(\eta\in\mathscr N^*\),
the map \(b_{\omega,\eta}\colon\mathscr M\times\mathscr N\to L^0(\XX)\), which we define as
\(b_{\omega,\eta}(v,w)\coloneqq\omega(v)\eta(w)\) for every \((v,w)\in\mathscr M\times\mathscr N\), is \(L^0(\XX)\)-bilinear.
Therefore, we deduce from \eqref{eq:gen_criter_null_tensor} that \(\sum_{i=1}^n\omega(v_i)\eta(w_i)=\sum_{i=1}^n b_{\omega,\eta}(v_i,w_i)=0\),
which proves that \eqref{eq:null_tensor_cl} holds.

Conversely, assume \eqref{eq:null_tensor_cl} holds. Fix an arbitrary \(L^0(\XX)\)-bilinear map \(b\colon\mathscr M\times\mathscr N\to Q\), for some \(L^0(\XX)\)-module \(Q\).
Denote by \(\mathscr V\) (resp.\ by \(\mathscr W\)) the \(L^0(\XX)\)-submodule of \(\mathscr M\) (resp.\ of \(\mathscr N\)) that is generated by \(v_1,\ldots,v_n\)
(resp.\ by \(w_1,\ldots,w_n\)). Given that the modules \(\mathscr V\) and \(\mathscr W\) are finitely-generated, they are Banach \(L^0(\XX)\)-modules. Now let \(D^{\mathscr V}_0,\ldots,D^{\mathscr V}_{\bar m}\)
and \(D^{\mathscr W}_0,\ldots,D^{\mathscr W}_{\bar q}\) be the dimensional decompositions of \(\mathscr V\) and \(\mathscr W\), respectively. To prove that \(\sum_{i=1}^n b(v_i,w_i)=0\)
amounts to showing that \(\1_{D_{m,q}}\cdot\sum_{i=1}^n b(v_i,w_i)=0\) holds for every \(m=1,\ldots,\bar m\) and \(q=1,\ldots,\bar q\), where we set \(D_{m,q}\coloneqq D^{\mathscr V}_m\cap D^{\mathscr W}_q\).
To this aim, fix a local basis \(x_1,\ldots,x_m\) of \(\mathscr V\) on \(D_{m,q}\) and a local basis \(y_1,\ldots,y_q\) of \(\mathscr W\) on \(D_{m,q}\). Given any \(v\in\mathscr V|_{D_{m,q}}\), we can find \(\tilde\omega_1(v),\ldots,\tilde\omega_m(v)\in L^0(\XX)|_{D_{m,q}}\)
in a unique way so that \(v=\sum_{j=1}^m\tilde\omega_j(v)\cdot x_j\). Moreover, each mapping \(\tilde\omega_j\colon\mathscr V|_{D_{m,q}}\to L^0(\XX)\)
is \(L^0(\XX)\)-linear, thus it is also continuous thanks to Lemma \ref{lem:fin-gen_cont}. An application of the
Hahn--Banach theorem for normed \(L^0\)-modules ensures the existence of some elements \(\omega_1,\ldots,\omega_m\in\mathscr M^*\)
such that \(\omega_j|_{\mathscr V|_{D_{m,q}}}=\tilde\omega_j\) for every \(j=1,\ldots,m\). Similarly, we can find
elements \(\eta_1,\ldots,\eta_q\in\mathscr N^*\) such that \(w=\sum_{k=1}^q\eta_k(w)\cdot y_k\) for every \(w\in\mathscr W|_{D_{m,q}}\).
Therefore, the \(L^0(\XX)\)-bilinearity of \(b\) yields
\[\begin{split}
\1_{D_{m,q}}\cdot\sum_{i=1}^n b(v_i,w_i)&=\sum_{i=1}^n b\big(\1_{D_{m,q}}\cdot v_i,\1_{D_{m,q}}\cdot w_i\big)
=\sum_{j=1}^m\sum_{k=1}^q\bigg(\sum_{i=1}^n\omega_j(v_i)\eta_k(w_i)\bigg)\cdot b(x_j,y_k)\overset{\eqref{eq:null_tensor_cl}}=0.
\end{split}\]
This implies that \(\sum_{i=1}^n b(v_i,w_i)=0\), whence it follows that \(\alpha=\sum_{i=1}^n v_i\otimes w_i=0\) by \eqref{eq:gen_criter_null_tensor}.
\end{proof}
\begin{remark}\label{rmk:alg_tens_diff}{\rm
We stress that Lemma \ref{lem:crit_null_tensor} shows that, in the case of Banach \(L^0(\XX)\)-modules, a null
tensor can be detected by checking only against (a class of) \(L^0(\XX)\)-bilinear maps taking values into the ring \(L^0(\XX)\).
It is not clear whether this happens for arbitrary \(L^0(\XX)\)-modules that are not equipped with a
pointwise norm; cf.\ with the discussion after \eqref{eq:gen_criter_null_tensor}. In other words, the proof of
Lemma \ref{lem:crit_null_tensor} is heavily relying on the fact that we are considering Banach \(L^0(\XX)\)-modules.
\fr}\end{remark}
\begin{corollary}\label{cor:null_tensor_conseq}
Let \(\XX\) be a \(\sigma\)-finite measure space. Let \(\mathscr M\) and \(\mathscr N\) be Banach \(L^0(\XX)\)-modules. Fix
any tensor \(\alpha=\sum_{i=1}^n v_i\otimes w_i\in\mathscr M\otimes\mathscr N\). Then the following conditions are equivalent:
\begin{itemize}
\item[\(\rm i)\)] \(\alpha=0\).
\item[\(\rm ii)\)] \(\sum_{i=1}^n\omega(v_i)\cdot w_i=0\) for every \(\omega\in\mathscr M^*\).
\item[\(\rm iii)\)] \(\sum_{i=1}^n\eta(w_i)\cdot v_i=0\) for every \(\eta\in\mathscr N^*\).
\end{itemize}
\end{corollary}
\begin{proof}
We prove only the equivalence between i) and ii); the proof of the equivalence between i) and iii) is very similar.
Assuming ii), we deduce that \(\sum_{i=1}^n\omega(v_i)\eta(w_i)=\eta\big(\sum_{i=1}^n\omega(v_i)\cdot w_i\big)=0\) for every \(\omega\in\mathscr M^*\)
and \(\eta\in\mathscr N^*\), so that \(\alpha=0\) by Lemma \ref{lem:crit_null_tensor}. Conversely, if \(\alpha=0\), then the same computation as above
shows that \(\eta\big(\sum_{i=1}^n\omega(v_i)\cdot w_i\big)=0\) for every \(\omega\in\mathscr M^*\) and \(\eta\in\mathscr N^*\), so that
\(\sum_{i=1}^n\omega(v_i)\cdot w_i=0\) for every \(\omega\in\mathscr M^*\) by the Hahn--Banach theorem, which gives ii).
\end{proof}
\section{Projective tensor products of Banach \texorpdfstring{\(L^0\)}{L0}-modules}\label{s:proj_tens}
\subsection{Definition and main properties}
We begin by introducing the projective pointwise norm:
\begin{theorem}
Let \(\XX\) be a \(\sigma\)-finite measure space. Let \(\mathscr M\), \(\mathscr N\) be Banach \(L^0(\XX)\)-modules. Define
\begin{equation}\label{eq:def_proj_ptwse_norm}
|\alpha|_\pi\coloneqq\bigwedge\bigg\{\sum_{i=1}^n|v_i||w_i|\;\bigg|\;n\in\N,\,(v_i)_{i=1}^n\subseteq\mathscr M,
\,(w_i)_{i=1}^n\subseteq\mathscr N,\,\alpha=\sum_{i=1}^n v_i\otimes w_i\bigg\}\in L^0(\XX)^+
\end{equation}
for every \(\alpha\in\mathscr M\otimes\mathscr N\). Then \(|\cdot|_\pi\colon\mathscr M\otimes\mathscr N\to L^0(\XX)^+\)
is a pointwise norm on \(\mathscr M\otimes\mathscr N\). Moreover,
\begin{equation}\label{eq:proj_norm_elem_tensor}
|v\otimes w|_\pi=|v||w|\quad\text{ for every }v\in\mathscr M\text{ and }w\in\mathscr N.
\end{equation}
\end{theorem}
\begin{proof}
To prove that \(|\cdot|_\pi\) is a pointwise norm on \(\mathscr M\otimes\mathscr N\) amounts to showing that:
\begin{itemize}
\item[\(\rm i)\)] If \(\alpha\in\mathscr M\otimes\mathscr N\) satisfies \(|\alpha|_\pi=0\), then \(\alpha=0\).
\item[\(\rm ii)\)] \(|\alpha+\beta|_\pi\leq|\alpha|_\pi+|\beta|_\pi\) for every \(\alpha,\beta\in\mathscr M\otimes\mathscr N\).
\item[\(\rm iii)\)] \(|f\cdot\alpha|_\pi=|f||\alpha|_\pi\) for every \(f\in L^0(\XX)\) and \(\alpha\in\mathscr M\otimes\mathscr N\).
\end{itemize}
Let us first check the validity of i). Assume \(|\alpha|_\pi=0\). Let \(\omega\in\mathscr M^*\) and \(\eta\in\mathscr N^*\)
be given. Then we define \(\theta_{\omega,\eta}\in L^0(\XX)\) as \(\theta_{\omega,\eta}\coloneqq\sum_{i=1}^n\omega(v_i)\eta(w_i)\)
for any \(v_1,\ldots,v_n\in\mathscr M\) and \(w_1,\ldots,w_n\in\mathscr N\) satisfying \(\alpha=\sum_{i=1}^n v_i\otimes w_i\);
thanks to Lemma \ref{lem:crit_null_tensor}, the function \(\theta_{\omega,\eta}\) is independent of the chosen
representation \(\sum_{i=1}^nv_i\otimes w_i\) of \(\alpha\). Now fix \(\varepsilon>0\). Then there exists a partition
\((E_k)_{k\in\N}\subseteq\Sigma\) of \(\X\) and \(v_1^k,\ldots,v_{n_k}^k\in\mathscr M\),
\(w_1^k,\ldots,w_{n_k}^k\in\mathscr N\) such that \(\alpha=\sum_{i=1}^{n_k}v_i^k\otimes w_i^k\)
and \(\1_{E_k}\sum_{i=1}^{n_k}|v_i^k||w_i^k|\leq\varepsilon\) for every \(k\in\N\). Therefore, can estimate
\[
|\theta_{\omega,\eta}|=\sum_{k\in\N}\1_{E_k}\bigg|\sum_{i=1}^{n_k}\omega(v_i^k)\eta(w_i^k)\bigg|
\leq\sum_{k\in\N}\1_{E_k}\sum_{i=1}^{n_k}|\omega(v_i^k)||\eta(w_i^k)|
\leq|\omega||\eta|\sum_{k\in\N}\sum_{i=1}^{n_k}|v_i^k||w_i^k|\leq\varepsilon|\omega||\eta|.
\]
Thanks to the arbitrariness of \(\varepsilon>0\), we deduce that \(\theta_{\omega,\eta}=0\), so that
\(\alpha=0\) by Lemma \ref{lem:crit_null_tensor}.

In order to prove ii), let us write \(\alpha=\sum_{i=1}^n v_i\otimes w_i\) and \(\beta=\sum_{j=1}^m\tilde v_j\otimes\tilde w_j\).
Then we have that
\[
|\alpha+\beta|_\pi\leq\sum_{i=1}^n|v_i||w_i|+\sum_{j=1}^m|\tilde v_j||\tilde w_j|,
\]
where we used the fact that \(\alpha+\beta=\sum_{i=1}^n v_i\otimes w_i+\sum_{j=1}^m\tilde v_j\otimes\tilde w_j\).
By passing to the infimum over all the possible representations of \(\alpha\) and \(\beta\), we conclude that
\(|\alpha+\beta|_\pi\leq|\alpha|_\pi+|\beta|_\pi\).

We now pass to the verification of iii). If \(\alpha=\sum_{i=1}^n v_i\otimes w_i\), then \(f\cdot\alpha=\sum_{i=1}^n(f\cdot v_i)\otimes w_i\).
It follows that \(|f\cdot\alpha|_\pi\leq\sum_{i=1}^n|f\cdot v_i||w_i|=|f|\sum_{i=1}^n|v_i||w_i|\).
By passing to the infimum over all the representations of \(\alpha\), we obtain that \(|f\cdot\alpha|_\pi\leq|f||\alpha|_\pi\).
Moreover, the same estimates yield
\[
|f||\alpha|_\pi=|f|\bigg|\frac{\1_{\{f\neq 0\}}}{f}\cdot(f\cdot\alpha)\bigg|_\pi
\leq|f|\frac{\1_{\{f\neq 0\}}}{|f|}|f\cdot\alpha|_\pi=\1_{\{f\neq 0\}}|f\cdot\alpha|_\pi\leq|f\cdot\alpha|_\pi.
\]
All in all, we have shown that \(|f\cdot\alpha|_\pi=|f||\alpha|_\pi\).

Finally, let us check that \eqref{eq:proj_norm_elem_tensor} holds. The inequality \(|v\otimes w|_\pi\leq|v||w|\) is
trivially verified. For the converse inequality, choose elements \(\omega\in\mathscr M^*\) and \(\eta\in\mathscr N^*\)
such that \(|\omega|,|\eta|\leq 1\), \(\omega(v)=|v|\), and \(\eta(w)=|w|\). The \(L^0(\XX)\)-linearisation
\(T\) of \(\mathscr M\times\mathscr N\ni(\tilde v,\tilde w)\mapsto\omega(\tilde v)\eta(\tilde w)\in L^0(\XX)\) satisfies
\[
|T(\alpha)|\leq\sum_{i=1}^n|T(v_i\otimes w_i)|=\sum_{i=1}^n|\omega(v_i)||\eta(w_i)|\leq\sum_{i=1}^n|v_i||w_i|
\quad\;\forall\alpha=\sum_{i=1}^n v_i\otimes w_i\in\mathscr M\otimes\mathscr N,
\]
whence it follows that \(|T(\alpha)|\leq|\alpha|_\pi\) for every \(\alpha\in\mathscr M\otimes\mathscr N\).
In particular, we have that
\[
|v||w|=\big|\omega(v)\eta(w)\big|=\big|T(v\otimes w)\big|\leq|v\otimes w|_\pi.
\]
All in all, we have shown that \(|v\otimes w|_\pi=|v||w|\), thus accordingly \eqref{eq:proj_norm_elem_tensor} is proved.
\end{proof}
\begin{definition}[Projective tensor product]
Let \(\XX\) be a \(\sigma\)-finite measure space. Let \(\mathscr M\), \(\mathscr N\) be Banach \(L^0(\XX)\)-modules. Then we
denote by \(\mathscr M\otimes_\pi\mathscr N\) the normed \(L^0(\XX)\)-module \((\mathscr M\otimes\mathscr N,|\cdot|_\pi)\),
where the pointwise norm \(|\cdot|_\pi\) is defined as in \eqref{eq:def_proj_ptwse_norm}. Moreover, the
\textbf{projective tensor product} of \(\mathscr M\), \(\mathscr N\) is the Banach \(L^0(\XX)\)-module
\(\mathscr M\hat\otimes_\pi\mathscr N\) defined as the \(L^0(\XX)\)-completion of \(\mathscr M\otimes_\pi\mathscr N\).
\end{definition}

Let us now consider the projective tensor product of homomorphisms of Banach \(L^0(\XX)\)-modules:
\begin{proposition}[Projective tensor products of homomorphisms]\label{prop:prod_norm_tens_hom}
Let \(\XX\) be a \(\sigma\)-finite measure space. Let \(T\colon\mathscr M\to\tilde{\mathscr M}\) and \(S\colon\mathscr N\to\tilde{\mathscr N}\) be homomorphisms
of Banach \(L^0(\XX)\)-modules. Then there exists a unique homomorphism of Banach \(L^0(\XX)\)-modules
\(T\otimes_\pi S\colon\mathscr M\hat\otimes_\pi\mathscr N\to\tilde{\mathscr M}\hat\otimes_\pi\tilde{\mathscr N}\) with
\[
(T\otimes_\pi S)(v\otimes w)=T(v)\otimes S(w)\quad\text{ for every }v\in\mathscr M\text{ and }w\in\mathscr N.
\]
Moreover, it holds that \(|T\otimes_\pi S|=|T||S|\).
\end{proposition}
\begin{proof}
By virtue of Lemma \ref{lem:alg_tensor_hom}, there is a unique \(L^0(\XX)\)-linear map \(T\otimes S\colon\mathscr M\otimes\mathscr N\to\tilde{\mathscr M}\otimes\tilde{\mathscr N}\)
with \((T\otimes S)(v\otimes w)=T(v)\otimes S(w)\) for all \((v,w)\in\mathscr M\times\mathscr N\). If \(\alpha=\sum_{i=1}^n v_i\otimes w_i\in\mathscr M\otimes\mathscr N\), then
\[
\big|(T\otimes S)(\alpha)\big|_\pi=\bigg|\sum_{i=1}^n T(v_i)\otimes S(w_i)\bigg|_\pi\leq\sum_{i=1}^n|T(v_i)||S(w_i)|\leq|T||S|\sum_{i=1}^n|v_i||w_i|.
\]
By taking the infimum over all representations of \(\alpha\), we obtain that \(\big|(T\otimes S)(\alpha)\big|_\pi\leq|T||S||\alpha|_\pi\). It follows that the operator \(T\otimes S\)
can be uniquely extended to a homomorphism of Banach \(L^0(\XX)\)-modules \(T\otimes_\pi S\colon\mathscr M\hat\otimes_\pi\mathscr N\to\tilde{\mathscr M}\hat\otimes_\pi\tilde{\mathscr N}\)
satisfying \(|T\otimes_\pi S|\leq|T||S|\). Finally, we have that
\[\begin{split}
|T||S|&=\bigvee_{v\in\mathbb S_{\mathscr M}}\bigvee_{w\in\mathbb S_{\mathscr N}}|T(v)||S(w)|=\bigvee_{v\in\mathbb S_{\mathscr M}}\bigvee_{w\in\mathbb S_{\mathscr N}}\big|T(v)\otimes S(w)\big|_\pi
=\bigvee_{v\in\mathbb S_{\mathscr M}}\bigvee_{w\in\mathbb S_{\mathscr N}}\big|(T\otimes_\pi S)(v\otimes w)\big|_\pi\\
&\leq|T\otimes_\pi S|\bigvee_{v\in\mathbb S_{\mathscr M}}\bigvee_{w\in\mathbb S_{\mathscr N}}|v\otimes w|_\pi=
|T\otimes_\pi S|\bigvee_{v\in\mathbb S_{\mathscr M}}\bigvee_{w\in\mathbb S_{\mathscr N}}|v||w|\leq|T\otimes_\pi S|.
\end{split}\]
Consequently, the identity \(|T\otimes_\pi S|=|T||S|\) is proved.
\end{proof}

One can easily check that \(L^0(\XX)\hat\otimes_\pi L^0(\XX)=L^0(\XX)\otimes_\pi L^0(\XX)\cong L^0(\XX)\) as Banach \(L^0(\XX)\)-modules via
\(L^0(\XX)\otimes_\pi L^0(\XX)\ni\sum_{i=1}^n f_i\otimes g_i\mapsto\sum_{i=1}^n f_i g_i\in L^0(\XX)\). In particular, up to this identification,
\[
\omega\otimes_\pi\eta\in(\mathscr M\hat\otimes_\pi\mathscr N)^*,\quad|\omega\otimes_\pi\eta|=|\omega||\eta|\quad\text{ for every }\omega\in\mathscr M^*\text{ and }\eta\in\mathscr N^*.
\]
\begin{lemma}\label{lem:proj_hom_quotient}
Let \(\XX\) be a \(\sigma\)-finite measure space. Let \(T\colon\mathscr M\to\tilde{\mathscr M}\) and \(S\colon\mathscr N\to\tilde{\mathscr N}\)
be quotient operators of Banach \(L^0(\XX)\)-modules. Then \(T\otimes_\pi S\colon\mathscr M\hat\otimes_\pi\mathscr N\to\tilde{\mathscr M}\hat\otimes_\pi\tilde{\mathscr N}\)
is a quotient operator.
\end{lemma}
\begin{proof}
In view of Remark \ref{rmk:ext_quotient_oper}, it is sufficient to prove that
\(T\otimes S\colon\mathscr M\otimes_\pi\mathscr N\to\tilde{\mathscr M}\otimes_\pi\tilde{\mathscr N}\) is a quotient operator.
Given any \(\beta=\sum_{i=1}^n\tilde v_i\otimes\tilde w_i\in\tilde{\mathscr M}\otimes\tilde{\mathscr N}\), we can
exploit the surjectivity of \(T\) and \(S\) to find \((v_i)_{i=1}^n\subseteq\mathscr M\) and \((w_i)_{i=1}^n\subseteq\mathscr N\)
such that \(\tilde v_i=T(v_i)\) and \(\tilde w_i=S(w_i)\) for all \(i=1,\ldots,n\), whence it follows that
\(\beta=\sum_{i=1}^n T(v_i)\otimes S(w_i)=(T\otimes S)\big(\sum_{i=1}^n v_i\otimes w_i\big)\). This shows that \(T\otimes S\)
is a surjective operator. Moreover, for any tensor \(\beta\in\tilde{\mathscr M}\otimes_\pi\tilde{\mathscr N}\) we can estimate
\[
|\beta|_\pi=\bigwedge_{\alpha\in(T\otimes S)^{-1}(\beta)}|(T\otimes S)(\alpha)|_\pi
\leq|T||S|\bigwedge_{\alpha\in(T\otimes S)^{-1}(\beta)}|\alpha|_\pi
\overset{\eqref{eq:norm_quotient_oper}}\leq\bigwedge_{\alpha\in(T\otimes S)^{-1}(\beta)}|\alpha|_\pi.
\]
In order to prove the converse inequality, fix \(\varepsilon\in(0,1)\). We can thus find a partition \((E_k)_{k\in\N}\subseteq\Sigma\) of \(\X\) and, for any \(k\in\N\),
a number \(n_k\in\N\) and elements \((\tilde v^k_i)_{i=1}^{n_k}\subseteq\tilde{\mathscr M}\), \((\tilde w^k_i)_{i=1}^{n_k}\subseteq\tilde{\mathscr N}\) such that
\[
\beta=\sum_{i=1}^{n_k}\tilde v^k_i\otimes\tilde w^k_i,\qquad
\1_{E_k}\sum_{i=1}^{n_k}|\tilde v^k_i||\tilde w^k_i|\leq\1_{E_k}(|\beta|_\pi+\varepsilon).
\]
Moreover, we can find elements \((v^k_i)_{i=1}^{n_k}\subseteq\mathscr M\) and \((w^k_i)_{i=1}^{n_k}\subseteq\mathscr N\)
with \(T(v^k_i)=\tilde v^k_i\) and \(S(w^k_i)=\tilde w^k_i\) for every \(i=1,\ldots,n_k\) such that
\(|v^k_i|\leq(1+\varepsilon)|\tilde v^k_i|\) and \(|w^k_i|\leq(1+\varepsilon)|\tilde w^k_i|\). Therefore, we have
\[
\1_{E_k}\bigg|\sum_{i=1}^{n_k}v^k_i\otimes w^k_i\bigg|_\pi\leq\1_{E_k}\sum_{i=1}^{n_k}|v^k_i||w^k_i|
\leq(1+\varepsilon)^2\1_{E_k}(|\beta|_\pi+\varepsilon)\leq\1_{E_k}|\beta|_\pi+\1_{E_k}(3|\beta|_\pi+4)\varepsilon.
\]
Since \((T\otimes S)\big(\1_{E_k}\cdot\sum_{i=1}^{n_k}v^k_i\otimes w^k_i\big)=\1_{E_k}\cdot\beta\) for every \(k\in\N\), we deduce that
\[
\bigwedge_{\alpha\in(T\otimes S)^{-1}(\beta)}|\alpha|_\pi\leq\sum_{k\in\N}\1_{E_k}\bigg|\sum_{i=1}^{n_k}v^k_i\otimes w^k_i\bigg|_\pi\leq|\beta|_\pi+(3|\beta|_\pi+4)\varepsilon.
\]
Thanks to the arbitrariness of \(\varepsilon\), we can finally conclude that \(\bigwedge_{\alpha\in(T\otimes S)^{-1}(\beta)}|\alpha|_\pi\leq|\beta|_\pi\).
\end{proof}
\begin{lemma}\label{lem:generators_proj_tensor}
Let \(\XX\) be a \(\sigma\)-finite measure space and \(\mathscr M\), \(\mathscr N\) Banach \(L^0(\XX)\)-modules.
Let \(G\subseteq\mathscr M\) and \(H\subseteq\mathscr N\) be generating subsets. Then the set
\(\{v\otimes w\;\big|\;v\in G,\,w\in H\}\) generates \(\mathscr M\hat\otimes_\pi\mathscr N\).
\end{lemma}
\begin{proof}
As the linear span of the elementary tensors is dense in \(\mathscr M\hat\otimes_\pi\mathscr N\), it suffices
to check that any given \(v\otimes w\) with \(v\in\mathscr M\) and \(w\in\mathscr N\) can be approximated by elements
of the \(L^0(\XX)\)-module generated by \(\big\{\tilde v\otimes\tilde w\;\big|\;\tilde v\in G,\,\tilde w\in H\big\}\).
Since \(G\) and \(H\) generate \(\mathscr M\) and \(\mathscr N\), respectively, we can find some
\((v_n)_{n\in\N}\subseteq\mathscr M\) and \((w_n)_{n\in\N}\subseteq\mathscr N\) that are \(L^0(\XX)\)-linear combinations
of elements of \(G\) and \(H\), respectively, such that \(|v_n-v|\to 0\) and \(|w_n-w|\to 0\) in the \(\mm\)-a.e.\ sense. Then
\[
|v\otimes w-v_n\otimes w_n|_\pi\leq|v\otimes w-v_n\otimes w|_\pi+|v_n\otimes w-v_n\otimes w_n|_\pi\overset{\eqref{eq:proj_norm_elem_tensor}}=
|v-v_n||w|+|v_n||w-w_n|\to 0
\]
in the \(\mm\)-a.e.\ sense. In particular, \(v_n\otimes w_n\to v\otimes w\) in \(\mathscr M\hat\otimes_\pi\mathscr N\).
The statement follows.
\end{proof}

Generalising the fact that \(\ell_1(I)\hat\otimes_\pi\B\cong\ell_1(I,\B)\) for every Banach space \(\B\), we have the following:
\begin{theorem}\label{thm:vv_sequence}
Let \(\XX\) be a \(\sigma\)-finite measure space, \(\mathscr M\) a Banach \(L^0(\XX)\)-module,
and \(I\neq\varnothing\) an index family. Then the unique linear continuous operator
\(\mathfrak i\colon L^0(\XX;\ell_1(I))\hat\otimes_\pi\mathscr M\to\ell_1(I,\mathscr M)\) satisfying
\begin{equation}\label{eq:vv_sequence}
\mathfrak i(a\otimes v)=\big(a(\cdot)_i\cdot v\big)_{i\in I}\quad\text{ for every }a\in L^0(\XX;\ell_1(I))\text{ and }v\in\mathscr M
\end{equation}
is an isomorphism of Banach \(L^0(\XX)\)-modules.
\end{theorem}
\begin{proof}
First, notice that \(L^0(\XX;\ell_1(I))\times\mathscr M\ni(a,v)\mapsto\big(a(\cdot)_i\cdot v\big)_{i\in I}\in\ell_1(I,\mathscr M)\)
is well-defined and \(L^0(\XX)\)-bilinear, thus we can consider its \(L^0(\XX)\)-linearisation
\(\mathfrak i\colon L^0(\XX;\ell_1(I))\otimes\mathscr M\to\ell_1(I,\mathscr M)\), i.e.
\[
\mathfrak i(\alpha)=\bigg(\sum_{j=1}^n a_j(\cdot)_i\cdot v_j\bigg)_{i\in I}
\quad\text{ for every }\alpha=\sum_{j=1}^n a_j\otimes v_j\in L^0(\XX;\ell_1(I))\otimes\mathscr M.
\]
Observe that \(\mathfrak i\) is the unique linear operator from \(L^0(\XX;\ell_1(I))\otimes\mathscr M\) to
\(\ell_1(I,\mathscr M)\) satisfying \eqref{eq:vv_sequence}.

On the one hand, given any tensor \(\alpha=\sum_{j=1}^n a_j\otimes v_j\in L^0(\XX;\ell_1(I))\otimes\mathscr M\) we can estimate
\[\begin{split}
|\mathfrak i(\alpha)|_1&=\sum_{i\in I}\Big|\sum_{j=1}^n a_j(\cdot)_i\cdot v_j\Big|\leq\sum_{i\in I}\sum_{j=1}^n|a_j(\cdot)_i||v_j|
=\sum_{j=1}^n\Big(\sum_{i\in I}|a_j(\cdot)_i|\Big)|v_j|\overset{\eqref{eq:basic_L0_l1(I)_cl2}}=\sum_{j=1}^n|a_j||v_j|.
\end{split}\]
By passing to the infimum over all representations of \(\alpha\), we deduce that \(|\mathfrak i(\alpha)|_1\leq|\alpha|_\pi\).
On the other hand, if \(\alpha\) is written as \(\sum_{j=1}^n a_j\otimes v_j\), then we claim that the elements \(w_i\coloneqq\sum_{j=1}^n a_j(\cdot)_i\cdot v_j\in\mathscr M\)
satisfy the following property: the family \(\{\underline{\sf e}_i\otimes w_i\}_{i\in I}\) is summable in \(L^0(\XX;\ell_1(I))\hat\otimes_\pi\mathscr M\) and
\begin{equation}\label{eq:vv_sequence_cl}
\sum_{i\in I}\underline{\sf e}_i\otimes w_i=\alpha.
\end{equation}
In order to prove it, let us first notice that
\begin{equation}\label{eq:vv_sequence_aux}
\underline{\sf e}_i\otimes w_i=\underline{\sf e}_i\otimes\Big(\sum_{j=1}^n a_j(\cdot)_i\cdot v_j\Big)=\sum_{j=1}^n a_j(\cdot)_i\cdot(\underline{\sf e}_i\otimes v_j)
=\sum_{j=1}^n\big(a_j(\cdot)_i\cdot\underline{\sf e}_i\big)\otimes v_j.
\end{equation}
Since \(L^0(\XX;\ell_1(I))\ni s\mapsto s\otimes v_j\in L^0(\XX;\ell_1(I))\hat\otimes_\pi\mathscr M\) is a homomorphism of Banach \(L^0(\XX)\)-modules,
\[
\alpha=\sum_{j=1}^n a_j\otimes v_j\overset{\eqref{eq:basic_L0_l1(I)_cl1}}=\sum_{j=1}^n\Big(\sum_{i\in I}a_j(\cdot)_i\cdot\underline{\sf e}_i\Big)\otimes v_j
\overset{\eqref{eq:summable_comp_dual}}=\sum_{i\in I}\bigg(\sum_{j=1}^n\big(a_j(\cdot)_i\cdot\underline{\sf e}_i\big)\otimes v_j\bigg)\overset{\eqref{eq:vv_sequence_aux}}=
\sum_{i\in I}\underline{\sf e}_i\otimes w_i.
\]
This proves the validity of the claim \eqref{eq:vv_sequence_cl}. By taking Remark \ref{rmk:summable_ptwse_norms} into account, we conclude that
\[
|\alpha|_\pi=\Big|\sum_{i\in I}\underline{\sf e}_i\otimes w_i\Big|_\pi\leq\sum_{i\in I}|\underline{\sf e}_i\otimes w_i|_\pi\overset{\eqref{eq:proj_norm_elem_tensor}}=
\sum_{i\in I}|\underline{\sf e}_i||w_i|=\sum_{i\in I}|w_i|=\big|(w_i)_{i\in I}\big|_1=|\mathfrak i(\alpha)|_1.
\]
All in all, we have shown that \(|\mathfrak i(\alpha)|_1=|\alpha|_\pi\) for every \(\alpha\in L^0(\XX;\ell_1(I))\otimes\mathscr M\). Therefore, the map \(\mathfrak i\) can be uniquely
extended to a homomorphism of Banach \(L^0(\XX)\)-modules from \(L^0(\XX;\ell_1(I))\hat\otimes_\pi\mathscr M\) to \(\ell_1(I,\mathscr M)\), which we still denote by \(\mathfrak i\).
Notice that the extension \(\mathfrak i\) preserves the pointwise norm.

In order to conclude, it only remains to check that \(\mathfrak i\colon L^0(\XX;\ell_1(I))\hat\otimes_\pi\mathscr M\to\ell_1(I,\mathscr M)\)
is surjective. Let \(v=(v_i)_{i\in I}\in\ell_1(I,\mathscr M)\) be fixed. Thanks to Proposition \ref{prop:Cauchy_sum_criterion},
it follows from the estimates
\[
\bigwedge_{F\in\mathscr P_f(I)}\,\bigvee_{G\in\mathscr P_f(I\setminus F)}\Big|\sum_{i\in G}\underline{\sf e}_i\otimes v_i\Big|
\leq\bigwedge_{F\in\mathscr P_f(I)}\,\bigvee_{G\in\mathscr P_f(I\setminus F)}\,\sum_{i\in G}|v_i|=0
\]
that \(\{\underline{\sf e}_i\otimes v_i\}_{i\in I}\) is summable in \(L^0(\XX;\ell_1(I))\hat\otimes_\pi\mathscr M\). Letting \(\alpha\coloneqq\sum_{i\in I}\underline{\sf e}_i\otimes v_i\), we have that
\[
v\overset{\eqref{eq:sum_elem_ell1_I_M}}=\sum_{i\in I}(\delta_{ij}v_i)_{j\in I}=\sum_{i\in I}(\underline{\sf e}_i(\cdot)_j\cdot v_i)_{j\in I}
\overset{\eqref{eq:vv_sequence}}=\sum_{i\in I}\mathfrak i(\underline{\sf e}_i\otimes v_i)\overset{\eqref{eq:summable_comp_dual}}=
\mathfrak i\Big(\sum_{i\in I}\underline{\sf e}_i\otimes v_i\Big)=\mathfrak i(\alpha),
\]
whence it follows that \(\mathfrak i\) is surjective. Consequently, the proof of the statement is complete.
\end{proof}
\begin{remark}{\rm
Under the assumptions of Theorem \ref{thm:vv_sequence},  for any \(i\in I\) we define the operator \(\iota_i\) as
\[\begin{split}
\iota_i\colon\mathscr M&\longrightarrow L^0(\XX;\ell_1(I))\hat\otimes_\pi\mathscr M\\
v&\longmapsto\underline{\sf e}_i\otimes v.
\end{split}\]
Combining Theorem \ref{thm:vv_sequence} with \cite[Theorem 3.12]{Pas22}, we obtain that \(\big(L^0(\XX;\ell_1(I))\hat\otimes_\pi\mathscr M,\{\iota_i\}_{i\in I}\big)\)
is the coproduct of \(\{\mathscr M_i\}_{i\in I}\), where \(\mathscr M_i\coloneqq\mathscr M\) for every \(i\in I\), in the category \({\bf BanMod}^1_\XX\).
\fr}\end{remark}
\begin{lemma}\label{lem:mod_quotient_ell1}
Let \(\XX\) be a \(\sigma\)-finite measure space. Let \(\mathscr M\) be a Banach \(L^0(\XX)\)-module. Define
\begin{equation}\label{eq:mod_quotient_ell1}
\varphi(f)\coloneqq\sum_{v\in\mathbb S_{\mathscr M}}f_v\cdot v\in\mathscr M\quad\text{ for every }
f=(f_v)_{v\in\mathbb S_{\mathscr M}}\in\ell_1(\mathbb S_{\mathscr M},L^0(\XX)).
\end{equation}
Then \(\varphi\colon\ell_1(\mathbb S_{\mathscr M},L^0(\XX))\to\mathscr M\) is a quotient operator.
In particular, \(\mathscr M\cong\ell_1(\mathbb S_{\mathscr M},L^0(\XX))/{\rm ker}(\varphi)\).
\end{lemma}
\begin{proof}
First of all, by using Proposition \ref{prop:Cauchy_sum_criterion} we obtain that
\[
\bigwedge_{F\in\mathscr P_f(\mathbb S_{\mathscr M})}\,\bigvee_{G\in\mathscr P_f(\mathbb S_{\mathscr M}\setminus F)}\Big|\sum_{v\in G}f_v\cdot v\Big|
\leq\bigwedge_{F\in\mathscr P_f(\mathbb S_{\mathscr M})}\,\bigvee_{G\in\mathscr P_f(\mathbb S_{\mathscr M}\setminus F)}\Big|\sum_{v\in G}f_v\Big|=0
\]
and thus that \((f_v\cdot v)_{v\in\mathbb S_{\mathscr M}}\) is summable in \(\mathscr M\).
Since \(\big|\sum_{v\in\mathbb S_{\mathscr M}}f_v\cdot v\big|\leq\sum_{v\in\mathbb S_{\mathscr M}}|f_v|=|f|_1\),
we have that \(\varphi\) is a well-defined linear operator satisfying \(|\varphi(f)|\leq|f|_1\) for every
\(f\in\ell_1(\mathbb S_{\mathscr M},L^0(\XX))\), thus in particular it is a homomorphism of Banach \(L^0(\XX)\)-modules.
Moreover, if \(w\in\mathscr M\) is given, then
\[
f^w=(f^w_v)_{v\in\mathbb S_{\mathscr M}}\in\ell_1(\mathbb S_{\mathscr M},L^0(\XX)),
\quad f^w_v\coloneqq\left\{\begin{array}{ll}
|w|\\
0
\end{array}\quad\begin{array}{ll}
\text{ if }v={\rm sgn}(w),\\
\text{ otherwise}
\end{array}\right.
\]
satisfies \(\varphi(f^w)=|w|\cdot{\rm sgn}(w)=w\) and \(|\varphi(f^w)|=|w|=|f^w|_1\). Hence, \(\varphi\) is a quotient operator, thus it
induces an isomorphism of Banach \(L^0(\XX)\)-modules between \(\ell_1(\mathbb S_{\mathscr M},L^0(\XX))/{\rm ker}(\varphi)\) and \(\mathscr M\).
\end{proof}

We conclude this section with a useful representation formula for the projective pointwise norm:
\begin{theorem}[Characterisation of the projective pointwise norm]\label{thm:altern_proj_tensor}
Let \(\XX\) be a \(\sigma\)-finite measure space. Let \(\mathscr M\), \(\mathscr N\) be Banach \(L^0(\XX)\)-modules. Then for every \(\alpha\in\mathscr M\hat\otimes_\pi\mathscr N\) it holds that
\begin{equation}\label{eq:altern_proj_tensor}
|\alpha|_\pi=\bigwedge\bigg\{\sum_{n\in\N}|v_n||w_n|\;\bigg|\;(v_n\otimes w_n)_{n\in\N}\in\ell_1(\N,\mathscr M\hat\otimes_\pi\mathscr N),\,\alpha=\sum_{n\in\N}v_n\otimes w_n\bigg\}.
\end{equation}
\end{theorem}
\begin{proof}
For brevity, we denote by \(q(\alpha)\) the right-hand side of \eqref{eq:altern_proj_tensor}. On the one hand, notice that
\[
|\alpha|_\pi=\Big|\sum_{n\in\N}v_n\otimes w_n\Big|_\pi\overset{\eqref{eq:summable_ptwse_norms}}\leq
\sum_{n\in\N}|v_n||w_n|\quad\text{ if }(v_n\otimes w_n)_{n\in\N}\in\ell_1(\N,\mathscr M\hat\otimes_\pi\mathscr N),
\text{ }\alpha=\sum_{n\in\N}v_n\otimes w_n,
\]
whence it follows that \(|\alpha|_\pi\leq q(\alpha)\) for every \(\alpha\in\mathscr M\hat\otimes_\pi\mathscr N\). On the other hand, let us denote by
\[\begin{split}
\varphi&\colon\ell_1(\mathbb S_{\mathscr M},L^0(\XX))\to\mathscr M,\\
\phi&\colon L^0(\XX;\ell_1(\mathbb S_{\mathscr M}))\to\ell_1(\mathbb S_{\mathscr M},L^0(\XX)),\\
\mathfrak i&\colon L^0(\XX;\ell_1(\mathbb S_{\mathscr M}))\hat\otimes_\pi\mathscr N\to\ell_1(\mathbb S_{\mathscr M},\mathscr N)
\end{split}\]
the operators given by Lemma \ref{lem:mod_quotient_ell1}, Corollary \ref{cor:two_ell1_L0}, and Theorem \ref{thm:vv_sequence},
respectively. Recall that \(\varphi\) is a quotient operator, while \(\phi\) and \(\mathfrak i\) are isomorphisms of Banach
\(L^0(\XX)\)-modules. In particular, the mapping \(\tilde\varphi\coloneqq\varphi\circ\phi\colon L^0(\XX;\ell_1(\mathbb S_{\mathscr M}))\to\mathscr M\)
is a quotient operator, so that accordingly
\[
\psi\coloneqq(\tilde\varphi\otimes_\pi{\rm id}_{\mathscr N})\circ\mathfrak i^{-1}\colon
\ell_1(\mathbb S_{\mathscr M},\mathscr N)\to\mathscr M\hat\otimes_\pi\mathscr N\quad\text{ is a quotient operator}
\]
by Lemma \ref{lem:proj_hom_quotient}. Hence, for any \(\alpha\in\mathscr M\hat\otimes_\pi\mathscr N\) and
\(\varepsilon>0\) we can find \(w=(w_v)_{v\in\mathbb S_{\mathscr M}}\in\ell_1(\mathbb S_{\mathscr M},\mathscr N)\) such that
\(\psi(w)=\alpha\) and \(|w|_1\leq|\alpha|_\pi+\varepsilon\). Given that
\[
\bigvee_{F\in\mathscr P_f(\mathbb S_{\mathscr M})}\sum_{v\in F}|v\otimes w_v|_\pi\leq\bigvee_{F\in\mathscr P_f(\mathbb S_{\mathscr M})}\sum_{v\in F}|w_v|\in L^0(\XX)^+,
\]
we see that \((v\otimes w_v)_{v\in\mathbb S_{\mathscr M}}\in\ell_1(\mathbb S_{\mathscr M},\mathscr M\hat\otimes_\pi\mathscr N)\). By unwrapping the various definitions, we obtain
\[\begin{split}
\alpha&\overset{\phantom{\eqref{eq:two_ell1_L0}}}=\psi\big((w_v)_{v\in\mathbb S_{\mathscr M}}\big)
\overset{\eqref{eq:sum_elem_ell1_I_M}}=\psi\bigg(\sum_{v\in\mathbb S_{\mathscr M}}(\delta_{vu}w_v)_{u\in\mathbb S_{\mathscr M}}\bigg)
\overset{\eqref{eq:summable_comp_dual}}=\sum_{v\in\mathbb S_{\mathscr M}}\psi\big((\delta_{vu}w_v)_{u\in\mathbb S_{\mathscr M}}\big)\\
&\overset{\phantom{\eqref{eq:two_ell1_L0}}}=\sum_{v\in\mathbb S_{\mathscr M}}\psi\big((\underline{\sf e}_v(\cdot)_u\cdot w_v)_{u\in\mathbb S_{\mathscr M}}\big)
\overset{\eqref{eq:vv_sequence}}=\sum_{v\in\mathbb S_{\mathscr M}}(\tilde\varphi\otimes_\pi{\rm id}_{\mathscr N})(\underline{\sf e}_v\otimes w_v)
=\sum_{v\in\mathbb S_{\mathscr M}}\varphi\big(\phi(\underline{\sf e}_v)\big)\otimes w_v\\
&\overset{\eqref{eq:two_ell1_L0}}=\sum_{v\in\mathbb S_{\mathscr M}}\varphi\big((\delta_{vu}\1_\X)_{u\in\mathbb S_{\mathscr M}}\big)\otimes w_v
\overset{\eqref{eq:mod_quotient_ell1}}=\sum_{v\in\mathbb S_{\mathscr M}}\Big(\sum_{u\in\mathbb S_{\mathscr M}}\delta_{vu}u\Big)\otimes w_v
=\sum_{v\in\mathbb S_{\mathscr M}}v\otimes w_v.
\end{split}\]
It follows that there exists a sequence \((v_n)_{n\in\N}\subseteq\mathbb S_{\mathscr M}\) such that, letting
\(w_n\coloneqq w_{v_n}\) for every \(n\in\N\), we have \((v_n\otimes w_n)_{n\in\N}\in\ell_1(\N,\mathscr M\hat\otimes_\pi\mathscr N)\),
\(\alpha=\sum_{n\in\N}v_n\otimes w_n\), and \(\sum_{n\in\N}|v_n||w_n|=|w|_1\leq|\alpha|_\pi+\varepsilon\).
Therefore, we have proved that \(q(\alpha)\leq|\alpha|_\pi+\varepsilon\). By letting \(\varepsilon\searrow 0\),
we conclude that \(|\alpha|_\pi=q(\alpha)\).
\end{proof}
\subsection{Relation with duals and pullbacks}
In order to provide a characterisation of the dual of the projective tensor product in Theorem \ref{thm:dual_proj_tensor},
we need to apply the following universal property:
\begin{theorem}[Universal property of the projective tensor product]\label{thm:univ_prop_proj_tensor}
Let \(\XX\) be a \(\sigma\)-finite measure space. Let \(\mathscr M\), \(\mathscr N\), \(\mathscr Q\) be Banach \(L^0(\XX)\)-modules. Then for any \(b\in{\rm B}(\mathscr M,\mathscr N;\mathscr Q)\)
there exists a unique \(\tilde b_\pi\in\textsc{Hom}(\mathscr M\hat\otimes_\pi\mathscr N;\mathscr Q)\) for which the following diagram commutes:
\[\begin{tikzcd}
\mathscr M\times\mathscr N \arrow[r,"b"] \arrow[d,swap,"\otimes"] & \mathscr Q \\
\mathscr M\hat\otimes_\pi\mathscr N \arrow[ur,swap,"\tilde b_\pi"] &
\end{tikzcd}\]
Also, \({\rm B}(\mathscr M,\mathscr N;\mathscr Q)\ni b\mapsto\tilde b_\pi\in\textsc{Hom}(\mathscr M\hat\otimes_\pi\mathscr N;\mathscr Q)\)
is an isomorphism of Banach \(L^0(\XX)\)-modules.
\end{theorem}
\begin{proof}
Let \(b\in{\rm B}(\mathscr M,\mathscr N;\mathscr Q)\) be fixed. Denote by \(\tilde b\colon\mathscr M\otimes\mathscr N\to\mathscr Q\) the \(L^0(\XX)\)-linearisation
of \(b\) given by Lemma \ref{lem:alg_tensor_hom}. For any tensor \(\alpha=\sum_{i=1}^n v_i\otimes w_i\in\mathscr M\otimes\mathscr N\), we can estimate
\[
|\tilde b(\alpha)|\leq\sum_{i=1}^n\big|\tilde b(v_i\otimes w_i)\big|=\sum_{i=1}^n|b(v_i,w_i)|\leq|b|\sum_{i=1}^n|v_i||w_i|.
\]
By passing to the infimum over all representations of \(\alpha\), we get \(|\tilde b(\alpha)|\leq|b||\alpha|_\pi\),
whence it follows that \(\tilde b\in\textsc{Hom}(\mathscr M\otimes_\pi\mathscr N;\mathscr Q)\) and \(|\tilde b|\leq|b|\).
Letting \(\tilde b_\pi\) be the unique element of \(\textsc{Hom}(\mathscr M\hat\otimes_\pi\mathscr N;\mathscr Q)\) extending \(\tilde b\),
we have that \(|\tilde b_\pi|=|\tilde b|\leq|b|\). On the other hand, we have that
\[
|b(v,w)|=\big|\tilde b_\pi(v\otimes w)\big|\leq|\tilde b_\pi||v\otimes w|_\pi=|\tilde b_\pi||v||w|
\quad\text{ for every }(v,w)\in\mathscr M\times\mathscr N,
\]
which implies that \(|b|\leq|\tilde b_\pi|\). All in all, we have shown that \(|\tilde b_\pi|=|b|\). Moreover, the resulting operator
\({\rm B}(\mathscr M,\mathscr N;\mathscr Q)\ni b\mapsto\tilde b_\pi\in\textsc{Hom}(\mathscr M\hat\otimes_\pi\mathscr N;\mathscr Q)\)
is clearly a homomorphism of Banach \(L^0(\XX)\)-modules. In order to conclude, it remains to check that such map is surjective.
To this aim, let \(T\in\textsc{Hom}(\mathscr M\hat\otimes_\pi\mathscr N;\mathscr Q)\) be fixed. Now define
\(b^T\colon\mathscr M\times\mathscr N\to\mathscr Q\) as \(b^T(v,w)\coloneqq T(v\otimes w)\) for every \((v,w)\in\mathscr M\times\mathscr N\).
Then \(b^T\in{\rm B}(\mathscr M,\mathscr N;\mathscr Q)\) by construction and \(\tilde b^T_\pi=T\) by the uniqueness part
of the statement. Therefore, the proof is complete.
\end{proof}

Choosing \(\mathscr Q\coloneqq L^0(\XX)\) in Theorem \ref{thm:univ_prop_proj_tensor},
we obtain the following characterisation of \(\mathscr M\hat\otimes_\pi\mathscr N\):
\begin{theorem}[Dual of \(\mathscr M\hat\otimes_\pi\mathscr N\)]\label{thm:dual_proj_tensor}
Let \(\XX\) be a \(\sigma\)-finite measure space. Let \(\mathscr M\) and \(\mathscr N\) be Banach \(L^0(\XX)\)-modules.
Then it holds that
\[
(\mathscr M\hat\otimes_\pi\mathscr N)^*\cong{\rm B}(\mathscr M,\mathscr N),
\]
an isomorphism of Banach \(L^0(\XX)\)-modules being given by \({\rm B}(\mathscr M,\mathscr N)\ni b\mapsto\tilde b_\pi\in(\mathscr M\hat\otimes_\pi\mathscr N)^*\).
\end{theorem}

As a consequence of Theorem \ref{thm:dual_proj_tensor}, we obtain a useful `dual representation formula' for \(|\cdot|_\pi\):
\begin{corollary}\label{cor:dual_formula_proj_norm}
Let \(\XX\) be a \(\sigma\)-finite measure space. Let \(\mathscr M\), \(\mathscr N\) be Banach \(L^0(\XX)\)-modules. Then
\begin{equation}\label{eq:dual_formula_proj_norm}
|\alpha|_\pi=\bigvee\big\{\tilde b_\pi(\alpha)\;\big|\;b\in{\rm B}(\mathscr M,\mathscr N),\,|b|\leq 1\big\}\quad\text{ for every }\alpha\in\mathscr M\hat\otimes_\pi\mathscr N.
\end{equation}
\end{corollary}
\begin{proof}
It follows from Theorem \ref{thm:dual_proj_tensor} and the Hahn--Banach theorem for normed \(L^0\)-modules.
\end{proof}

We conclude the section by proving that `pullbacks and projective tensor products commute':
\begin{theorem}[Pullbacks vs.\ projective tensor products]\label{thm:pullback_and_proj}
Let \(\XX=(\X,\Sigma_\X,\mm_\X)\), \(\YY=(\Y,\Sigma_\Y,\mm_\Y)\) be separable, \(\sigma\)-finite measure spaces.
Let \(\varphi\colon\X\to\Y\) be a measurable map such that \(\varphi_\#\mm_\X\ll\mm_\Y\). Let \(\mathscr M\) and \(\mathscr N\)
be Banach \(L^0(\YY)\)-modules. Then it holds that
\[
\varphi^*(\mathscr M\hat\otimes_\pi\mathscr N)\cong(\varphi^*\mathscr M)\hat\otimes_\pi(\varphi^*\mathscr N),
\]
the pullback map \(\varphi^*\colon\mathscr M\hat\otimes_\pi\mathscr N\to(\varphi^*\mathscr M)\hat\otimes_\pi(\varphi^*\mathscr N)\)
being the unique homomorphism such that
\[
\varphi^*(v\otimes w)=(\varphi^*v)\otimes(\varphi^*w)\quad\text{ for every }v\in\mathscr M\text{ and }w\in\mathscr N.
\]
\end{theorem}
\begin{proof}
First, we define the map \(T\colon\mathscr M\otimes_\pi\mathscr N\to(\varphi^*\mathscr M)\otimes_\pi(\varphi^*\mathscr N)\) as
\[
T\bigg(\sum_{i=1}^n v_i\otimes w_i\bigg)\coloneqq\sum_{i=1}^n(\varphi^*v_i)\otimes(\varphi^*w_i)
\quad\text{ for every }\sum_{i=1}^n v_i\otimes w_i\in\mathscr M\otimes_\pi\mathscr N.
\]
In order to prove that the map \(T\) is well-posed, it is sufficient to show that
\begin{equation}\label{eq:pullback_and_proj_aux1}
(v_i)_{i=1}^n\subseteq\mathscr M,\,(w_i)_{i=1}^n\subseteq\mathscr N,\,\sum_{i=1}^n v_i\otimes w_i=0
\quad\Longrightarrow\quad\sum_{i=1}^n(\varphi^*v_i)\otimes(\varphi^*w_i)=0.
\end{equation}
Let \({\sf I}_\varphi\colon\varphi^*\mathscr N^*\to(\varphi^*\mathscr N)^*\) be the isometric embedding
defined in \eqref{eq:def_I_varphi}. Corollary \ref{cor:null_tensor_conseq} yields
\[
\sum_{i=1}^n{\sf I}_\varphi(\varphi^*\eta)(\varphi^*w_i)\cdot(\varphi^*v_i)=\sum_{i=1}^n(\eta(w_i)\circ\varphi)\cdot(\varphi^*v_i)
=\varphi^*\bigg(\sum_{i=1}^n\eta(w_i)\cdot v_i\bigg)=0
\]
for every \(\eta\in\mathscr N^*\), whence it follows that \(\sum_{i=1}^n{\sf I}_\varphi(\theta)(\varphi^*w_i)\cdot(\varphi^*v_i)=0\)
for every \(\theta\in\mathscr G(\varphi^*[\mathscr N^*])\).
Using Theorem \ref{thm:seq_weak-star_density} and the density of \(\mathscr G(\varphi^*[\mathscr N^*])\)
in \(\varphi^*\mathscr N^*\), we obtain \(\sum_{i=1}^n\Theta(\varphi^*w_i)\cdot(\varphi^*v_i)=0\)
for every \(\Theta\in(\varphi^*\mathscr N)^*\), so that \(\sum_{i=1}^n(\varphi^*v_i)\otimes(\varphi^*w_i)=0\)
by Corollary \ref{cor:null_tensor_conseq}. This proves \eqref{eq:pullback_and_proj_aux1}.

Observe that \(T\) is linear by construction. Moreover, for any given \(\alpha\in\mathscr M\otimes_\pi\mathscr N\) we can estimate
\[\begin{split}
|T(\alpha)|_\pi&\leq\bigwedge\bigg\{\sum_{i=1}^n|\varphi^*v_i||\varphi^*w_i|\;\bigg|\;(v_i)_{i=1}^n\subseteq\mathscr M,\,(w_i)_{i=1}^n\subseteq\mathscr N,\,\alpha=\sum_{i=1}^n v_i\otimes w_i\bigg\}\\
&=\bigwedge\bigg\{\bigg(\sum_{i=1}^n|v_i||w_i|\bigg)\circ\varphi\;\bigg|\;(v_i)_{i=1}^n\subseteq\mathscr M,\,(w_i)_{i=1}^n\subseteq\mathscr N,\,\alpha=\sum_{i=1}^n v_i\otimes w_i\bigg\}
=|\alpha|_\pi\circ\varphi.
\end{split}\]
Now let us pass to the verification of the converse inequality. Given any \(b\in{\rm B}(\mathscr M,\mathscr N)\), we define
\[
b^\varphi\bigg(\sum_{i=1}^n\1_{E_i}\cdot\varphi^*v_i,\sum_{j=1}^m\1_{F_j}\cdot\varphi^*w_j\bigg)\coloneqq\sum_{i=1}^n\sum_{j=1}^m\1_{E_i\cap F_j}\,b(v_i,w_j)\circ\varphi\in L^0(\XX)
\]
for every \(\sum_{i=1}^n\1_{E_i}\cdot\varphi^*v_i\in\mathscr G(\varphi^*[\mathscr M])\) and \(\sum_{j=1}^m\1_{F_j}\cdot\varphi^*w_j\in\mathscr G(\varphi^*[\mathscr N])\). Notice that
\[\begin{split}
\bigg|\sum_{i=1}^n\sum_{j=1}^m\1_{E_i\cap F_j}\,b(v_i,w_j)\circ\varphi\bigg|&=\sum_{i=1}^n\sum_{j=1}^m\1_{E_i\cap F_j}|b(v_i,w_j)|\circ\varphi\\
&=|b|\circ\varphi\bigg|\sum_{i=1}^n\1_{E_i}\cdot\varphi^*v_i\bigg|\bigg|\sum_{j=1}^m\1_{F_j}\cdot\varphi^*w_j\bigg|.
\end{split}\]
Therefore, \(b^\varphi\colon\mathscr G(\varphi^*[\mathscr M])\times\mathscr G(\varphi^*[\mathscr N])\to L^0(\XX)\) can be uniquely extended to an \(L^0(\XX)\)-bilinear operator
\(b^\varphi\in{\rm B}(\varphi^*\mathscr M,\varphi^*\mathscr N)\) satisfying \(|b^\varphi|\leq|b|\circ\varphi\). Thanks to Corollary \ref{cor:dual_formula_proj_norm}, we deduce that
\[\begin{split}
|\alpha|_\pi\circ\varphi&=\bigvee\big\{\tilde b_\pi(\alpha)\circ\varphi\;\big|\;b\in{\rm B}(\mathscr M,\mathscr N),\,|b|\leq 1\big\}\\
&=\bigvee\bigg\{\sum_{i=1}^n b^\varphi(\varphi^*v_i,\varphi^*w_i)\;\bigg|\;b\in{\rm B}(\mathscr M,\mathscr N),\,|b|\leq 1\bigg\}\\
&\leq\bigvee\big\{\tilde B_\pi(T(\alpha))\;\big|\;B\in{\rm B}(\varphi^*\mathscr M,\varphi^*\mathscr N),\,|B|\leq 1\big\}=|T(\alpha)|_\pi
\end{split}\]
for every tensor \(\alpha=\sum_{i=1}^n v_i\otimes w_i\in\mathscr M\otimes_\pi\mathscr N\). All in all, we have shown that \(|T(\alpha)|_\pi=|\alpha|_\pi\circ\varphi\)
for every \(\alpha\in\mathscr M\otimes_\pi\mathscr N\). It follows that \(T\) can be uniquely extended to a linear operator
\[
\varphi^*\colon\mathscr M\hat\otimes_\pi\mathscr N\to(\varphi^*\mathscr M)\hat\otimes_\pi(\varphi^*\mathscr N)
\]
satisfying \(|\varphi^*\alpha|_\pi=|\alpha|_\pi\circ\varphi\) for every \(\alpha\in\mathscr M\hat\otimes_\pi\mathscr N\). Finally, it remains to check that
\(\varphi^*[\mathscr M\hat\otimes_\pi\mathscr N]\) generates \((\varphi^*\mathscr M)\hat\otimes_\pi(\varphi^*\mathscr N)\). Given any \(v\in\mathscr M\) and \(w\in\mathscr N\), we have
\((\varphi^*v)\otimes(\varphi^*w)=\varphi^*(v\otimes w)\). This shows that \(S\coloneqq\{(\varphi^*v)\otimes(\varphi^*w)\,:\,v\in\mathscr M,\,w\in\mathscr N\}\subseteq\varphi^*[\mathscr M\hat\otimes_\pi\mathscr N]\).
Given that \(\varphi^*[\mathscr M]\) and \(\varphi^*[\mathscr N]\) generate \(\varphi^*\mathscr M\) and \(\varphi^*\mathscr N\), respectively, we know from Lemma \ref{lem:generators_proj_tensor} that \(S\) generates
the space \((\varphi^*\mathscr M)\hat\otimes_\pi(\varphi^*\mathscr N)\), thus a fortiori \(\varphi^*[\mathscr M\hat\otimes_\pi\mathscr N]\) generates \((\varphi^*\mathscr M)\hat\otimes_\pi(\varphi^*\mathscr N)\).
\end{proof}
\subsection{Other consequences of the universal property}
Let us now discuss other consequences of Theorem \ref{thm:univ_prop_proj_tensor} and Corollary
\ref{cor:dual_formula_proj_norm}. Our first goal is to rewrite \eqref{eq:dual_formula_proj_norm} in a different fashion.
\begin{proposition}\label{prop:alternative_B(M,N)}
Let \(\XX\) be a \(\sigma\)-finite measure space. Let \(\mathscr M\), \(\mathscr N\) be Banach \(L^0(\XX)\)-modules.
Then the map sending \(b\) to \(v\to b(v,\cdot)\) is an isomorphism of Banach \(L^0(\XX)\)-modules from
\({\rm B}(\mathscr M,\mathscr N)\) to \(\textsc{Hom}(\mathscr M;\mathscr N^*)\). In particular, it holds that
\[
{\rm B}(\mathscr M,\mathscr N)\cong\textsc{Hom}(\mathscr M;\mathscr N^*).
\]
\end{proposition}
\begin{proof}
One can readily check that the map \(\varphi\colon b\mapsto\big(\mathscr M\ni v\to b(v,\cdot)\in\mathscr N^*\big)\)
is a homomorphism of Banach \(L^0(\XX)\)-modules between \({\rm B}(\mathscr M,\mathscr N)\) and \(\textsc{Hom}(\mathscr M;\mathscr N^*)\).
For any \(b\in{\rm B}(\mathscr M,\mathscr N)\), we have
\[
|\varphi(b)|=\bigvee_{v\in\mathscr M}\frac{\1_{\{|v|>0\}}|b(v,\cdot)|}{|v|}=
\bigvee_{v\in\mathscr M}\bigvee_{w\in\mathscr N}\frac{\1_{\{|v|>0\}}\1_{\{|w|>0\}}|b(v,w)|}{|v||w|}=|b|.
\]
Finally, we check that \(\varphi\) is surjective. For any \(T\in\textsc{Hom}(\mathscr M;\mathscr N^*)\),
we define \(b^T\colon\mathscr M\times\mathscr N\to L^0(\XX)\) as \(b^T(v,w)\coloneqq T(v)(w)\) for every
\(v\in\mathscr M\) and \(w\in\mathscr N\). Then \(b^T\in{\rm B}(\mathscr M,\mathscr N)\) and \(\varphi(b^T)=T\).
\end{proof}

Similarly, we have that \({\rm B}(\mathscr M,\mathscr N)\cong\textsc{Hom}(\mathscr N;\mathscr M^*)\),
an isomorphism of Banach \(L^0(\XX)\)-modules being given by the operator
\({\rm B}(\mathscr M,\mathscr N)\ni b\mapsto\big(\mathscr N\ni w\mapsto b(\cdot,w)\in\mathscr M^*\big)\in\textsc{Hom}(\mathscr N;\mathscr M^*)\).
\begin{corollary}\label{cor:alt_proj_norm_with_oper}
Let \(\XX\) be a \(\sigma\)-finite measure space. Let \(\mathscr M\), \(\mathscr N\) be Banach \(L^0(\XX)\)-modules. Then
\[
\big(T(v_n)(w_n)\big)_{n\in\N}\in\ell_1(\N,L^0(\XX))\quad\;\forall
(v_n\otimes w_n)_{n\in\N}\in\ell_1(\N,\mathscr M\hat\otimes_\pi\mathscr N),\,T\in\textsc{Hom}(\mathscr M;\mathscr N^*).
\]
Moreover, for every \(\alpha\in\mathscr M\hat\otimes_\pi\mathscr N\) we have that
\[
|\alpha|_\pi=\bigvee\bigg\{\Big|\sum_{n\in\N}T(v_n)(w_n)\Big|\,\bigg|\;(v_n\otimes w_n)_n\in\ell_1(\N,\mathscr M\hat\otimes_\pi\mathscr N),\,
\sum_{n\in\N}v_n\otimes w_n=\alpha,\,T\in\mathbb D_{\textsc{Hom}(\mathscr M;\mathscr N^*)}\bigg\}.
\]
\end{corollary}
\begin{proof}
Let \((v_n\otimes w_n)_{n\in\N}\in\ell_1(\N,\mathscr M\hat\otimes_\pi\mathscr N)\). Denote \(\alpha\coloneqq\sum_{n\in\N}v_n\otimes w_n\in\mathscr M\hat\otimes_\pi\mathscr N\). Then
\[
\bigg|\sum_{n\in F}T(v_n)(w_n)\bigg|\leq|T|\sum_{n\in F}|v_n||w_n|=|T|\sum_{n\in F}|v_n\otimes w_n|_\pi
\quad\text{ for every }T\in\textsc{Hom}(\mathscr M;\mathscr N^*)
\]
and for every \(F\in\mathscr P_f(\N)\). By passing to the supremum over all \(F\in\mathscr P_f(\N)\), we thus obtain that
\begin{equation}\label{eq:alt_proj_norm_with_oper_aux}
\big|\big(T(v_n)(w_n)\big)_{n\in\N}\big|_1\leq|T|\big|(v_n\otimes w_n)_{n\in\N}\big|_1\in L^0(\XX)^+,
\end{equation}
which ensures that \(\big(T(v_n)(w_n)\big)_{n\in\N}\in\ell_1(\N,L^0(\XX))\). Now, let us introduce the shorthand notation
\[
Q(\alpha)\coloneqq\bigvee\bigg\{\Big|\sum_{n\in\N}T(v_n)(w_n)\Big|\,\bigg|\;(v_n\otimes w_n)_n\in\ell_1(\N,\mathscr M\hat\otimes_\pi\mathscr N),\,
\sum_{n\in\N}v_n\otimes w_n=\alpha,\,T\in\mathbb D_{\textsc{Hom}(\mathscr M;\mathscr N^*)}\bigg\}
\]
for every \(\alpha\in\mathscr M\hat\otimes_\pi\mathscr N\). On the one hand, whenever \((v_n\otimes w_n)_n\) and \(T\) are
competitors for \(Q(\alpha)\), we have that \(\big|\sum_{n\in\N}T(v_n)(w_n)\big|\leq\sum_{n\in\N}|v_n||w_n|\) by
\eqref{eq:alt_proj_norm_with_oper_aux}, so that \(Q(\alpha)\leq|\alpha|_\pi\) by Theorem \ref{thm:altern_proj_tensor}. On the
other hand, take any \(b\in{\rm B}(\mathscr M,\mathscr N)\) such that \(|b|\leq 1\). Proposition \ref{prop:alternative_B(M,N)}
tells that the element \(T_b\in\textsc{Hom}(\mathscr M;\mathscr N^*)\), which we define as \(T_b(v)\coloneqq b(v,\cdot)\)
for all \(v\in\mathscr M\), satisfies \(|T_b|\leq 1\). Hence, Lemma \ref{lem:summable_comp_dual} yields
\(\tilde b_\pi(\alpha)=\sum_{n\in\N}\tilde b_\pi(v_n\otimes w_n)=\sum_{n\in\N}T_b(v_n)(w_n)\). It follows that
\[
|\alpha|_\pi=\bigvee_{b\in\mathbb D_{{\rm B}(\mathscr M,\mathscr N)}}\tilde b_\pi(\alpha)=
\bigvee_{b\in\mathbb D_{{\rm B}(\mathscr M,\mathscr N)}}\sum_{n\in\N}T_b(v_n)(w_n)\leq Q(\alpha)
\]
thanks to Corollary \ref{cor:dual_formula_proj_norm}. Consequently, the statement is finally achieved.
\end{proof}

The following symmetric statement is verified as well: \(\big(S(w_n)(v_n)\big)_n\in\ell_1(\N,L^0(\XX))\) holds for every
\((v_n\otimes w_n)_n\in\ell_1(\N,\mathscr M\hat\otimes_\pi\mathscr N)\) and \(S\in\textsc{Hom}(\mathscr N;\mathscr M^*)\),
and we have that
\[
|\alpha|_\pi=\bigvee\bigg\{\Big|\sum_{n\in\N}S(w_n)(v_n)\Big|\,\bigg|\;(v_n\otimes w_n)_n\in\ell_1(\N,\mathscr M\hat\otimes_\pi\mathscr N),\,
\sum_{n\in\N}v_n\otimes w_n=\alpha,\,S\in\mathbb D_{\textsc{Hom}(\mathscr N;\mathscr M^*)}\bigg\}.
\]
These claims can be proved by arguing exactly as we did in the proof of Corollary \ref{cor:alt_proj_norm_with_oper}. 
\medskip

Next, we use Corollary \ref{cor:dual_formula_proj_norm} to characterise the `tensor diagonal' in \(L^0(\XX;\ell_2(I))\hat\otimes_\pi L^0(\XX;\ell_2(I))\):
\begin{proposition}
Let \(\XX\) be a \(\sigma\)-finite measure space. Then the Banach \(L^0(\XX)\)-submodule
of the space \(L^0(\XX;\ell_2(I))\hat\otimes_\pi L^0(\XX;\ell_2(I))\) generated by
\(\{\underline{\sf e}_i\otimes\underline{\sf e}_i\,:\,i\in I\}\) is isomorphic to \(L^0(\XX;\ell_1(I))\).
\end{proposition}
\begin{proof}
Let \(\mathscr M\) be the Banach \(L^0(\XX)\)-submodule of \(L^0(\XX;\ell_2(I))\hat\otimes_\pi L^0(\XX;\ell_2(I))\) that is generated by
\(\{\underline{\sf e}_i\otimes\underline{\sf e}_i\,:\,i\in I\}\). Observe that \(\mathscr M\) can be described as \(\mathscr M={\rm cl}_{ L^0(\XX;\ell_2(I))\hat\otimes_\pi L^0(\XX;\ell_2(I))}(M)\), where
\[
M\coloneqq\bigg\{\sum_{i\in F}f_i\cdot(\underline{\sf e}_i\otimes\underline{\sf e}_i)\;\bigg|\;F\in\mathscr P_f(I),\,\{f_i\}_{i\in F}\subseteq L^0(\XX)\bigg\}.
\]
For \(F\in\mathscr P_f(I)\) and \(f=\{f_i\}_{i\in F}\subseteq L^0(\XX)\), we define \(b^f\colon L^0(\XX;\ell_2(I))\times L^0(\XX;\ell_2(I))\to L^0(\XX)\) as
\[
b^f(g,h)\coloneqq\sum_{i\in F}{\rm sgn}(f_i)g(\cdot)_i\,h(\cdot)_i\quad\text{ for every }g,h\in L^0(\XX;\ell_2(I)).
\]
The map \(b^f\) is \(L^0(\XX)\)-bilinear by construction. Also, for any \(g,h\in L^0(\XX;\ell_2(I))\) we can estimate
\[
|b^f(g,h)|\leq\sum_{i\in F}|g(\cdot)_i||h(\cdot)_i|\leq\bigg(\sum_{i\in F}|g(\cdot)_i|^2\bigg)^{1/2}\bigg(\sum_{i\in F}|h(\cdot)_i|^2\bigg)^{1/2}\leq|g||h|,
\]
which yields \(b^f\in{\rm B}\big(L^0(\XX;\ell_2(I)),L^0(\XX;\ell_2(I))\big)\) and \(|b^f|\leq 1\). Now define \(\psi\colon M\to\ell_1(I,L^0(\XX))\) as
\[
\psi\bigg(\sum_{i\in F}f_i\cdot(\underline{\sf e}_i\otimes\underline{\sf e}_i)\bigg)\coloneqq\sum_{i\in F}(\delta_{ij}f_i)_{j\in I}
\quad\text{ for every }F\in\mathscr P_f(I)\text{ and }\{f_i\}_{i\in F}\subseteq L^0(\XX).
\]
By virtue of Corollary \ref{cor:dual_formula_proj_norm}, for any \(F\in\mathscr P_f(I)\) and \(f=\{f_i\}_{i\in F}\subseteq L^0(\XX)\) we can estimate
\[\begin{split}
\bigg|\sum_{i\in F}(\delta_{ij}f_i)_{j\in I}\bigg|_1&=\bigg|\Big(\sum_{i\in F}\delta_{ij}f_i\Big)_{j\in I}\bigg|_1=\sum_{i\in F}|f_i|=\sum_{i\in F}f_i\,{\rm sgn}(f_i)=\sum_{i\in F}f_i\sum_{j\in F}{\rm sgn}(f_j)\delta_{ij}^2\\
&=\sum_{i\in F}f_i\sum_{j\in F}{\rm sgn}(f_j)\underline{\sf e}_i(\cdot)_j^2=\sum_{i\in F}f_i\,b^f(\underline{\sf e}_i,\underline{\sf e}_i)=\sum_{i\in F}f_i\,\tilde b^f_\pi(\underline{\sf e}_i\otimes\underline{\sf e}_i)\\
&=\tilde b^f_\pi\bigg(\sum_{i\in F}f_i\cdot(\underline{\sf e}_i\otimes\underline{\sf e}_i)\bigg)\leq\bigg|\sum_{i\in F}f_i\cdot(\underline{\sf e}_i\otimes\underline{\sf e}_i)\bigg|_\pi.
\end{split}\]
This shows that the operator \(\psi\) is well-defined (thus also \(L^0(\XX)\)-linear by construction) and that it satisfies \(|\psi(\alpha)|_1\leq|\alpha|_\pi\) for all \(\alpha\in M\).
Conversely, for any \(\alpha=\sum_{i\in F}f_i\cdot(\underline{\sf e}_i\otimes\underline{\sf e}_i)\in M\) we have
\[
|\alpha|_\pi\leq\sum_{i\in F}|f_i||\underline{\sf e}_i\otimes\underline{\sf e}_i|_\pi=\sum_{i\in F}|f_i||\underline{\sf e}_i|^2=\sum_{i\in F}|f_i|=|\psi(\alpha)|_1.
\]
All in all, we have shown that \(\psi\) preserves the pointwise norm, thus it can be uniquely extended to a homomorphism \(\bar\psi\in\textsc{Hom}\big(\mathscr M;\ell_1(I,L^0(\XX))\big)\)
that satisfies \(|\bar\psi(\alpha)|_1=|\alpha|_\pi\) for every \(\alpha\in\mathscr M\). Letting \(\phi\colon L^0(\XX;\ell_1(I))\to\ell_1(I,L^0(\XX))\) be the isomorphism given by Corollary
\ref{cor:two_ell1_L0}, we deduce that \(\varphi\coloneqq\phi^{-1}\circ\bar\psi\in\textsc{Hom}\big(\mathscr M;L^0(\XX;\ell_1(I))\big)\) satisfies \(|\varphi(\alpha)|=|\alpha|_\pi\)
for every \(\alpha\in\mathscr M\). Finally, we verify that \(\varphi\) is surjective. Fix any \(a\in L^0(\XX;\ell_1(I))\). Thanks to Lemma \ref{lem:basic_L0_l1(I)}, we can find an increasing
sequence \((F_n)_{n\in\N}\subseteq\mathscr P_f(I)\) with \(a_n\coloneqq\sum_{i\in F_n}a(\cdot)_i\cdot\underline{\sf e}_i\to a\) as \(n\to\infty\). Given that
\[
\varphi\bigg(\sum_{i\in F_n}a(\cdot)_i\cdot(\underline{\sf e}_i\otimes\underline{\sf e}_i)\bigg)=\sum_{i\in F_n}\phi^{-1}\Big(\big(\delta_{ij}a(\cdot)_i\big)_{j\in I}\Big)
=\sum_{i\in F_n}a(\cdot)_i\cdot\underline{\sf e}_i=a_n,
\]
we deduce that \((a_n)_{n\in\N}\subseteq\varphi[M]\), whence it follows that \(a\in\varphi[\mathscr M]\). The proof is complete.
\end{proof}

Finally, we discuss a categorical consequence of Theorem \ref{thm:dual_proj_tensor}. We first need
to introduce the two functors \(\mathscr M\hat\otimes_\pi-\colon{\bf BanMod}_\XX\to{\bf BanMod}_\XX\)
and \(\textsc{Hom}(\mathscr M;-)\colon{\bf BanMod}_\XX\to{\bf BanMod}_\XX\), where \(\mathscr M\) is a
Banach \(L^0(\XX)\)-module. The functors \(\mathscr M\hat\otimes_\pi-\) and \(\textsc{Hom}(\mathscr M;-)\) are given as follows:
\begin{itemize}
\item[\(\rm i)\)] For any object \(\mathscr N\) of \({\bf BanMod}_\XX\), we define
\((\mathscr M\hat\otimes_\pi-)(\mathscr N)\coloneqq\mathscr M\hat\otimes_\pi\mathscr N\).
For any morphism \(T\colon\mathscr N\to\tilde{\mathscr N}\) in \({\bf BanMod}_\XX\), we define the morphism
\((\mathscr M\hat\otimes_\pi-)(T)\colon\mathscr M\hat\otimes_\pi\mathscr N\to\mathscr M\hat\otimes_\pi\tilde{\mathscr N}\)
as \((\mathscr M\hat\otimes_\pi-)(T)\coloneqq{\rm id}_{\mathscr M}\otimes_\pi T\).
\item[\(\rm ii)\)] For any object \(\mathscr Q\) of \({\bf BanMod}_\XX\), we set
\(\textsc{Hom}(\mathscr M;-)(\mathscr Q)\coloneqq\textsc{Hom}(\mathscr M;\mathscr Q)\).
For any morphism \(T\colon\mathscr Q\to\tilde{\mathscr Q}\) in \({\bf BanMod}_\XX\), we set
\(
\textsc{Hom}(\mathscr M;-)(T)\colon\textsc{Hom}(\mathscr M;\mathscr Q)\to\textsc{Hom}(\mathscr M;\tilde{\mathscr Q})
\)
as \(\textsc{Hom}(\mathscr M;-)(T)(S)\coloneqq T\circ S\) for every \(S\in\textsc{Hom}(\mathscr M;\mathscr Q)\).
\end{itemize}
We can now pass to the ensuing result, which states that \(\mathscr M\hat\otimes_\pi-\) is the left adjoint of \(\textsc{Hom}(\mathscr M;-)\):
\begin{proposition}
Let \(\XX\) be a \(\sigma\)-finite measure space and \(\mathscr M\) a Banach \(L^0(\XX)\)-module. Then
\[
(\mathscr M\hat\otimes_\pi-)\dashv\textsc{Hom}(\mathscr M;-).
\]
\end{proposition}
\begin{proof}
Our goal is to find a natural isomorphism \(\Phi\colon\textsc{Hom}(\mathscr M\hat\otimes_\pi-;-)\to\textsc{Hom}(-;\textsc{Hom}(\mathscr M;-)),\)
which means that \(\big(\mathscr M\hat\otimes_\pi-,\textsc{Hom}(\mathscr M;-),\Phi\big)\) is a hom-set adjunction. To this aim,
fix two Banach \(L^0(\XX)\)-modules \(\mathscr N\) and \(\mathscr Q\). We define \(\Phi_{\mathscr N,\mathscr Q}\colon
\textsc{Hom}(\mathscr M\hat\otimes_\pi\mathscr N;\mathscr Q)\to\textsc{Hom}(\mathscr N;\textsc{Hom}(\mathscr M;\mathscr Q))\) as
\[
\Phi_{\mathscr N,\mathscr Q}(T)(w)(v)\coloneqq T(v\otimes w)\quad\text{ for every }
T\in\textsc{Hom}(\mathscr M\hat\otimes_\pi\mathscr N;\mathscr Q)\text{ and }(w,v)\in\mathscr N\times\mathscr M.
\]
One can readily check that \(\Phi_{\mathscr N,\mathscr Q}\) is a morphism and \(|\Phi_{\mathscr N,\mathscr Q}|\leq 1\).
On the other hand, let \(L\) be a given element of \(\textsc{Hom}(\mathscr N;\textsc{Hom}(\mathscr M;\mathscr Q))\).
Define \(b^L\colon\mathscr M\times\mathscr N\to\mathscr Q\) as \(b^L(v,w)\coloneqq L(w)(v)\) for every
\((v,w)\in\mathscr M\times\mathscr N\). Since \(b^L\in{\rm B}(\mathscr M,\mathscr N;\mathscr Q)\) and \(|b^L|\leq|L|\),
we know from Theorem \ref{thm:dual_proj_tensor} that the element
\(\tilde b^L_\pi\in\textsc{Hom}(\mathscr M\hat\otimes_\pi\mathscr N;\mathscr Q)\) satisfies \(|\tilde b^L_\pi|\leq|L|\).
Since \(\tilde b^L_\pi(v\otimes w)=b^L(v,w)=L(w)(v)\) for every \((v,w)\in\mathscr M\times\mathscr N\), we deduce that
\(\Phi_{\mathscr N,\mathscr Q}(\tilde b^L_\pi)=L\) and \(|\Phi_{\mathscr N,\mathscr Q}(\tilde b^L_\pi)|=|L|\geq|\tilde b^L_\pi|\).
All in all, we have shown that \(\Phi_{\mathscr N,\mathscr Q}\) is an isomorphism. Let us finally check the naturality of
\(\Phi\). Given any two morphisms \(T\colon\tilde{\mathscr N}\to\mathscr N\) and \(S\colon\mathscr Q\to\tilde{\mathscr Q}\)
in \({\bf BanMod}_\XX\), we consider the morphisms
\[\begin{split}
\textsc{Hom}(\mathscr M\hat\otimes_\pi T;S)&\colon\textsc{Hom}(\mathscr M\hat\otimes_\pi\mathscr N;\mathscr Q)
\to\textsc{Hom}(\mathscr M\hat\otimes_\pi\tilde{\mathscr N};\tilde{\mathscr Q}),\\
\textsc{Hom}(T;\textsc{Hom}(\mathscr M;S))&\colon\textsc{Hom}(\mathscr N;\textsc{Hom}(\mathscr M;\mathscr Q))
\to\textsc{Hom}(\tilde{\mathscr N};\textsc{Hom}(\mathscr M;\tilde{\mathscr Q})),
\end{split}\]
which are given by
\(\textsc{Hom}(\mathscr M\hat\otimes_\pi T;S)(\varphi)\coloneqq S\circ\varphi\circ({\rm id}_{\mathscr M}\otimes_\pi T)\)
for every \(\varphi\in\textsc{Hom}(\mathscr M\hat\otimes_\pi\mathscr N;\mathscr Q)\) and
\(\textsc{Hom}(T;\textsc{Hom}(\mathscr M;S))(\psi)(\tilde w)\coloneqq S\circ(\psi\circ T)(\tilde w)\) for every
\(\psi\in\textsc{Hom}(\mathscr N;\textsc{Hom}(\mathscr M;\mathscr Q))\) and \(\tilde w\in\tilde{\mathscr N}\).
Unwrapping the various definitions, one can see that the following diagram is commutative:
\[\begin{tikzcd}
\textsc{Hom}(\mathscr M\hat\otimes_\pi\mathscr N;\mathscr Q) \arrow[d,swap,"\textsc{Hom}(\mathscr M\hat\otimes_\pi T;S)"]
\arrow[r,"\Phi_{\mathscr N,\mathscr Q}"] & \textsc{Hom}(\mathscr N;\textsc{Hom}(\mathscr M;\mathscr Q))
\arrow[d,"\textsc{Hom}(T;\textsc{Hom}(\mathscr M;S))"] \\
\textsc{Hom}(\mathscr M\hat\otimes_\pi\tilde{\mathscr N};\tilde{\mathscr Q})
\arrow[r,swap,"\Phi_{\tilde{\mathscr N},\tilde{\mathscr Q}}"] &
\textsc{Hom}(\tilde{\mathscr N};\textsc{Hom}(\mathscr M;\tilde{\mathscr Q}))
\end{tikzcd}\]
whence it follows that \(\Phi\) is a natural isomorphism. Consequently, the proof is complete.
\end{proof}

The previous result implies that the functor \(\mathscr M\hat\otimes_\pi-\) is cocontinuous, i.e.\ it preserves colimits.
Notice however that we are considering \(\mathscr M\hat\otimes_\pi-\) as an endofunctor on \({\bf BanMod}_\XX\),
which is only finitely cocomplete, and not on the cocomplete category \({\bf BanMod}_\XX^1\).
\section{Injective tensor products of Banach \texorpdfstring{\(L^0\)}{L0}-modules}\label{s:inj_tens}
\subsection{Definition and main properties}
We begin by introducing the injective pointwise norm:
\begin{theorem}
Let \(\XX\) be a \(\sigma\)-finite measure space. Let \(\mathscr M\), \(\mathscr N\) be Banach \(L^0(\XX)\)-modules. Define
\begin{equation}\label{eq:def_inj_ptwse_norm}
|\alpha|_\varepsilon\coloneqq\bigvee\bigg\{\Big|\sum_{i=1}^n\omega(v_i)\eta(w_i)\Big|\;\bigg|\;
\sum_{i=1}^n v_i\otimes w_i=\alpha,\,\omega\in\mathbb D_{\mathscr M^*},\,\eta\in\mathbb D_{\mathscr N^*}\bigg\}
\end{equation}
for every \(\alpha\in\mathscr M\otimes\mathscr N\). Then \(|\cdot|_\varepsilon\colon\mathscr M\otimes\mathscr N\to L^0(\XX)^+\)
is a pointwise norm on \(\mathscr M\otimes\mathscr N\). Moreover,
\begin{equation}\label{eq:inj_norm_elem_tensor}
|v\otimes w|_\varepsilon=|v||w|\quad\text{ for every }v\in\mathscr M\text{ and }w\in\mathscr N.
\end{equation}
\end{theorem}
\begin{proof}
One can readily check that \(|\cdot|_\varepsilon\) verifies the pointwise norm axioms; the fact that
\(|\alpha|_\varepsilon=0\) implies \(\alpha=0\) is a consequence of Lemma \ref{lem:crit_null_tensor}.
To prove \eqref{eq:inj_norm_elem_tensor}, notice first that Lemma \ref{lem:crit_null_tensor} yields
\[
\bigg|\sum_{i=1}^n\omega(v_i)\eta(w_i)\bigg|=\big|\omega(v)\eta(w)\big|\leq|\omega||v||\eta||w|\leq|v||w|
\]
whenever \(\sum_{i=1}^n v_i\otimes w_i\) is a representation of the tensor \(v\otimes w\) and for all
\((\omega,\eta)\in\mathbb D_{\mathscr M^*}\times\mathbb D_{\mathscr N^*}\). Hence, we obtain that
\(|v\otimes w|_\varepsilon\leq|v||w|\). Conversely, an application of the Hahn--Banach theorem gives
two elements \(\omega_v\in\mathbb S_{\mathscr M^*}\) and \(\eta_w\in\mathbb S_{\mathscr N^*}\) such
that \(\omega_v(v)=|v|\) and \(\eta_w(w)=|w|\). Therefore, we have that
\(|v||w|=\omega_v(v)\eta_w(w)\leq|v\otimes w|_\varepsilon\). All in all, the identity in
\eqref{eq:inj_norm_elem_tensor} is proved.
\end{proof}
\begin{remark}{\rm
Observe that \(|\alpha|_\varepsilon\leq|\alpha|_\pi\) for all \(\alpha\in\mathscr M\otimes\mathscr N\).
Indeed, if \(\alpha=\sum_{i=1}^n v_i\otimes w_i\in\mathscr M\otimes\mathscr N\), \(\omega\in\mathbb D_{\mathscr M^*}\),
and \(\eta\in\mathbb D_{\mathscr N^*}\), then
\(\big|\sum_{i=1}^n\omega(v_i)\eta(w_i)\big|\leq\sum_{i=1}^n|\omega(v_i)||\eta(w_i)|\leq\sum_{i=1}^n|v_i||w_i|\).
\fr}\end{remark}
\begin{definition}[Injective tensor product]
Let \(\XX\) be a \(\sigma\)-finite measure space. Let \(\mathscr M\) and \(\mathscr N\) be Banach \(L^0(\XX)\)-modules.
Then we denote by \(\mathscr M\otimes_\varepsilon\mathscr N\) the normed \(L^0(\XX)\)-module
\((\mathscr M\otimes\mathscr N,|\cdot|_\varepsilon)\), where the pointwise norm \(|\cdot|_\varepsilon\)
is defined as in \eqref{eq:def_inj_ptwse_norm}. Moreover, the \textbf{injective tensor product} of \(\mathscr M\)
and \(\mathscr N\) is the Banach \(L^0(\XX)\)-module \(\mathscr M\hat\otimes_\varepsilon\mathscr N\) defined as
the \(L^0(\XX)\)-completion of \(\mathscr M\otimes_\varepsilon\mathscr N\).
\end{definition}

The injective tensor product can be realised as a Banach \(L^0(\XX)\)-submodule of \({\rm B}(\mathscr M^*,\mathscr N^*)\):
\begin{proposition}
Let \(\XX\) be a \(\sigma\)-finite measure space. Let \(\mathscr M\) and \(\mathscr N\) be Banach \(L^0(\XX)\)-modules.
Given any tensor \(\alpha=\sum_{i=1}^n v_i\otimes w_i\in\mathscr M\otimes\mathscr N\), we define
\(B_\alpha\colon\mathscr M^*\times\mathscr N^*\to L^0(\XX)\) as
\begin{equation}\label{eq:def_B_alpha}
B_\alpha(\omega,\eta)\coloneqq\sum_{i=1}^n\omega(v_i)\eta(w_i)\quad\text{ for every }\omega\in\mathscr M^*\text{ and }\eta\in\mathscr N^*.
\end{equation}
Then \(B_\alpha\) is well-defined and belongs to \({\rm B}(\mathscr M^*,\mathscr N^*)\). Moreover, the resulting operator
\[
\mathscr M\otimes_\varepsilon\mathscr N\ni\alpha\mapsto B_\alpha\in{\rm B}(\mathscr M^*,\mathscr N^*)
\]
can be uniquely extended to an isomorphism of Banach \(L^0(\XX)\)-modules from the injective tensor product
\(\mathscr M\hat\otimes_\varepsilon\mathscr N\) to the closure of \(\{B_\alpha\,:\,\alpha\in\mathscr M\otimes\mathscr N\}\)
in \({\rm B}(\mathscr M^*,\mathscr N^*)\).
\end{proposition}
\begin{proof}
The well-posedness of \eqref{eq:def_B_alpha} follows from Lemma \ref{lem:crit_null_tensor}, while the rest is straightforward.
\end{proof}

The following result provides other two representations of the injective tensor product \(\mathscr M\hat\otimes_\varepsilon\mathscr N\):
\begin{proposition}\label{prop:alt_inj_with_Hom}
Let \(\XX\) be a \(\sigma\)-finite measure space. Let \(\mathscr M\) and \(\mathscr N\) be Banach \(L^0(\XX)\)-modules.
Given any tensor \(\alpha=\sum_{i=1}^n v_i\otimes w_i\in\mathscr M\otimes\mathscr N\), we define
\(L_\alpha\colon\mathscr M^*\to\mathscr N\) and \(R_\alpha\colon\mathscr N^*\to\mathscr M\) as
\[
L_\alpha(\omega)\coloneqq\sum_{i=1}^n\omega(v_i)\cdot w_i,\quad R_\alpha(\eta)\coloneqq\sum_{i=1}^n\eta(w_i)\cdot v_i
\quad\text{ for every }\omega\in\mathscr M^*\text{ and }\eta\in\mathscr N^*,
\]
respectively. Then \(L_\alpha\in\textsc{Hom}(\mathscr M^*;\mathscr N)\) and \(R_\alpha\in\textsc{Hom}(\mathscr N^*;\mathscr M)\).
Moreover, the resulting maps
\[
\mathscr M\otimes_\varepsilon\mathscr N\ni\alpha\mapsto L_\alpha\in\textsc{Hom}(\mathscr M^*;\mathscr N),
\qquad\mathscr M\otimes_\varepsilon\mathscr N\ni\alpha\mapsto R_\alpha\in\textsc{Hom}(\mathscr N^*;\mathscr M)
\]
can be uniquely extended to pointwise norm preserving homomorphisms defined on \(\mathscr M\hat\otimes_\varepsilon\mathscr N\).
\end{proposition}
\begin{proof}
We consider only \(L_\alpha\), the proof for \(R_\alpha\) being analogous. The well-posedness of \(L_\alpha\) follows
from Corollary \ref{cor:null_tensor_conseq}. It is then easy to check that \(L_\alpha\in\textsc{Hom}(\mathscr M^*;\mathscr N)\)
for every \(\alpha\in\mathscr M\otimes\mathscr N\) and that
\(\mathscr M\otimes_\varepsilon\mathscr N\ni\alpha\mapsto L_\alpha\in\textsc{Hom}(\mathscr M^*;\mathscr N)\)
is a homomorphism of normed \(L^0(\XX)\)-modules. Also,
\[
|L_\alpha|=\bigvee_{\omega\in\mathbb D_{\mathscr M^*}}|L_\alpha(\omega)|
=\bigvee_{\omega\in\mathbb D_{\mathscr M^*}}\bigvee_{\eta\in\mathbb D_{\mathscr N^*}}\big|\eta\big(L_\alpha(\omega)\big)\big|
=\bigvee_{\omega\in\mathbb D_{\mathscr M^*}}\bigvee_{\eta\in\mathbb D_{\mathscr N^*}}\bigg|\sum_{i=1}^n\omega(v_i)\eta(w_i)\bigg|=|\alpha|_\varepsilon
\]
for every \(\alpha=\sum_{i=1}^n v_i\otimes w_i\in\mathscr M\otimes_\varepsilon\mathscr N\) by the Hahn--Banach theorem.
The statement follows.
\end{proof}
\begin{corollary}\label{cor:inj_norm_with_norming_sets}
Let \(\XX\) be a \(\sigma\)-finite measure space. Let \(\mathscr M\), \(\mathscr N\) be Banach \(L^0(\XX)\)-modules.
Let \(\mathcal T\) and \(\mathcal S\) be norming subsets of \(\mathscr M^*\) and \(\mathscr N^*\), respectively.
Then it holds that
\[
|\alpha|_\varepsilon=\bigvee_{\omega\in\mathcal T}\bigg|\sum_{i=1}^n\omega(v_i)\cdot w_i\bigg|=\bigvee_{\eta\in\mathcal S}\bigg|\sum_{i=1}^n\eta(w_i)\cdot v_i\bigg|
\quad\text{ for every }\alpha=\sum_{i=1}^n v_i\otimes w_i\in\mathscr M\otimes_\varepsilon\mathscr N.
\]
\end{corollary}
\begin{proof}
Since \(|L_\alpha|=\bigvee_{\omega\in\mathcal T}|L_\alpha(\omega)|\) and \(|R_\alpha|=\bigvee_{\eta\in\mathcal S}|R_\alpha(\eta)|\),
it follows from Proposition \ref{prop:alt_inj_with_Hom}.
\end{proof}

Let us now consider the injective tensor product of homomorphisms of Banach \(L^0(\XX)\)-modules:
\begin{proposition}[Injective tensor products of homomorphisms]
Let \(\XX\) be a \(\sigma\)-finite measure space.
Let \(T\colon\mathscr M\to\tilde{\mathscr M}\) and \(S\colon\mathscr N\to\tilde{\mathscr N}\) be homomorphisms
of Banach \(L^0(\XX)\)-modules. Then there exists a unique homomorphism of Banach \(L^0(\XX)\)-modules
\(T\otimes_\varepsilon S\colon\mathscr M\hat\otimes_\varepsilon\mathscr N\to\tilde{\mathscr M}\hat\otimes_\varepsilon\tilde{\mathscr N}\) with
\[
(T\otimes_\varepsilon S)(v\otimes w)=T(v)\otimes S(w)\quad\text{ for every }v\in\mathscr M\text{ and }w\in\mathscr N.
\]
Moreover, it holds that \(|T\otimes_\varepsilon S|=|T||S|\).
\end{proposition}
\begin{proof}
Let \(T\otimes S\colon\mathscr M\otimes\mathscr N\to\tilde{\mathscr M}\otimes\tilde{\mathscr N}\) be as in Lemma \ref{lem:alg_tensor_hom}.
If \(\alpha=\sum_{i=1}^n v_i\otimes w_i\in\mathscr M\otimes\mathscr N\), then
\[\begin{split}
\big|(T\otimes S)(\alpha)\big|_\varepsilon&=\bigvee\bigg\{\Big|\sum_{i=1}^n\tilde\omega(T(v_i))\tilde\eta(S(w_i))\Big|\;
\bigg|\;\tilde\omega\in\mathbb D_{\tilde{\mathscr M}^*},\,\tilde\eta\in\mathbb D_{\tilde{\mathscr N}^*}\bigg\}\\
&\leq|T||S|\bigvee\bigg\{\Big|\sum_{i=1}^n\omega(v_i)\eta(w_i)\Big|\;\bigg|\;\omega\in\mathbb D_{\mathscr M^*},\,\eta\in\mathbb D_{\mathscr N^*}\bigg\}=|T||S||\alpha|_\varepsilon,
\end{split}\]
where we used the fact that \(\frac{\1_{\{|T|>0\}}}{|T|}\cdot(\tilde\omega\circ T)\in\mathbb D_{\mathscr M^*}\) and
\(\frac{\1_{\{|S|>0\}}}{|S|}\cdot(\tilde\eta\circ S)\in\mathbb D_{\mathscr N^*}\).
It follows that the operator \(T\otimes S\colon\mathscr M\otimes\mathscr N\to\tilde{\mathscr M}\otimes\tilde{\mathscr N}\)
can be uniquely extended to a homomorphism of Banach \(L^0(\XX)\)-modules
\(T\otimes_\varepsilon S\colon\mathscr M\hat\otimes_\varepsilon\mathscr N\to\tilde{\mathscr M}\hat\otimes_\varepsilon\tilde{\mathscr N}\)
satisfying \(|T\otimes_\varepsilon S|\leq|T||S|\). Finally, the validity of the converse inequality
\(|T\otimes_\varepsilon S|\geq|T||S|\) can be proved arguing as in Proposition \ref{prop:prod_norm_tens_hom}.
\end{proof}

One can easily check that \(L^0(\XX)\hat\otimes_\varepsilon L^0(\XX)=L^0(\XX)\otimes_\varepsilon L^0(\XX)\cong L^0(\XX)\) as Banach \(L^0(\XX)\)-modules via
\(L^0(\XX)\otimes_\varepsilon L^0(\XX)\ni\sum_{i=1}^n f_i\otimes g_i\mapsto\sum_{i=1}^n f_i g_i\in L^0(\XX)\). In particular, up to this identification,
\[
\omega\otimes_\varepsilon\eta\in(\mathscr M\hat\otimes_\varepsilon\mathscr N)^*,\quad|\omega\otimes_\varepsilon\eta|=|\omega||\eta|\quad\text{ for every }\omega\in\mathscr M^*\text{ and }\eta\in\mathscr N^*.
\]
\begin{lemma}\label{lem:generators_inj_tensor}
Let \(\XX\) be a \(\sigma\)-finite measure space and \(\mathscr M\), \(\mathscr N\) Banach \(L^0(\XX)\)-modules.
Let \(G\subseteq\mathscr M\) and \(H\subseteq\mathscr N\) be generating subsets. Then the set
\(\{v\otimes w\;\big|\;v\in G,\,w\in H\}\) generates \(\mathscr M\hat\otimes_\varepsilon\mathscr N\).
\end{lemma}
\begin{proof}
Argue exactly as in the proof of Lemma \ref{lem:generators_proj_tensor}, by just replacing \(|\cdot|_\pi\) with \(|\cdot|_\varepsilon\).
\end{proof}
\subsection{Pointwise crossnorms}\label{s:crossnorms}
Let us now introduce a class of `tensor product pointwise norms':
\begin{definition}[Reasonable pointwise crossnorm]\label{def:crossnorm}
Let \(\XX\) be a \(\sigma\)-finite measure space. Let \(\mathscr M\) and \(\mathscr N\) be Banach \(L^0(\XX)\)-modules.
Then a pointwise norm \(|\cdot|_c\) on \(\mathscr M\otimes\mathscr N\) is said to be a \textbf{reasonable pointwise crossnorm}
provided the following properties are verified:
\begin{itemize}
\item[\(\rm i)\)] \(|v\otimes w|_c\leq|v||w|\) for every \(v\in\mathscr M\) and \(w\in\mathscr N\).
\item[\(\rm ii)\)] \(\omega\otimes\eta\in(\mathscr M\otimes\mathscr N)_c^*\) and \(|\omega\otimes\eta|_{c^*}\leq|\omega||\eta|\)
for every \(\omega\in\mathscr M^*\) and \(\eta\in\mathscr N^*\), where we denote by \(((\mathscr M\otimes\mathscr N)_c^*,|\cdot|_{c^*})\)
the dual of the normed \(L^0(\XX)\)-module \((\mathscr M\otimes\mathscr N,|\cdot|_c)\).
\end{itemize}
\end{definition}

The projective pointwise norm and the injective pointwise norm are examples of reasonable pointwise crossnorms.
In fact, they are the `greatest' and the `least' crossnorms, respectively:
\begin{theorem}[Characterisation of reasonable pointwise crossnorms]\label{thm:char_crossnorm}
Let \(\XX\) be a \(\sigma\)-finite measure space. Let \(\mathscr M\) and \(\mathscr N\) be Banach \(L^0(\XX)\)-modules.
Let \(|\cdot|_c\) be a given pointwise norm on \(\mathscr M\otimes\mathscr N\). Then \(|\cdot|_c\) is a reasonable pointwise crossnorm if and only if
\begin{equation}\label{eq:char_crossnorm}
|\alpha|_\varepsilon\leq|\alpha|_c\leq|\alpha|_\pi\quad\text{ for every }\alpha\in\mathscr M\otimes\mathscr N.
\end{equation}
\end{theorem}
\begin{proof}
Suppose \(|\cdot|_c\) is a reasonable pointwise crossnorm. If \(\alpha=\sum_{i=1}^n v_i\otimes w_i\in\mathscr M\otimes\mathscr N\), then
\[
|\alpha|_c\leq\sum_{i=1}^n|v_i\otimes w_i|_c\leq\sum_{i=1}^n|v_i||w_i|,
\]
thus by taking the infimum over all representations of \(\alpha\) we get \(|\alpha|_c\leq|\alpha|_\pi\). Moreover, we have
\[
|\alpha|_\varepsilon=\bigvee\big\{|(\omega\otimes\eta)(\alpha)|\;\big|\;(\omega,\eta)\in\mathbb D_{\mathscr M^*}\times\mathbb D_{\mathscr N^*}\big\}
\leq\bigvee\big\{|\Theta(\alpha)|\;\big|\:\Theta\in\mathbb D_{(\mathscr M\otimes\mathscr N)_c^*}\big\}=|\alpha|_c.
\]
Conversely, if \eqref{eq:char_crossnorm} holds, then \(|v\otimes w|_c\leq|v\otimes w|_\pi=|v||w|\) for every \((v,w)\in\mathscr M\times\mathscr N\). Moreover,
\[
|(\omega\otimes\eta)(\alpha)|\leq|\omega\otimes\eta|_{\varepsilon^*}|\alpha|_\varepsilon\leq|\omega||\eta||\alpha|_c\quad\text{ for every }(\omega,\eta)\in\mathscr M^*\times\mathscr N^*,
\]
thus \(\omega\otimes\eta\in(\mathscr M\otimes\mathscr N)_c^*\) and \(|\omega\otimes\eta|_{c^*}\leq|\omega||\eta|\). Hence, \(|\cdot|_c\) is a reasonable pointwise crossnorm.
\end{proof}
\begin{proposition}
Let \(\XX\) be a \(\sigma\)-finite measure space. Let \(\mathscr M\) and \(\mathscr N\) be Banach \(L^0(\XX)\)-modules. Let \(|\cdot|_c\) be a
reasonable pointwise crossnorm on \(\mathscr M\otimes\mathscr N\). Then it holds that
\[
|v\otimes w|_c=|v||w|,\quad|\omega\otimes\eta|_{c^*}=|\omega||\eta|\quad\text{ for every }(v,w,\omega,\eta)\in\mathscr M\times\mathscr N\times\mathscr M^*\times\mathscr N^*.
\]
\end{proposition}
\begin{proof}
First, \eqref{eq:char_crossnorm} and \eqref{eq:inj_norm_elem_tensor} yield \(|v\otimes w|_c\geq|v\otimes w|_\varepsilon=|v||w|\). Proposition \ref{prop:prod_norm_tens_hom} and \eqref{eq:char_crossnorm} yield
\[\begin{split}
|\omega\otimes\eta|_{c^*}&=\bigvee\big\{|(\omega\otimes\eta)(\alpha)|\;\big|\;\alpha\in\mathscr M\otimes\mathscr N,\,|\alpha|_c\leq 1\big\}\\
&\geq\bigvee\big\{|(\omega\otimes_\pi\eta)(\alpha)|\;\big|\;\alpha\in\mathbb D_{\mathscr M\otimes_\pi\mathscr N}\big\}=|\omega\otimes_\pi\eta|=|\omega||\eta|.
\end{split}\]
Therefore, the proof of the statement is complete.
\end{proof}

We conclude the section with another important example of reasonable pointwise crossnorm. A Banach \(L^0(\XX)\)-module \(\mathscr H\) is a \textbf{Hilbert \(L^0(\XX)\)-module}
if \(\langle\cdot,\cdot\rangle\in{\rm B}(\mathscr H,\mathscr H)\), where we define
\[
\langle v,w\rangle\coloneqq\frac{|v+w|^2-|v|^2-|w|^2}{2}\in L^0(\XX)\quad\text{ for every }v,w\in\mathscr H.
\]
The Riesz representation theorem for Hilbert \(L^0(\XX)\)-modules states that the space \(\mathscr H\) is canonically isomorphic to its dual \(\mathscr H^*\)
via the operator \(\mathscr H\ni v\mapsto\langle v,\cdot\rangle\in\mathscr H^*\).
Now let \(\mathscr H\) and \(\mathscr K\) be Hilbert \(L^0(\XX)\)-modules. Then we define the \textbf{Hilbert--Schmidt pointwise norm} on \(\mathscr H\otimes\mathscr K\) as
\begin{equation}\label{eq:def_HS}
|\alpha|_{\rm HS}\coloneqq\bigg(\sum_{i,j=1}^n\langle v_i,v_j\rangle\langle w_i,w_j\rangle\bigg)^{1/2}\in L^0(\XX)^+\quad\text{ for every }\alpha=\sum_{i=1}^n v_i\otimes w_i\in\mathscr H\otimes\mathscr K.
\end{equation}
Following \cite[Section 1.5]{Gigli14}, we define the \textbf{tensor product of Hilbert modules} \(\mathscr H\otimes_{\rm HS}\mathscr K\) as the completion of \((\mathscr H\otimes\mathscr K,|\cdot|_{\rm HS})\).
It holds that \(\mathscr H\otimes_{\rm HS}\mathscr K\) is a Hilbert \(L^0(\XX)\)-module. Also,
\[
|\cdot|_{\rm HS}\quad\text{ is a reasonable pointwise crossnorm on }\mathscr H\otimes\mathscr K.
\]
Indeed, the identity \(|v\otimes w|_{\rm HS}=|v||w|\) for all \((v,w)\in\mathscr H\times\mathscr K\) is a direct consequence of \eqref{eq:def_HS}, thus Definition \ref{def:crossnorm} i) holds.
Definition \ref{def:crossnorm} ii) then follows as well, thanks to the Riesz representation theorem for Hilbert \(L^0(\XX)\)-modules. In particular, Theorem \ref{thm:char_crossnorm} ensures that
\[
|\alpha|_\varepsilon\leq|\alpha|_{\rm HS}\leq|\alpha|_\pi\quad\text{ for every }\alpha\in\mathscr H\otimes\mathscr K.
\]
\subsection{Relation with order-continuous maps}\label{s:rel_with_ord-cont}
As we already mentioned in the first paragraph of Section \ref{s:cont_mod-valued}, the Banach space \({\rm C}(K)\) (where
\(K\) is a compact, Hausdorff topological space) has a special relevance in connection with injective tensor products.
For instance, it holds that \({\rm C}(K)\hat\otimes_\varepsilon\B\cong{\rm C}(K;\B)\) for every Banach space \(\B\),
whence it follows that any quotient operator \(f\colon\B_1\to\B_2\) between Banach spaces induces a quotient operator
\({\rm id}\otimes_\varepsilon f\colon{\rm C}(K)\hat\otimes_\varepsilon\B_1\to{\rm C}(K)\hat\otimes_\varepsilon\B_2\).
The goal of the present section is to extend these results to the setting of Banach \(L^0(\XX)\)-modules,
taking as \(K\) a compact, Hausdorff uniform space, and replacing \({\rm C}(K)\) with \({\rm UC}_{\rm ord}(K;L^0(\XX))\).
\begin{theorem}\label{thm:inj_tens_UC}
Let \((K,\Phi)\) be a compact, Hausdorff uniform space. Let \(\XX\) be a \(\sigma\)-finite measure space and
\(\mathscr M\) a Banach \(L^0(\XX)\)-module. Then the unique linear and continuous operator
\[
\mathfrak j\colon{\rm UC}_{\rm ord}(K;L^0(\XX))\hat\otimes_\varepsilon\mathscr M\to{\rm UC}_{\rm ord}(K;\mathscr M)
\]
satisfying \(\mathfrak j(f\otimes v)(\cdot)=f(\cdot)\cdot v\) for every \(f\in{\rm UC}_{\rm ord}(K;L^0(\XX))\)
and \(v\in\mathscr M\) is an isomorphism.
\end{theorem}
\begin{proof}
Note that \(|f(\cdot)\cdot v|=|f(\cdot)||v|\) and \({\rm Var}(f(\cdot)\cdot v;\mathcal U)=|v|{\rm Var}(f;\mathcal U)\)
for all \(f\in{\rm UC}_{\rm ord}(K;L^0(\XX))\), \(v\in\mathscr M\), and \(\mathcal U\in\Phi\), which implies
that \(f(\cdot)\cdot v\in{\rm UC}_{\rm ord}(K;\mathscr M)\). Therefore, it makes sense to define the map
\(\mathfrak j\colon{\rm UC}_{\rm ord}(K;L^0(\XX))\otimes_\varepsilon\mathscr M\to{\rm UC}_{\rm ord}(K;\mathscr M)\)
in the following way:
\[
\mathfrak j\bigg(\sum_{i=1}^n f_i\otimes v_i\bigg)\coloneqq\sum_{i=1}^n f_i(\cdot)\cdot v_i\quad
\text{ for every }\sum_{i=1}^n f_i\otimes v_i\in{\rm UC}_{\rm ord}(K;L^0(\XX))\otimes_\varepsilon\mathscr M.
\]
To prove that the definition of \(\mathfrak j\) is well-posed amounts to showing that
\begin{equation}\label{eq:inj_tens_UC_aux}
(f_i)_{i=1}^n\subseteq{\rm UC}_{\rm ord}(K;L^0(\XX)),\,(v_i)_{i=1}^n\subseteq\mathscr M,\,\sum_{i=1}^n f_i\otimes v_i=0
\quad\Longrightarrow\quad\sum_{i=1}^n f_i(\cdot)\cdot v_i=0.
\end{equation}
Assuming \(\sum_{i=1}^n f_i\otimes v_i=0\), we have that \(\sum_{i=1}^n f_i(p)\cdot v_i=\sum_{i=1}^n\delta_p(f_i)\cdot v_i=0\)
for every \(p\in K\) by Remark \ref{rmk:evaluations_dual_UC} and Corollary \ref{cor:null_tensor_conseq}, thus showing
that \eqref{eq:inj_tens_UC_aux} holds. Moreover, if \(\alpha=\sum_{i=1}^n f_i\otimes v_i\) is a given element
of \({\rm UC}_{\rm ord}(K;L^0(\XX))\otimes_\varepsilon\mathscr M\), then by Corollary \ref{cor:inj_norm_with_norming_sets}
and Remark \ref{rmk:evaluations_dual_UC} we can compute
\[
|\mathfrak j(\alpha)|=\bigvee_{p\in K}\bigg|\sum_{i=1}^n f_i(p)\cdot v_i\bigg|=
\bigvee_{p\in K}\bigg|\sum_{i=1}^n\delta_p(f_i)\cdot v_i\bigg|=|\alpha|_\varepsilon.
\]
Since \(\mathfrak j\) is also linear by construction, it can be uniquely extended to a homomorphism of Banach \(L^0(\XX)\)-modules
\(\mathfrak j\colon{\rm UC}_{\rm ord}(K;L^0(\XX))\hat\otimes_\varepsilon\mathscr M\to{\rm UC}_{\rm ord}(K;\mathscr M)\)
that preserves the pointwise norm.

In order to conclude, it remains to check that the isometric embedding map \(\mathfrak j\) is also surjective.
Let \(v\in{\rm UC}_{\rm ord}(K;\mathscr M)\) be given. Then we can find \((\mathcal U_n)_{n\in\N}\subseteq\Phi\)
with \({\rm Var}(v;\mathcal U_n)\to 0\) in \(L^0(\XX)\). Fix any \(n\in\N\). Given that \(\{\mathcal U_n[p]\}_{p\in K}\)
is an open cover of the compact set \(K\), there exist \(k_n\in\N\) and \((p^n_i)_{i=1}^{k_n}\subseteq K\)
such that \(K=\bigcup_{i=1}^{k_n}\mathcal U_n[p^n_i]\). Now, take a continuous partition of unity \((\eta^n_i)_{i=1}^{k_n}\)
subordinated to \((\mathcal U_n[p^n_i])_{i=1}^{k_n}\) (see e.g.\ \cite{rudin1987real}), i.e.\ \(\eta^n_i\colon K\to[0,1]\)
is continuous, supported in \(\mathcal U_n[p^n_i]\), and \(\sum_{i=1}^{k_n}\eta^n_i=1\) on \(K\).
With no loss of generality, we can also assume that for any \(i=1,\ldots,k_n\) there exists \(q^n_i\in\mathcal U_n[p^n_i]\)
such that \(\eta^n_i(q^n_i)=1\); this fact will be used in Remark \ref{rmk:density_in_UC}. Let us define
\[
\alpha_n\coloneqq\sum_{i=1}^{k_n}(\eta^n_i(\cdot)\1_\X)\otimes v(p^n_i)\in{\rm UC}_{\rm ord}(K;L^0(\XX))\otimes_\varepsilon\mathscr M.
\]
Observe that for any given point \(p\in K\) we can estimate
\[
\big|\mathfrak j(\alpha_n)(p)-v(p)\big|=\bigg|\sum_{i=1}^{k_n}\eta^n_i(p)\big(v(p^n_i)-v(p)\big)\bigg|
\leq\sum_{i=1}^{k_n}\eta^n_i(p)\big|v(p^n_i)-v(p)\big|\leq{\rm Var}(v;\mathcal U_n).
\]
By passing to the supremum over all \(p\in K\), we deduce that \(|\mathfrak j(\alpha_n)-v|\leq{\rm Var}(v;\mathcal U_n)\),
whence it follows that \(\mathfrak j(\alpha_n)\to v\) in \({\rm UC}_{\rm ord}(K;\mathscr M)\). This shows that
the map \(\mathfrak j\) is surjective, as desired.
\end{proof}
\begin{remark}\label{rmk:density_in_UC}{\rm
We isolate a useful byproduct of the proof of Theorem \ref{thm:inj_tens_UC}: letting
\[
\mathscr F_n\coloneqq\bigg\{(\eta_i)_{i=1}^n\subseteq{\rm C}(K;[0,1])\;\bigg|\;
\{\eta_i=1\}\neq\varnothing\text{ for every }i=1,\ldots,n,\,\sum_{i=1}^n\eta_i=1\bigg\}\quad\text{ for all }n\in\N,
\]
where \({\rm C}(K;[0,1])\) stands for the space of all continuous functions from \(K\) to \([0,1]\), we have that
\[
\mathscr D\coloneqq\bigcup_{n\in\N}\bigg\{\sum_{i=1}^n(\eta_i(\cdot)\1_\X)\otimes v_i\;\bigg|\;(\eta_i)_{i=1}^n\in\mathscr F_n,
\,(v_i)_{i=1}^n\subseteq\mathscr M\bigg\}\quad\text{ is dense in }{\rm UC}_{\rm ord}(K;L^0(\XX))\hat\otimes_\varepsilon\mathscr M,
\]
or equivalently \(\bigcup_{n\in\N}\big\{\sum_{i=1}^n\eta_i(\cdot)v_i\;\big|\;(\eta_i)_{i=1}^n\in\mathscr F_n,
\,(v_i)_{i=1}^n\subseteq\mathscr M\big\}\) is dense in \({\rm UC}_{\rm ord}(K;\mathscr M)\). Also,
\begin{equation}\label{eq:density_UC_form}
|\alpha|_\varepsilon=\bigvee_{i=1}^n|v_i|\quad\text{ for every }\alpha=\sum_{i=1}^n(\eta_i(\cdot)\1_\X)\otimes v_i\in\mathscr D.
\end{equation}
In order to prove it, take \((q_i)_{i=1}^n\subseteq K\) such that \(\eta_i(q_i)=1\) for every \(i=1,\ldots,n\). Therefore,
\[
|\alpha|_\varepsilon=|\mathfrak j(\alpha)|=\bigvee_{p\in K}\bigg|\sum_{i=1}^n\eta_i(p)v_i\bigg|\leq
\bigvee_{p\in K}\sum_{i=1}^n\eta_i(p)|v_i|\leq\bigvee_{i=1}^n|v_i|=\bigvee_{i=1}^n|\mathfrak j(\alpha)(q_i)|
\leq|\alpha|_\varepsilon,
\]
which shows the validity of \eqref{eq:density_UC_form}.
\fr}\end{remark}
\begin{proposition}
Let \((K,\Phi)\) be a compact, Hausdorff uniform space. Let \(\XX\) be a \(\sigma\)-finite measure space 
and let \(\mathscr M\), \(\mathscr N\) be Banach \(L^0(\XX)\)-modules. Let \(T\colon\mathscr M\to\mathscr N\)
be a quotient operator. Then
\[
{\rm id}\otimes_\varepsilon T\colon{\rm UC}_{\rm ord}(K;L^0(\XX))\hat\otimes_\varepsilon\mathscr M
\to{\rm UC}_{\rm ord}(K;L^0(\XX))\hat\otimes_\varepsilon\mathscr N\quad\text{ is a quotient operator.}
\]
\end{proposition}
\begin{proof}
First, notice that \(|{\rm id}\otimes_\varepsilon T|=|T|\leq 1\).
Our goal is to apply Lemma \ref{lem:suff_cond_quotient_oper}. To this aim, we shorten
\({\rm UC}\coloneqq{\rm UC}_{\rm ord}(K;L^0(\XX))\), and we fix any
\(\beta\in{\rm UC}\hat\otimes_\varepsilon\mathscr N\) and \(\varepsilon>0\).
By virtue of Remark \ref{rmk:density_in_UC}, we can find \(n\in\N\), \((\eta_i)_{i=1}^n\in\mathscr F_n\), and
\((w_i)_{i=1}^n\subseteq\mathscr N\) such that \(\tilde\beta\coloneqq\sum_{i=1}^n(\eta_i(\cdot)\1_\X)\otimes w_i\)
satisfies \(\sfd_{{\rm UC}\hat\otimes_\varepsilon\mathscr N}(\tilde\beta,\beta)<\varepsilon/2\).
Since \(T\) is a quotient operator, for any \(i=1,\ldots,n\) we can find \(v_i\in\mathscr M\) such that
\(T(v_i)=w_i\) and \(|v_i|\leq|w_i|+\delta\), where \(\delta>0\) is chosen so that
\(\sfd_{L^0(\XX)}(\delta\1_\X,0)<\varepsilon/2\). Now define \(\alpha\coloneqq\sum_{i=1}^n(\eta_i(\cdot)\1_\X)\otimes v_i
\in{\rm UC}\hat\otimes_\varepsilon\mathscr M\). Since
\(({\rm id}\otimes_\varepsilon T)(\alpha)=\tilde\beta\), we have
\(\sfd_{{\rm UC}\hat\otimes_\varepsilon\mathscr N}(({\rm id}\otimes_\varepsilon T)(\alpha),\beta)<\varepsilon\).
Moreover, recalling \eqref{eq:density_UC_form} we see that \(|\alpha|_\varepsilon=\bigvee_{i=1}^n|v_i|\leq
\delta+\bigvee_{i=1}^n|w_i|=|\tilde\beta|_\varepsilon+\delta\leq|\alpha|_\varepsilon+\delta\), whence it follows that
\(\sfd_{L^0(\XX)}(|\alpha|_\varepsilon,|\beta|_\varepsilon)\leq\sfd_{L^0(\XX)}(\delta\1_\X,0)+\sfd_{L^0(\XX)}
(|\tilde\beta|_\varepsilon,|\beta|_\varepsilon)<\varepsilon\). Therefore, we can apply Lemma
\ref{lem:suff_cond_quotient_oper}, which gives that the map
\({\rm id}\otimes_\varepsilon T\colon{\rm UC}\hat\otimes_\varepsilon\mathscr M\to{\rm UC}\hat\otimes_\varepsilon\mathscr N\)
is a quotient operator.
\end{proof}
\subsection{Relation with duals and pullbacks}
First of all, we provide a characterisation of the dual of \(\mathscr M\hat\otimes_\varepsilon\mathscr N\).
By \(\mathbb D_{\mathscr M^*}^{w^*}\) we will mean the unit disc of \(\mathscr M^*\) endowed with the restriction
of the weak\(^*\) topology. Moreover, the space \(\mathbb D_{\mathscr M^*}^{w^*}\times\mathbb D_{\mathscr N^*}^{w^*}\)
will be tacitly equipped with the product topology.
\begin{theorem}[Dual of \(\mathscr M\hat\otimes_\varepsilon\mathscr N\)]\label{thm:dual_inj_prod}
Let \(\XX\) be a \(\sigma\)-finite measure space. Let \(\mathscr M\) and \(\mathscr N\) be Banach \(L^0(\XX)\)-modules.
Then there exists a unique homomorphism of Banach \(L^0(\XX)\)-modules
\[
\iota\colon\mathscr M\hat\otimes_\varepsilon\mathscr N
\to{\rm C}_{\rm pb}(\mathbb D_{\mathscr M^*}^{w^*}\times\mathbb D_{\mathscr N^*}^{w^*};L^0(\XX))
\]
such that \(\iota(v\otimes w)(\omega,\eta)=\omega(v)\eta(w)\) for every \((v,w,\omega,\eta)\in\mathscr M\times\mathscr N\times\mathbb D_{\mathscr M^*}\times\mathbb D_{\mathscr N^*}\).
Moreover, the homomorphism \(\iota\) preserves the pointwise norm. In particular, it holds that
\[
(\mathscr M\hat\otimes_\varepsilon\mathscr N)^*\cong{\rm C}_{\rm pb}(\mathbb D_{\mathscr M^*}^{w^*}\times
\mathbb D_{\mathscr N^*}^{w^*};L^0(\XX))^*/(\mathscr M\hat\otimes_\varepsilon\mathscr N)^\perp.
\]
\end{theorem}
\begin{proof}
First of all, let us define the operator \(\iota\colon\mathscr M\otimes_\varepsilon\mathscr N\to
{\rm C}_{\rm pb}(\mathbb D_{\mathscr M^*}^{w^*}\times\mathbb D_{\mathscr N^*}^{w^*};L^0(\XX))\) as
\[
\iota(\alpha)(\omega,\eta)\coloneqq\sum_{i=1}^n\omega(v_i)\eta(w_i)\quad\text{ for every }\alpha=\sum_{i=1}^n v_i\otimes w_i
\in\mathscr M\otimes_\varepsilon\mathscr N\text{ and }(\omega,\eta)\in\mathbb D_{\mathscr M^*}\times\mathbb D_{\mathscr N^*}.
\]
It can be easily checked that \(\iota\) is well-posed and \(L^0(\XX)\)-linear. Moreover, \eqref{eq:def_ptwse_norm_UC}
and \eqref{eq:def_inj_ptwse_norm} yield
\[
|\iota(\alpha)|=\bigvee_{(\omega,\eta)\in\mathbb D_{\mathscr M^*}\times\mathbb D_{\mathscr N^*}}\big|\iota(\alpha)(\omega,\eta)\big|=
\bigvee_{(\omega,\eta)\in\mathbb D_{\mathscr M^*}\times\mathbb D_{\mathscr N^*}}\bigg|\sum_{i=1}^n\omega(v_i)\eta(w_i)\bigg|=|\alpha|_\varepsilon
\]
for all \(\alpha=\sum_{i=1}^n v_i\otimes w_i\in\mathscr M\otimes_\varepsilon\mathscr N\), thus \(\iota\) can be
uniquely extended to a pointwise norm preserving homomorphism \(\iota\colon\mathscr M\hat\otimes_\varepsilon\mathscr N
\to{\rm C}_{\rm pb}(\mathbb D_{\mathscr M^*}^{w^*}\times\mathbb D_{\mathscr N^*}^{w^*};L^0(\XX))\).
For the last claim, see Lemma \ref{lem:annihilator}.
\end{proof}

We stress that in Theorem \ref{thm:dual_inj_prod} we consider the space \({\rm C}_{\rm pb}\),
differently from Section \ref{s:rel_with_ord-cont}. It seems that in Theorem \ref{thm:dual_inj_prod} the space
\({\rm C}_{\rm pb}\) cannot be replaced by the smaller space \({\rm UC}_{\rm ord}\). Furthermore, we point out
that the description of \((\mathscr M\hat\otimes_\varepsilon\mathscr N)^*\) provided by Theorem \ref{thm:dual_inj_prod}
is rather implicit if compared with the corresponding one for Banach spaces (see \cite[Proposition 3.14]{Ryan02}).
Indeed, it is not clear whether the space
\({\rm C}_{\rm pb}(\mathbb D_{\mathscr M^*}^{w^*}\times\mathbb D_{\mathscr N^*}^{w^*};L^0(\XX))^*\)
can be described as a space of measures.
\medskip

We conclude the paper by proving that `pullbacks and injective tensor products commute':
\begin{theorem}[Pullbacks vs.\ injective tensor products]\label{thm:pullback_and_inj}
Let \(\XX=(\X,\Sigma_\X,\mm_\X)\), \(\YY=(\Y,\Sigma_\Y,\mm_\Y)\) be separable, \(\sigma\)-finite measure spaces. Let \(\varphi\colon\X\to\Y\) be a measurable map
such that \(\varphi_\#\mm_\X\ll\mm_\Y\). Let \(\mathscr M\) and \(\mathscr N\) be Banach \(L^0(\YY)\)-modules. Then it holds that
\[
\varphi^*(\mathscr M\hat\otimes_\varepsilon\mathscr N)\cong(\varphi^*\mathscr M)\hat\otimes_\varepsilon(\varphi^*\mathscr N),
\]
the pullback map \(\varphi^*\colon\mathscr M\hat\otimes_\varepsilon\mathscr N\to(\varphi^*\mathscr M)\hat\otimes_\varepsilon(\varphi^*\mathscr N)\)
being the unique homomorphism such that
\[
\varphi^*(v\otimes w)=(\varphi^*v)\otimes(\varphi^*w)\quad\text{ for every }v\in\mathscr M\text{ and }w\in\mathscr N.
\]
\end{theorem}
\begin{proof}
Define \(T\colon\mathscr M\otimes_\varepsilon\mathscr N\to(\varphi^*\mathscr M)\otimes_\varepsilon(\varphi^*\mathscr N)\)
as \(T\big(\sum_{i=1}^n v_i\otimes w_i\big)\coloneqq\sum_{i=1}^n(\varphi^*v_i)\otimes(\varphi^*w_i)\) for every
\(\sum_{i=1}^n v_i\otimes w_i\in\mathscr M\otimes_\varepsilon\mathscr N\). The well-posedness of \(T\) can be proved
exactly as in Theorem \ref{thm:pullback_and_proj}, while its linearity is clear by construction.
Moreover, for any \(\alpha=\sum_{i=1}^n v_i\otimes w_i\in\mathscr M\otimes_\varepsilon\mathscr N\),
\[\begin{split}
|\alpha|_\varepsilon\circ\varphi&=\bigvee\bigg\{\Big|\sum_{i=1}^n{\sf I}_\varphi(\varphi^*\omega)(\varphi^*v_i)\,{\sf I}_\varphi(\varphi^*\eta)(\varphi^*w_i)\Big|\;\bigg|\;\omega\in\mathbb D_{\mathscr M^*},\,\eta\in\mathbb D_{\mathscr N^*}\bigg\}\\
&\leq\bigvee\bigg\{\Big|\sum_{i=1}^n\Xi(\varphi^*v_i)\Theta(\varphi^*w_i)\Big|\;\bigg|\;
\Xi\in\mathbb D_{(\varphi^*\mathscr M)^*},\,\Theta\in\mathbb D_{(\varphi^*\mathscr N)^*}\bigg\}=|T(\alpha)|_\varepsilon.
\end{split}\]
Conversely, if \(\xi=\sum_{j=1}^m\1_{E_j}\cdot\varphi^*\omega_j\in\mathscr G(\varphi^*[\mathbb D_{\mathscr M^*}])\)
and \(\theta=\sum_{k=1}^\ell\1_{F_k}\cdot\varphi^*\eta_k\in\mathscr G(\varphi^*[\mathbb D_{\mathscr N^*}])\), then
\[\begin{split}
\bigg|\sum_{i=1}^n{\sf I}_\varphi(\xi)(\varphi^*v_i)\,{\sf I}_\varphi(\theta)
(\varphi^*w_i)\bigg|&=\sum_{j=1}^m\sum_{k=1}^\ell\1_{E_j\cap F_k}\bigg|\sum_{i=1}^n\omega_j(v_i)\eta_k(w_i)\bigg|\circ\varphi
\leq|\alpha|_\varepsilon\circ\varphi.
\end{split}\]
Using Theorem \ref{thm:seq_weak-star_density}, as well as the density of \(\mathscr G(\varphi^*[\mathbb D_{\mathscr M^*}])\)
and \(\mathscr G(\varphi^*[\mathbb D_{\mathscr N^*}])\) in \(\mathbb D_{\varphi^*\mathscr M^*}\) and \(\mathbb D_{\varphi^*\mathscr N^*}\),
respectively, we deduce that \(\big|\sum_{i=1}^n\Xi(\varphi^*v_i)\Theta(\varphi^*w_i)\big|\leq|\alpha|_\varepsilon\circ\varphi\)
for all \(\Xi\in\mathbb D_{\varphi^*\mathscr M^*}\) and \(\Theta\in\mathbb D_{\varphi^*\mathscr N^*}\). It follows
that \(|T(\alpha)|_\varepsilon\leq|\alpha|_\varepsilon\circ\varphi\) for every \(\alpha\in\mathscr M\otimes_\varepsilon\mathscr N\),
thus accordingly \(T\) can be uniquely extended to a linear operator
\(\varphi^*\colon\mathscr M\hat\otimes_\varepsilon\mathscr N\to(\varphi^*\mathscr M)\hat\otimes_\varepsilon(\varphi^*\mathscr N)\)
satisfying \(|\varphi^*\alpha|_\varepsilon=|\alpha|_\varepsilon\circ\varphi\) for every \(\alpha\in\mathscr M\hat\otimes_\varepsilon\mathscr N\).
Finally, the fact that \(\varphi^*[\mathscr M\hat\otimes_\varepsilon\mathscr N]\) generates
\((\varphi^*\mathscr M)\hat\otimes_\varepsilon(\varphi^*\mathscr N)\) can be proved as in Theorem \ref{thm:pullback_and_proj},
using Lemma \ref{lem:generators_inj_tensor} in place of Lemma \ref{lem:generators_proj_tensor}. The proof is then complete.
\end{proof}
\def\cprime{$'$} \def\cprime{$'$}

\end{document}